\documentclass[11pt,letterpaper]{amsart}
\usepackage{amsmath,amssymb,latexsym,cancel,rotating,amsthm,mathrsfs,fancyhdr,graphicx}

\newtheorem{theorem}{Theorem}[section]
\newtheorem{lemma}[theorem]{Lemma}

\newtheorem{corollary}[theorem]{Corollary}
\newtheorem{definition}[theorem]{Definition}
\newtheorem{proposition}[theorem]{Proposition}

\def\bnu{\boldsymbol{\nu}}

\begin{document}

\title[Canonical bases of tensor products arising from framed constructions]{Canonical bases of tensor products of integrable highest weight modules arising from framed constructions}

\author[Jiepeng Fang, Yixin Lan]{Jiepeng Fang, Yixin Lan}

\address{Department of Mathematics and New Cornerstone Science Laboratory, The University of Hong Kong, Pokfulam, Hong Kong, Hong Kong SAR, P. R. China}
\email{fangjp@hku.hk (J.Fang)}

\address{Academy of Mathematics and Systems Science, Chinese Academy of Sciences, Beijing 100190, P.R.China}
\email{lanyixin@amss.ac.cn (Y.Lan)}

\subjclass[2020]{17B37}

\keywords{Canonical basis, Tensor product, Quantum group, Positivity}

\bibliographystyle{abbrv}

\begin{abstract}
Given a quantum group, we prove that the canonical bases of the tensor products of its integrable highest weight modules can be obtained from the canonical bases of the integrable highest weight modules of a bigger quantum group. As a result, based on the positivity of the canonical bases of the integrable highest weight modules due to Lusztig, we prove that the canonical bases of the tensor products have the positivity.
\end{abstract}

\maketitle

\setcounter{tocdepth}{1}\tableofcontents

\section{Introduction}
For any symmetric Cartan datum $(I,\cdot)$ and $Y$-regular root datum $(Y,X,\langle,\rangle,...)$ of type $(I,\cdot)$, let $\mathbf{f}$ be the corresponding algebra and $\mathbf{U}$ be the corresponding quantized enveloping algebra defined in \cite[Chapter 1, 3]{Lusztig-1993}. For any dominant weight $\lambda$, let $\Lambda_{\lambda}$ be the integrable highest weight module of $\mathbf{U}$ with the highest weight $\lambda$, see section \ref{Preliminaries} for details.

In \cite{Lusztig-1990,Lusztig-1991}, Lusztig categorified the algebra $\mathbf{f}$ by perverse sheaves on the varieties of quiver representations, and constructed the canonical basis $\mathbf{B}$ of $\mathbf{f}$ with some very remarkable properties, such as the integrality, the positivity for the structure constants with respect to the multiplication and the comultiplication of $\mathbf{f}$, and the compatibility with various natural filtrations. In particular, the canonical basis $\mathbf{B}$ naturally gives rise to the canonical basis $\mathbf{B}(\Lambda_{\lambda})$ of $\Lambda_{\lambda}$, see section \ref{Preliminaries} for details. Lusztig proved that the canonical basis $\mathbf{B}(\Lambda_{\lambda})$ has the positivity for the structure constants with respect to the actions of Chevalley generators $F_i,E_i\in \mathbf{U}$, see \cite[Theorem 22.1.7]{Lusztig-1993}. We remark that this proof relies on the positivity of $\mathbf{B}$. In \cite{Kashiwara-1990,Kashiwara-1991}, Kashiwara constructed the canonical bases $\mathbf{B}$ of $\mathbf{f}$ and $\mathbf{B}(\Lambda_{\lambda})$ of $\Lambda_{\lambda}$ by his crystal basis theory. However, this elegant construction does not provide the positivity. The proof for the positivity of the canonical basis $\mathbf{B}(\Lambda_{\lambda})$ given by the categorification of the module $\Lambda_{\lambda}$ is expected. In \cite{Kang-Kashiwara-2012}, Kang-Kashiwara categorified $\Lambda_{\lambda}$ by finitely generated graded projective modules of cyclotomic Khovanov-Lauda-Rouquier algebra. In \cite{Zheng-2014}, Zheng categorified $\Lambda_{\lambda}$ by perverse sheaves on the stacks of representations of the framed quiver, and proved the positivity of $\mathbf{B}(\Lambda_{\lambda})$. In \cite{Fang-Lan-Xiao-2023}, based on Zheng's work, Fang-Lan-Xiao gave a categorification of $\Lambda_{\lambda}$ by Lusztig sheaves for the framed quiver, and proved the positivity of $\mathbf{B}(\Lambda_{\lambda})$. We remark that the framed quiver construction is inspired by Nakajima's work \cite{Nakajima-1998}.

In \cite{Lusztig-1992}, by using the quasi-$\mathcal{R}$-matrix, Lusztig constructed the canonical basis of the tensor product of a simple integrable lowest weight module followed by an integrable highest weight module. In \cite[Chapter 27]{Lusztig-1993}, he generalized this method to construct the canonical basis of the tensor product of based modules when $(I,\cdot)$ is of finite type. In \cite{Bao-Wang-2016}, Bao-Wang generalized Lusztig's method to construct the canonical basis of the tensor product of several integrable lowest weight modules followed by integrable highest weight modules for any $(I,\cdot)$.

It is natural to consider the question whether the canonical basis of the tensor product has the positivity for the structure constants with respect to the actions of Chevalley generators $F_i,E_i\in \mathbf{U}$. There are positive answers for the tensor product of several highest weight modules when $(I,\cdot)$ is symmetric. In \cite{Webster-2015}, Webster categorified the tensor product building on works of \cite{Khovanov-Lauda-2009,Rouquier-2012,Varagnolo-Vasserot-2011}, and proved the positivity of the canonical basis. In \cite{Zheng-2014}, Zheng categorified the tensor product by perverse sheaves on the stacks of representations of the framed quiver. In \cite{Fang-Lan-2023}, Fang-Lan categorified the tensor product by Lusztig sheaves for the $N$-framed quiver, and proved the positivity of the canonical basis.

When Fang discussed with Professor Huanchen Bao in Singapore, Bao pointed out that the construction in \cite{Fang-Lan-2023} could be revised to give an algebraic proof for the positivity of canonical basis of the tensor product of the integrable highest weight modules. The present paper achieves this goal, and our strategy is as follows.\\ 
\textbf{Step 1}. For the Cartan datum $(I,\cdot)$ and the root datum $(Y,X,\langle,\rangle,...)$, we enlarge them to be the framed Cartan datum $(\tilde{I},\cdot)$ and the framed root datum $(\tilde{Y},\tilde{X},\langle,\rangle,...)$, see section \ref{Framed Cartan datum and root datum} for details. Then we have the corresponding algebra $\tilde{\mathbf{f}}$ and the corresponding quantized enveloping algebra $\tilde{\mathbf{U}}$ such that $\mathbf{f}$ is a subalgebra of $\tilde{\mathbf{f}}$, and $\mathbf{U}$ is a subalgebra of $\tilde{\mathbf{U}}$.\\
\textbf{Step 2}. For any dominant weights $\xi,\lambda$ with respect to $(Y,X,\langle,\rangle,...)$, we define a dominant weight $\xi\odot\lambda$ with respect to $(\tilde{Y},\tilde{X},\langle,\rangle,...)$, see section \ref{framed construction of tensor product} for details. Then we have the integrable highest weight module $\tilde{\Lambda}_{\xi\odot\lambda}$ of $\tilde{\mathbf{U}}$ with the canonical basis $\tilde{\mathbf{B}}(\tilde{\Lambda}_{\xi\odot\lambda})$.\\
\textbf{Step 3}. We restrict the $\tilde{\mathbf{U}}$-module $\tilde{\Lambda}_{\xi\odot\lambda}$ to be a $\mathbf{U}$-module, and construct a $\mathbf{U}$-submodule $\Lambda_{\xi,\lambda}$ of it, see Definition \ref{Lambdaxilambda}. Moreover, we construct a subset $\tilde{\mathbf{B}}(\Lambda_{\xi,\lambda})\subset \tilde{\mathbf{B}}(\tilde{\Lambda}_{\xi\odot\lambda})$ such that $\tilde{\mathbf{B}}(\Lambda_{\xi,\lambda})$ is a basis of $\Lambda_{\xi,\lambda}$, see Definition \ref{basis definition} and Corollary \ref{framed basis}.\\
\textbf{Step 4}. By using the comultiplication of $\tilde{\mathbf{f}}$, we construct a $\mathbf{U}$-module isomorphism $\varphi:\Lambda_{\xi,\lambda}\rightarrow \Lambda_{\xi}\otimes \Lambda_{\lambda}$, see Theorem \ref{framed construction of tensor}. Moreover, we prove that the image of the basis $\tilde{\mathbf{B}}(\Lambda_{\xi,\lambda})$ under $\varphi$ coincides with the canonical basis of the tensor product $\Lambda_{\xi}\otimes \Lambda_{\lambda}$, see Theorem \ref{framed construction of canonical basis}.\\
\textbf{Step 5}. Based on the positivity of the canonical basis $\tilde{\mathbf{B}}(\tilde{\Lambda}_{\xi\odot\lambda})$, we complete the proof of the positivity of the canonical basis of the tensor product $\Lambda_{\xi}\otimes \Lambda_{\lambda}$, see Theorem \ref{positive}.

The main results of this paper have been introduced in \textbf{Step 1}-\textbf{Step 5}. We add some comments as follows. Our notations follow that of \cite{Lusztig-1993}, and most definitions we recall in section \ref{Preliminaries} can be found in it. The key observation of the present paper is that the canonical basis of the tensor product $\Lambda_{\xi}\otimes \Lambda_{\lambda}$ can be obtained from the canonical basis of a bigger algebra $\tilde{\mathbf{f}}$ via the comultiplication in \textbf{Step 4}. We give an explicit example for Theorem \ref{framed construction of canonical basis}, see section \ref{Example}. Our construction can be generalized to solve the case for the tensor product of several integrable highest weight modules by further enlarging the Cartan datum, see section \ref{generalized case}.

We remark that the $\mathbf{U}$-module $\Lambda_{\xi,\lambda}$ has been constructed by Li in \cite[Section 3.2, 3.3]{Li-2014}. Moreover, he constructed a $\mathbf{U}$-module isomorphism $\Lambda_{\xi}\otimes \Lambda_{\lambda}\rightarrow \Lambda_{\xi,\lambda}$. This isomorphism is rather different from the isomorphism $\varphi:\Lambda_{\xi,\lambda}\rightarrow \Lambda_{\xi}\otimes \Lambda_{\lambda}$ in \textbf{Step 4}. By using the construction of $\varphi$, we can prove the positivity of the transition matrix between the bases $\mathbf{B}(\Lambda_{\xi}\otimes \Lambda_{\lambda})$ and $\mathbf{B}(\Lambda_{\xi})\otimes \mathbf{B}(\Lambda_{\lambda})$, see Corollary \ref{positivity of transition matrix}. This positivity result and the positivity of the canonical basis $\mathbf{B}(\Lambda_{\xi}\otimes \Lambda_{\lambda})$ are not contained in \cite{Li-2014}.

We remark that Bao-He gave the thickening construction to study the twisted products of flag varieties in \cite{Bao-He-2022}. The thickening construction provided another method to enlarge the Cartan datum, and also inspired us to enlarge the Cartan datum to study the canonical bases of the tensor products.

\subsection*{Acknowledgments}
{This work was supported by the New Cornerstone Science Foundation through the New Cornerstone Investigator Program awarded to Professor Xuhua He; and the National Natural Science Foundation of China [Grant numbers 1288201,12471030]. The authors would like to thank Professor Jie Xiao for his helpful discussions and suggestions. The authors would like to thank Professor Huanchen Bao for pointing out that our idea in \cite{Fang-Lan-2023} could be combined with \cite[Theorem 22.1.7]{Lusztig-1993} due to Lusztig to give an algebraic proof for the positivity which is independent of the geometric methods in \cite{Fang-Lan-2023}. The authors are grateful to the referees for helpful suggestions and corrections.}

\section{Preliminaries}\label{Preliminaries}

Throughout this paper, we fix a symmetric Cartan datum $(I,\cdot)$ and a $Y$-regular root datum $(Y,X,\langle,\rangle,...)$ of type $(I,\cdot)$. More precisely, we fix a finite set $I$; a symmetric bilinear form $\mathbb{Z}[I]\times \mathbb{Z}[I]\rightarrow \mathbb{Z}, (\nu,\nu')\mapsto \nu\cdot \nu'$; two finitely generated free abelian groups $Y$ and $X$; a perfect bilinear pairing $\langle,\rangle:Y\times X\rightarrow \mathbb{Z}$; two embeddings $I\hookrightarrow Y, i\mapsto i$ and $I\hookrightarrow X, i\mapsto i^*$, such that\\
(a) $i\cdot i=2$ for any $i\in I$;\\
(b) $i\cdot j\leqslant 0$ for any $i\not=j$ in $I$;\\
(c) $\langle i,j^*\rangle=i\cdot j$ for any $i,j\in I$;\\
(d) the image of the imbedding $I\subset Y$ is linearly independent in $Y$.\\
The embeddings $I\hookrightarrow Y$ and $I\hookrightarrow X$ induce homomorphisms $\mathbb{Z}[I]\rightarrow Y$ and $\mathbb{Z}[I]\rightarrow X$, and we shall often denote the images of $\nu\in \mathbb{Z}[I]$ under either of these homomorphisms by the same notation $\nu$. 

For any $\nu=\sum_{i\in I}\nu_ii\in \mathbb{N}[I]$, we set $\mathrm{tr}\,\nu=\sum_{i\in I}\nu_i\in \mathbb{N}$.

Let $\mathbb{Q}(v)$ be the field of rational functions in $v$ with coefficients in $\mathbb{Q}$, $\mathcal{A}=\mathbb{Z}[v,v^{-1}]$ and $\mathbf{A}=\mathbb{Q}[[v^{-1}]]\cap \mathbb{Q}(v)$. For any $n\in \mathbb{Z}, m\in \mathbb{N}$, we denote  
$$[n]=\frac{v^n-v^{-n}}{v-v^{-1}}, [m]!=\prod^m_{k=1}[k], \begin{bmatrix}n\\m\end{bmatrix}=\frac{\prod_{k=0}^{m-1}(v^{n-k}-v^{-n+k})}{\prod_{k=1}^m(v^k-v^{-k})}.$$

\subsection{The algebra $\mathbf{f}$}\label{The algebra f}
For the Cartan datum $(I,\cdot)$, the algebra $\mathbf{f}$ is the $\mathbb{Q}(v)$-algebra with $1$ generated by $\{\theta_i\mid i\in I\}$ subject to the quantum Serre relations
$$\sum_{n=0}^{1-i\cdot j}(-1)^n\theta_i^{(n)}\theta_j\theta_i^{(1-i\cdot j-n)}=0\ \textrm{for any $i\not=j$ in $I$},$$
where $\theta_i^{(n)}=\theta^n_i/[n]!$ for any $i\in I$ and $n\in \mathbb{N}$. The integral form ${_{\mathcal{A}}\mathbf{f}}$ is the $\mathcal{A}$-subalgebra of $\mathbf{f}$ generated by $\{\theta_i^{(n)}\mid i\in I,n\in \mathbb{N}\}$.

For any $\nu\in \mathbb{N}[I]$, we define $\mathcal{V}_\nu$ to be the set of sequences $(a_1i_1,...a_ni_n)$, where $a_k\in \mathbb{N}$, $i_k\in I,k=1,...,n$ and $n\in \mathbb{N}$ such that $\sum_{k=1}^n a_ki_k=\nu$. For any $\bnu=(a_1i_1,...a_ni_n)\in \mathcal{V}_\nu$, we denote $\theta_{\bnu}=\theta^{(a_1)}_{i_1}...\theta^{(a_n)}_{i_n}$. Then the algebra $\mathbf{f}=\bigoplus_{\nu\in \mathbb{N}[I]}\mathbf{f}_\nu$
is $\mathbb{N}[I]$-graded, where the $\nu$-homogeneous component $\mathbf{f}_\nu$ is the subspace of $\mathbf{f}$ spanned by $\{\theta_{\bnu}\mid\bnu\in \mathcal{V}_\nu\}$. For any homogeneous $x\in \mathbf{f}_\nu$, we set $|x|=\nu$. 

The tensor products $\mathbf{f}\otimes \mathbf{f}$ is an algebra with the multiplication 
$$(x_1\otimes x_2)(y_1\otimes y_2)=v^{|x_2|\cdot |y_1|}x_1y_1\otimes x_2y_2,$$
where $x_1,x_2,y_1,y_2\in \mathbf{f}$ are homogeneous. The comultiplication of $\mathbf{f}$ is the unique algebra homomorphism $r:\mathbf{f}\rightarrow \mathbf{f}\otimes \mathbf{f}$ such that 
$$r(\theta_i)=\theta_i\otimes 1+1\otimes \theta_i\ \textrm{for any $i\in I$}.$$

The bar-involution of $\mathbf{f}$ is the unique $\mathbb{Q}$-algebra involution $\bar{}:\mathbf{f}\rightarrow \mathbf{f}$ such that 
$$\overline{\theta_i}=\theta_i, \overline{v^n}=v^{-n}\ \textrm{for any $i\in I$ and $n\in \mathbb{Z}$}.$$

\begin{proposition}[{\cite[Proposition 1.2.3]{Lusztig-1993}}]
There is a unique non-degenerate symmetric bilinear form $(,):\mathbf{f}\times \mathbf{f}\rightarrow \mathbb{Q}(v)$ such that\\
{\rm{(a)}} $(\theta_i,\theta_j)=\delta_{i,j}(1-v^{-2})^{-1}$ for any $i,j\in I$;\\
{\rm{(b)}} $(x_1x_2,x)=(x_1\otimes x_2,r(x))$ for any $x_1,x_2,x\in \mathbf{f}$, where the bilinear form on $\mathbf{f}\otimes \mathbf{f}$ is given by $(x_1\otimes x_2,x'_1\otimes x'_2)=(x_1,x'_1)(x_2,x'_2)$ for any $x_1,x_2,x'_1,x'_2\in \mathbf{f}$.
\end{proposition}

\subsection{The algebra $\mathbf{U}$}\label{The algebra U}
For the root datum $(Y,X,\langle,\rangle,...)$, the quantized enveloping algebra $\mathbf{U}$ is the $\mathbb{Q}(v)$-algebra with $1$ generated by $\{E_i,F_i,K_\mu\mid i\in I,\mu\in Y\}$ subject to the following relations
\begin{align*}
&K_0=1,\ K_{\mu}K_{\mu'}=K_{\mu+\mu'}\ \textrm{for any $\mu,\mu'\in Y$};\\
&K_{\mu}E_i=v^{\langle \mu,i^*\rangle}E_iK_{\mu}\ \textrm{for any $i\in I$ and $\mu\in Y$};\\
&K_{\mu}F_i=v^{-\langle \mu,i^*\rangle}F_iK_{\mu}\ \textrm{for any $i\in I$ and $\mu\in Y$};\\
&E_iF_j-F_jE_i=\delta_{i,j}\frac{K_i-K_{-i}}{v-v^{-1}}\ \textrm{for any $i,j\in I$};\\
&\sum_{n=0}^{1-i\cdot j}(-1)^nE_i^{(n)}E_jE_i^{(1-i\cdot j-n)}=0\ \textrm{for any $i\not=j$ in $I$};\\
&\sum_{n=0}^{1-i\cdot j}(-1)^nF_i^{(n)}F_jF_i^{(1-i\cdot j-n)}=0\ \textrm{for any $i\not=j$ in $I$},
\end{align*}
where $E_i^{(n)}=E^n_i/[n]!, F_i^{(n)}=F^n_i/[n]!$ for any $i\in I$ and $n\in \mathbb{N}$.

Let $\mathbf{U}^+$ and $\mathbf{U}^-$ be the subalgebras of $\mathbf{U}$ generated by $\{E_i\mid i\in I\}$ and $\{F_i\mid i\in I\}$ respectively. There are algebra isomorphisms $\mathbf{f}\rightarrow \mathbf{U}^+, x\mapsto x^+$ and $\mathbf{f}\rightarrow \mathbf{U}^-, x\mapsto x^-$
such that $E_i=\theta_i^+$ and $F_i=\theta_i^-$ for any $i\in I$.

The tensor product $\mathbf{U}\otimes \mathbf{U}$ is an algebra in the standard way. The comultiplication of $\mathbf{U}$ is the unique algebra homomorphism $\Delta:\mathbf{U}\rightarrow \mathbf{U}\otimes \mathbf{U}$ such that 
\begin{align*}
&\Delta(E_i)=E_i\otimes 1+K_i\otimes E_i\ \textrm{for any $i\in I$};\\
&\Delta(F_i)=1\otimes F_i+F_i\otimes K_{-i}\ \textrm{for any $i\in I$};\\
&\Delta(K_\mu)=K_\mu\otimes K_\mu \ \textrm{for any $\mu\in Y$}.
\end{align*}

The bar-involution of $\mathbf{U}$ is the unique $\mathbb{Q}$-algebra involution $\bar{}:\mathbf{U}\rightarrow \mathbf{U}$ such that 
$$\overline{E_i}=E_i, \overline{F_i}=F_i, \overline{K_{\mu}}=K_{-\mu}, \overline{v^n}=v^{-n}\ \textrm{for any $i\in I, \mu\in Y$ and $n\in \mathbb{Z}$}.$$

\subsection{The module $\Lambda_\lambda$}\label{The module Lambda}
A $\mathbf{U}$-module $M$ is a weight module, if it has a decomposition $M=\bigoplus_{\lambda\in X}M^{\lambda}$ such that $K_{\mu}m=v^{\langle\mu,\lambda\rangle}m$ for any $\mu\in Y,\lambda\in X$ and $m\in M^\lambda$. If $M^\lambda\not=0$, then $M^\lambda$ is called a weight subspace of $M$. Throughout this paper, we only consider weight modules.

\begin{lemma}[{\cite[Section 3.4.2]{Lusztig-1993}}]\label{3.4.2}
For any $i,j\in I,a,b\in \mathbb{N},\lambda\in X$ and $m\in M^\lambda$, we have
\begin{align*}
E_i^{(a)}F_j^{(b)}m=F_j^{(b)}E_i^{(a)}m+\delta_{i,j}\sum_{k=1}^{\mathrm{min}(a,b)}\begin{bmatrix}
a-b+\langle i,\lambda\rangle\\k
\end{bmatrix}F_i^{(b-k)}E_i^{(a-k)}m.
\end{align*}
\end{lemma}

Let $\rho:\mathbf{U}\rightarrow \mathbf{U}^{\mathrm{opp}}$ be the algebra isomorphism such that $$\rho(E_i)=vK_iF_i, \rho(F_i)=vK_{-i}E_i, \rho(K_{\mu})=K_{\mu}\ \textrm{for any $i\in I$ and $\mu\in Y$.}$$
A symmetric bilinear form $(,):M\times M\rightarrow \mathbb{Q}(v)$ is called admissible, if 
\begin{align*}
&(M^{\lambda},M^{\lambda'})=0\ \textrm{for any $\lambda\not=\lambda'$ in $X$};\\
&(um,m')=(m,\rho(u)m')\ \textrm{for any $u\in \mathbf{U}$ and $m,m'\in M$}.
\end{align*}

For any $\lambda\in X$, the Verma module $M_\lambda$ is $\mathbf{f}$ as a $\mathbb{Q}(v)$-vector space with the $\mathbf{U}$-module structure such that $E_i1=0, F_ix=\theta_ix, K_{\mu}x=v^{\langle\mu,\lambda-|x|\rangle}x$ for any $i\in I,\mu\in Y$ and homogeneous $x\in \mathbf{f}$. Note that $M_{\lambda}^{\lambda'}=\bigoplus_{\nu\in \mathbb{N}[I], \lambda'=\lambda-\nu}\mathbf{f}_{\nu}$ for any $\lambda'\in X$.

For any dominant weight $\lambda\in X$, that is, $\langle i,\lambda\rangle \geqslant 0$ for any $i\in I$, let $\sum_{i\in I}\mathbf{f}\theta_i^{\langle i,\lambda\rangle+1}$ be the left ideal of $\mathbf{f}$ generated by $\{\theta_i^{\langle i,\lambda\rangle+1}\mid i\in I\}$. This left ideal is a submodule of $M_{\lambda}$, see \cite[Proposition 3.5.6(a)]{Lusztig-1993}. The integrable highest weight module $\Lambda_\lambda$ is the quotient 
$$\Lambda_\lambda=M_{\lambda}/\sum_{i\in I}\mathbf{f}\theta_i^{\langle i,\lambda\rangle+1},$$
where the highest weight vector $\eta_\lambda\in \Lambda_\lambda$ is the image of $1\in \mathbf{f}$. The integral form ${_{\mathcal{A}}\Lambda_{\lambda}}$ is the image of ${_{\mathcal{A}}\mathbf{f}}$ under the natural projection $\mathbf{f}\rightarrow \Lambda_{\lambda}$.

The bar-involution of $\Lambda_{\lambda}$ is the unique $\mathbb{Q}$-linear involution $\bar{}:\Lambda_{\lambda}\rightarrow \Lambda_{\lambda}$ such that 
$$\overline{u\eta_{\lambda}}=\overline{u}\eta_{\lambda}\ \textrm{for any $u\in \mathbf{U}$}.$$
By definitions, we have $\overline{um}=\bar{u}\bar{m}$ for any $u\in \mathbf{U}$ and $m\in\Lambda_{\lambda}$.

\begin{proposition}[{\cite[Proposition 19.1.2]{Lusztig-1993}}]\label{19.1.2}
There is a unique admissible symmetric bilinear form $(,)_{\lambda}:\Lambda_{\lambda}\times \Lambda_{\lambda}\rightarrow \mathbb{Q}(v)$ such that $(\eta_{\lambda},\eta_{\lambda})_{\lambda}=1$.
\end{proposition}

\subsection{The tensor product $\Lambda_{\xi}\otimes \Lambda_{\lambda}$}
For any two dominant weights $\xi,\lambda\in X$, the tensor product $\Lambda_{\xi}\otimes \Lambda_{\lambda}$ is naturally a $(\mathbf{U}\otimes \mathbf{U})$-module such that
$$(u_1\otimes u_2)(m_1\otimes m_2)=u_1m_1\otimes u_2m_2\ \textrm{for any $u_1,u_2\in \mathbf{U}$ and $m_1\in \Lambda_{\xi}, m_2\in \Lambda_{\lambda}$}.$$
Then $\Lambda_{\xi}\otimes \Lambda_{\lambda}$ is a $\mathbf{U}$-module via the comultiplication $\Delta:\mathbf{U}\rightarrow \mathbf{U}\otimes \mathbf{U}$.

The bar-involution of $\Lambda_{\xi}\otimes \Lambda_{\lambda}$ is the $\mathbb{Q}$-linear involution $\bar{}:\Lambda_{\xi}\otimes \Lambda_{\lambda}\rightarrow \Lambda_{\xi}\otimes \Lambda_{\lambda}$ such that 
$$\overline{m_1\otimes m_2}=\overline{m_1}\otimes \overline{m_2}\ \textrm{for any $m_1\in \Lambda_{\xi}$ and $m_2\in \Lambda_{\lambda}$}.$$
Note that $\overline{um}=\bar{u}\bar{m}$ does not hold for any $u\in \mathbf{U}$ and $m\in\Lambda_{\xi}\otimes \Lambda_{\lambda}$ in general, since the comultiplication and the bar-involution of $\mathbf{U}$ do not commute with the each other. We will define a new involution $\Psi:\Lambda_{\xi}\otimes \Lambda_{\lambda}\rightarrow \Lambda_{\xi}\otimes \Lambda_{\lambda}$ in \ref{Canonical basis of tensor product}.

We define a symmetric bilinear form $(,)_{\xi,\lambda}:(\Lambda_{\xi}\otimes \Lambda_{\lambda})\times (\Lambda_{\xi}\otimes \Lambda_{\lambda})\rightarrow \mathbb{Q}(v)$ by
$$(m_1\otimes m_2,m'_1\otimes m'_2)_{\xi,\lambda}=(m_1,m'_1)_{\xi}(m_2,m'_2)_{\lambda}\ \textrm{for any $m_1,m'_1\in \Lambda_{\xi}$ and $m_2,m,'_2\in \Lambda_{\lambda}$}.$$
By \cite[lemma 17.1.3(b)]{Lusztig-1993}, the bilinear form $(,)_{\xi,\lambda}$ is admissible.

\subsection{Canonical basis}\label{Canonical basis}

A basis $B$ of a $\mathbb{Q}(v)$-vector space $V$ is called almost orthonormal with respect to a symmetric bilinear form $(,):V\times V\rightarrow \mathbb{Q}(v)$, if 
$$(b,b')\in\delta_{b,b'}+v^{-1}\mathbb{Z}[[v^{-1}]]\cap \mathbb{Q}(v)\ \textrm{for any $b,b'\in B$}.$$

\begin{lemma}[{\cite[Lemma 14.2.2]{Lusztig-1993}}]\label{14.2.2}
Let $B$ be an almost orthonormal basis of a $\mathbb{Q}(v)$-vector space $V$ with respect to a symmetric bilinear form $(,):V\times V\rightarrow \mathbb{Q}(v)$, $_{\mathcal{A}}V$ be the $\mathcal{A}$-submodule of $V$ generated by $B$ and $L(V)=\{x\in V\mid (x,x)\in \mathbf{A}\}$.\\
{\rm{(a)}} The set $L(V)$ is a $\mathbf{A}$-submodule of $V$ with a basis $B$.\\
{\rm{(b)}} For any $x\in {_{\mathcal{A}}V}$ satisfying $(x,x)\in 1+v^{-1}\mathbf{A}$, there exists $b\in B$ such that $$x\in \pm b\ \mathrm{mod}\, v^{-1}L(V).$$
\end{lemma}

\subsubsection{Canonical basis of $\mathbf{f}$}\label{Canonical basis of f}

We refer to \cite[Chapter 14]{Lusztig-1993} and \cite[Sections 10]{Lusztig-1991} for the definition of the canonical basis $\mathbf{B}$ of $\mathbf{f}$. We summarize some properties of $\mathbf{B}$ as follows.

\begin{theorem}[{\cite[Theorem 14.2.3, 14.4.3, 14.4.13]{Lusztig-1993}}]\label{properties of B}
The canonical basis $\mathbf{B}$ of $\mathbf{f}$ has the following properties:\\
{\rm{(integral)}} $\mathbf{B}$ is a $\mathcal{A}$-basis of the integral form ${_{\mathcal{A}}\mathbf{f}}$;\\
{\rm{(bar-invariant)}} for any $b\in \mathbf{B}$, we have $\overline{b}=b$;\\
{\rm{(almost orthonormal)}} for any $b,b'\in \mathbf{B}$, we have $(b,b')\in \delta_{b,b'}+v^{-1}\mathbb{Z}[[v^{-1}]]$;\\
{\rm{(positive)}} for any $b_1,b_2,b\in \mathbf{B}$, we have 
$$b_1b_2\in \sum_{b_3\in \mathbf{B}}\mathbb{N}[v,v^{-1}]b_3,\ r(b)\in \sum_{b',b''\in \mathbf{B}}\mathbb{N}[v,v^{-1}]b'\otimes b''.$$
\end{theorem}

Let $\sigma:\mathbf{f}\rightarrow \mathbf{f}^{\mathrm{opp}}$ be the unique algebra homomorphism such that $\sigma(\theta_i)=\theta_i$ for any $i\in I$, see \cite[Section 1.2.9]{Lusztig-1993}. For any $i\in I$ and $n\in \mathbb{N}$, we define the subsets
\begin{align*}
&\mathbf{B}_{i,\geqslant n}=\mathbf{B}\cap \theta_i^n\mathbf{f},\ \mathbf{B}_{i,n}=\mathbf{B}_{i,\geqslant n}\setminus \mathbf{B}_{i,\geqslant n+1},\\
&\mathbf{B}_{i,\geqslant n}^{\sigma}=\mathbf{B}\cap \mathbf{f}\theta_i^n,\ \mathbf{B}_{i,n}^{\sigma}=\mathbf{B}_{i,\geqslant n}^{\sigma}\setminus \mathbf{B}_{i,\geqslant n+1}^{\sigma}.
\end{align*}
Then we have $\mathbf{B}_{i,\geqslant n}^{\sigma}=\sigma(\mathbf{B}_{i,\geqslant n}),\mathbf{B}_{i,n}^{\sigma}=\sigma(\mathbf{B}_{i,n})$ and 
$$\mathbf{B}=\bigcup_{n\in \mathbb{N}}\mathbf{B}_{i,\geqslant n}=\bigsqcup_{n\in \mathbb{N}}\mathbf{B}_{i,n}=\bigcup_{n\in \mathbb{N}}\mathbf{B}_{i,\geqslant n}^{\sigma}=\bigsqcup_{n\in \mathbb{N}}\mathbf{B}_{i,n}^{\sigma},$$
see \cite[Section 14.3.1]{Lusztig-1993}.

\begin{theorem}[{\cite[Theorem 14.3.2]{Lusztig-1993} or \cite[Lemma 6.4, Theorem 11.7]{Lusztig-1991}}]\label{14.3.2}
For any $i\in I$ and $n\in \mathbb{N}$, let $p_{ni}:\mathbf{f}\rightarrow \mathbf{f}_{ni}$ be the natural projection from $\mathbf{f}$ to its homogeneous component $\mathbf{f}_{ni}$. Then we have\\
{\rm{(a)}} the subset $\mathbf{B}_{i,\geqslant n}\subset \mathbf{B}$ is a basis of the subspace $\theta_i^n\mathbf{f}\subset \mathbf{f}$; \\
{\rm{(b)}} the subset $\mathbf{B}_{i,\geqslant n}^{\sigma}\subset \mathbf{B}$ is a basis of the subspace $\mathbf{f}\theta_i^n\subset \mathbf{f}$;\\
{\rm{(c)}} there is a bijection $\pi_{i,n}:\mathbf{B}_{i,0}\rightarrow \mathbf{B}_{i,n}$ such that 
\begin{align*}
\theta_i^{(n)}b=\pi_{i,n}(b)+\sum_{b'\in \mathbf{B}_{i,\geqslant n+1}}c_{b'}b',\ (p_{ni}\otimes \mathrm{Id})r\pi_{i,n}(b)=\theta_i^{(n)}\otimes b+\sum_{b''\in \mathbf{B}_{i,\geqslant 1}}d_{b''}\theta_i^{(n)}\otimes b''
\end{align*}
for any $b\in \mathbf{B}_{i,0}$, where $c_{b'},d_{b''}\in \mathbb{N}[v,v^{-1}]$;\\
{\rm{(d)}} there is a bijection $\pi_{i,n}^{\sigma}:\mathbf{B}_{i,0}^{\sigma}\rightarrow \mathbf{B}_{i,n}^{\sigma}$ such that 
\begin{align*}
b\theta_i^{(n)}=\pi_{i,n}^{\sigma}(b)+\sum_{b'\in \mathbf{B}_{i,\geqslant n+1}^{\sigma}}c'_{b'}b',\ (\mathrm{Id}\otimes p_{ni})r\pi_{i,n}^{\sigma}(b)=b\otimes \theta_i^{(n)}+\sum_{b''\in \mathbf{B}_{i,\geqslant 1}^{\sigma}}d'_{b''}b''\otimes \theta_i^{(n)}
\end{align*}
for any $b\in \mathbf{B}_{i,0}^{\sigma}$, where $c'_{b'},d'_{b''}\in \mathbb{N}[v,v^{-1}]$.
\end{theorem}

By \cite[Section 14.2.5(a)]{Lusztig-1993}, we have $\mathbf{B}=\bigsqcup_{\nu\in \mathbb{N}[I]}\mathbf{B}_{\nu}$, where $\mathbf{B}_{\nu}=\mathbf{B}\cap \mathbf{f}_{\nu}$ is a basis of $\mathbf{f}_{\nu}$. For any $\nu=\sum_{i\in I}\nu_ii\in \mathbb{N}[I], b\in \mathbf{B}_{\nu}$ and $i\in I$, let $t_i(b),t_i^{\sigma}(b)\in [0,\nu_i]$ be the unique integers such that $b\in \mathbf{B}_{i,t_i(b)}\cap \mathbf{B}_{i,t_i^{\sigma}(b)}^{\sigma}$. Let $s_i(b) \in [0,\nu_i]$ to be the largest integer $r$ such that there exists $x\in \mathbf{f}$ such that $\theta_i^{(r)}x=\sum_{b'\in \mathbf{B}}c_{b'}b'$ with $c_b\not=0$. Dually, let $s_i^{\sigma}(b)\in [0,\nu_i]$ to be the largest integer $r$ such that there exists $x\in \mathbf{f}$ such that $x\theta_i^{(r)}=\sum_{b'\in \mathbf{B}}c_{b'}b'$ with $c_b\not=0$.

\begin{corollary}\label{t_i=s_i}
For any $b\in \mathbf{B}$ and $i\in I$, we have 
$$t_i(b)=s_i(b),\ t_i^{\sigma}(b)=s_i^{\sigma}(b).$$
\end{corollary}
\begin{proof}
By Theorem \ref{14.3.2}, we have $t_i(b)\leqslant s_i(b)$. By \cite[Section 11.6(c)]{Lusztig-1991}, we have $b\in \theta_i^{s_i(b)}\mathbf{f}$, and so $b\in \mathbf{B}_{i,\geqslant s_i(b)}$. Thus we have $s_i(b)\leqslant t_i(b)$, and so $t_i(b)=s_i(b)$. The other statement can be proved dually.
\end{proof}

\subsubsection{Canonical basis of $\Lambda_{\lambda}$}\label{Canonical basis of Lambdalambda}

By Theorem \ref{14.3.2}, we know that $\bigcup_{i\in I}\mathbf{B}_{i,\geqslant \langle i,\lambda\rangle+1}^{\sigma}$ is a basis of $\sum_{i\in I}\mathbf{f}\theta_i^{\langle i,\lambda\rangle+1}$. Let $\mathbf{B}(\lambda)=\bigcap_{i\in I}\bigcup_{0\leqslant n\leqslant \langle i,\lambda\rangle}\mathbf{B}_{i,n}^{\sigma}$. The canonical basis of $\Lambda_{\lambda}$ is defined to be
$$\mathbf{B}(\Lambda_{\lambda})=\{b^-\eta_{\lambda}\mid b\in \mathbf{B}(\lambda)\},$$
see \cite[Theorem 14.4.11]{Lusztig-1993}. We summarize some properties of $\mathbf{B}(\Lambda_{\lambda})$ as follows.

\begin{theorem}[{\cite[Proposition 19.3.3, Theorem 19.3.5, 22.1.7]{Lusztig-1993}}]\label{properties of B(Lambda)}
The canonical basis $\mathbf{B}(\Lambda_{\lambda})$ of $\Lambda_{\lambda}$ has the following properties:\\
{\rm{(integral)}} $\mathbf{B}(\Lambda_{\lambda})$ is a $\mathcal{A}$-basis of the integral form ${_{\mathcal{A}}\Lambda_{\lambda}}$;\\
{\rm{(bar-invariant)}} for any $b\in \mathbf{B}(\Lambda_{\lambda})$, we have $\overline{b}=b$;\\
{\rm{(almost orthonormal)}} for any $b,b'\in \mathbf{B}(\Lambda_{\lambda})$, we have $(b,b')_{\lambda}\in \delta_{b,b'}+v^{-1}\mathbb{Z}[v^{-1}]$;\\
{\rm{(positive)}} for any $i\in I$ and $b\in \mathbf{B}(\Lambda_{\lambda})$, we have 
$$E_ib, F_ib\in \sum_{b'\in \mathbf{B}(\Lambda_{\lambda})}\mathbb{N}[v,v^{-1}]b'.$$
\end{theorem}

We remark that the proof of \cite[Theorem 22.1.7]{Lusztig-1993} is still valid for any symmetric Cartan datum, not only the simply-laced Cartan datum.

\subsubsection{Canonical basis of $\Lambda_{\xi}\otimes \Lambda_{\lambda}$}\label{Canonical basis of tensor product}

We refer to \cite{Bao-Wang-2016} for the constructions of the canonical basis $\mathbf{B}(\Lambda_{\xi}\otimes \Lambda_{\lambda})$ of the tensor product $\Lambda_{\xi}\otimes \Lambda_{\lambda}$ which is a generalization of \cite[Chapter 24]{Lusztig-1993}. For convenience, we sketch the construction as follows.

For any $\nu\in \mathbb{N}[I]$, let $\{b^*\mid b\in \mathbf{B}_{\nu}\}$ be the basis of $\mathbf{f}_{\nu}$ dual to the canonical basis $\mathbf{B}_{\nu}$ under the bilinear form $(,)$ of $\mathbf{f}$. The quasi-$\mathcal{R}$-matrix is the linear map $\Theta:\Lambda_{\xi}\otimes \Lambda_{\lambda}\rightarrow \Lambda_{\xi}\otimes \Lambda_{\lambda}$ defined by 
$$\Theta(m_1\otimes m_2)=\sum_{\nu\in \mathbb{N}[I]}(-v)^{\mathrm{tr}\,\nu}\sum_{b\in \mathbf{B}_{\nu}}b^-m_1\otimes b^{*+}m_2\ \textrm{for any $m_1\in \Lambda_{\xi}$ and $m_2\in \Lambda_{\lambda}$}.$$
Note that this is well-defined, since $\Lambda_{\lambda}$ is a highest weight module, and there are only finitely many $b\in \mathbf{B}$ such that $b^{*+}m_2\not=0$ for any $m_2\in \Lambda_{\lambda}$. The $\Psi$-involution of $\Lambda_{\xi}\otimes \Lambda_{\lambda}$ is the $\mathbb{Q}$-linear involution $\Psi:\Lambda_{\xi}\otimes \Lambda_{\lambda}\rightarrow \Lambda_{\xi}\otimes \Lambda_{\lambda}$ defined by
$$\Psi(m)=\Theta(\overline{m})\ \textrm{for any $m\in \Lambda_{\xi}\otimes \Lambda_{\lambda}$}.$$
Then we have $\Psi(um)=\overline{u}\Psi(m)$ for any $u\in \mathbf{U}$ and $m\in \Lambda_{\xi}\otimes \Lambda_{\lambda}$ by \cite[Lemma 24.1.2]{Lusztig-1993}.

Let $\mathbf{B}(\Lambda_{\xi})\otimes \mathbf{B}(\Lambda_{\lambda})=\{b_1\otimes b_2\mid b_1\in \mathbf{B}(\Lambda_{\xi}),b_2\in \mathbf{B}(\Lambda_{\lambda})\}$. Since $\mathbf{B}(\Lambda_{\xi})$ and $\mathbf{B}(\Lambda_{\lambda})$ are almost orthonormal, by definitions, we have 
$$(b_1\otimes b_2,b'_1\otimes b'_2)_{\xi,\lambda}=(b_1,b'_1)_{\xi}(b_2,b'_2)_{\lambda}\in \delta_{b_1\otimes b_2,b'_1\otimes b'_2}+v^{-1}\mathbb{Z}[v^{-1}]\subset \delta_{b_1\otimes b_2,b'_1\otimes b'_2}+v^{-1}\mathbb{Z}[[v^{-1}]]$$
for any $b_1\otimes b_2,b'_1\otimes b'_2\in \mathbf{B}(\Lambda_{\xi})\otimes \mathbf{B}(\Lambda_{\lambda})$, and so $\mathbf{B}(\Lambda_{\xi})\otimes \mathbf{B}(\Lambda_{\lambda})$ is an almost orthonormal basis of $\Lambda_{\xi}\otimes \Lambda_{\lambda}$. Let $\mathcal{L}_{\xi,\lambda}$ be the $\mathbb{Z}[v^{-1}]$-submodule of $\Lambda_{\xi}\otimes \Lambda_{\lambda}$ generated by $\mathbf{B}(\Lambda_{\xi})\otimes \mathbf{B}(\Lambda_{\lambda})$.

\begin{theorem}[{\cite[Theorem 2.9]{Bao-Wang-2016}}]\label{unique canonical basis element}
{\rm{(a)}} For any $b_1\in \mathbf{B}(\Lambda_{\xi})$ and $b_2\in \mathbf{B}(\Lambda_{\lambda})$, there exists a unique element $b_1\diamondsuit b_2\in \mathcal{L}_{\xi,\lambda}$ such that
$$\Psi(b_1\diamondsuit b_2)=b_1\diamondsuit b_2\ \textrm{and}\  b_1\diamondsuit b_2-b_1\otimes b_2\in v^{-1}\mathcal{L}_{\xi,\lambda}.$$
{\rm{(b)}} For any $b_1\in \mathbf{B}(\Lambda_{\xi})$ and $b_2\in \mathbf{B}(\Lambda_{\lambda})$, we have $b_1\diamondsuit b_2\in b_1\otimes b_2+\sum v^{-1}\mathbb{Z}[v^{-1}]b'_1\otimes b'_2$, where the sum is taken over $b'_1\otimes b'_2\in \mathbf{B}(\Lambda_{\xi})\otimes \mathbf{B}(\Lambda_{\lambda})\setminus\{b_1\otimes b_2\}$.\\
{\rm{(c)}} The set $\{b_1\diamondsuit b_2\mid b_1\in \mathbf{B}(\Lambda_{\xi}),b_2\in \mathbf{B}(\Lambda_{\lambda})\}$ is a basis of $\Lambda_{\xi}\otimes \Lambda_{\lambda}$.
\end{theorem}

The canonical basis of $\Lambda_{\xi}\otimes \Lambda_{\lambda}$ is defined to be
$$\mathbf{B}(\Lambda_{\xi}\otimes \Lambda_{\lambda})=\{b_1\diamondsuit b_2\mid b_1\in \mathbf{B}(\Lambda_{\xi}),b_2\in \mathbf{B}(\Lambda_{\lambda})\}.$$

\section{Framed construction of the tensor product}

We fix two dominant weights $\xi,\lambda\in X$ with respect to the root datum $(Y,X,\langle,\rangle,...)$.

\subsection{Framed Cartan datum and root datum}\label{Framed Cartan datum and root datum}
For the Cartan datum $(I,\cdot)$, we define its framed Cartan datum to be $(\tilde{I},\cdot)$, where $\tilde{I}=I\sqcup I'=I\sqcup \{i'\mid i\in I\}$,
and the symmetric bilinear form $\mathbb{Z}[\tilde{I}]\times \mathbb{Z}[\tilde{I}]\rightarrow \mathbb{Z}$ is extended from the bilinear form in $(I,\cdot)$ such that 
$$i\cdot j'=-\delta_{i,j},\ i'\cdot j'=2\delta_{i,j}\ \textrm{for any $i,j\in I$}.$$
Then $(\tilde{I},\cdot)$ is a symmetric Cartan datum. For the root datum $(Y,X,\langle,\rangle,...)$, suppose that $Y=Y^0\oplus \mathbb{Z}[I]$ and $X\cong \mathrm{Hom}_{\mathbb{Z}}(Y^0,\mathbb{Z})\oplus\mathrm{Hom}_{\mathbb{Z}}(\mathbb{Z}[I],\mathbb{Z})$, we define its framed root datum to be $(\tilde{Y},\tilde{X},\langle,\rangle,...)$, where 
$\tilde{Y}=Y\oplus \mathbb{Z}[I'],\ \tilde{X}=X\oplus \mathbb{Z}[\{i'^*\mid i\in I\}]$ with the perfect bilinear pairing $\tilde{Y}\times \tilde{X}\rightarrow \mathbb{Z}$ extended from the pairing in $(Y,X,\langle,\rangle,...)$, and the embeddings $\tilde{I}\hookrightarrow\tilde{Y},\tilde{I}\hookrightarrow\tilde{X}$ extended from the embeddings in $(Y,X,\langle,\rangle,...)$ such that
$$\langle Y^0,i'^*\rangle=\langle i',\mathrm{Hom}_{\mathbb{Z}}(Y^0,\mathbb{Z})\rangle=0, \langle i,j'^*\rangle=\langle i',j^*\rangle=-\delta_{i,j}, \langle i',j'^*\rangle=2\delta_{i,j}\ \textrm{for any $i,j\in I$}.$$
Then $(\tilde{Y},\tilde{X},\langle,\rangle,...)$ is a $\tilde{Y}$-regular root datum of type $(\tilde{I},\cdot)$. The embeddings $\tilde{I}\hookrightarrow \tilde{Y}$ and $\tilde{I}\hookrightarrow \tilde{X}$ induce homomorphisms $\mathbb{Z}[\tilde{I}]\rightarrow \tilde{Y}$ and $\mathbb{Z}[\tilde{I}]\rightarrow \tilde{X}$, and we shall often denote the images of $\nu\in \mathbb{Z}[\tilde{I}]$ under either of these homomorphisms by the same notation $\nu$.

For the framed Cartan datum $(\tilde{I},\cdot)$, we denote by $\tilde{\mathbf{f}}$ the corresponding algebra defined in \ref{The algebra f}. For the framed root datum $(\tilde{Y},\tilde{X},\langle,\rangle,...)$, we denote by $\tilde{\mathbf{U}}$ the corresponding algebra defined in \ref{The algebra U}. It is clear that the natural embedding $I\hookrightarrow \tilde{I}$ induces subalgebra embeddings $\mathbf{f}\hookrightarrow \tilde{\mathbf{f}}$ and $\mathbf{U}\hookrightarrow \tilde{\mathbf{U}}$.

\subsection{Framed construction of $\Lambda_{\xi}\otimes \Lambda_\lambda$}\label{framed construction of tensor product}

For the dominant weight $\lambda\in X$, we define the element 
$$\theta_{\lambda}=\prod_{i\in I}\theta_{i'}^{(\langle i,\lambda\rangle)}\in \tilde{\mathbf{f}}.$$
Note that it is well-defined, since $i'\cdot j'=0$ and $\theta_{i'}\theta_{j'}=\theta_{j'}\theta_{i'}$ for any $i\not=j$ in $I$ by the quantum Serre relations. Consider the sequence $(\langle i,\lambda\rangle i')_{i\in I}\in \mathcal{V}_{|\theta_{\lambda}|}$ in arbitrary order, then $\theta_{\lambda}$ is accordance with the notation $\theta_{(\langle i,\lambda\rangle i')_{i\in I}}$ defined in \S \ref{The algebra f}.

For the dominant weights $\xi,\lambda\in X$, we regard them as elements of $\tilde{X}$ via the natural embedding $X\hookrightarrow \tilde{X}$. For any $i\in I$, let $\omega_{i'}\in \tilde{X}$ be the fundament weight corresponding to $i'\in I'$. By definitions, we have $\langle Y,\omega_{i'}\rangle=0,\langle j',\omega_{i'}\rangle=\delta_{i,j}$ for any $j\in I$. We consider
$$\xi\odot \lambda=\xi+\lambda+|\theta_{\lambda}|+\sum_{i\in I}(\langle i,\lambda\rangle-\langle i',\xi+\lambda+|\theta_{\lambda}|\rangle)\omega_{i'}\in \tilde{X},$$
and denote by $\tilde{M}_{\xi\odot \lambda}$ the Verma module of $\tilde{\mathbf{U}}$ defined in \ref{The module Lambda}. Note that
\begin{align*}
&\langle \mu,\xi\odot \lambda\rangle=\langle \mu,\xi+\lambda+|\theta_{\lambda}|\rangle\ \textrm{for any $\mu\in Y$};\\
&\langle i,\xi\odot \lambda\rangle=\langle i,\xi\rangle,\ \langle i',\xi\odot \lambda\rangle=\langle i,\lambda\rangle\ \textrm{for any $i\in I$}.
\end{align*}
In particular, $\xi\odot \lambda\in \tilde{X}$ is dominant with respect to the framed root datum $(\tilde{Y},\tilde{X},\langle,\rangle,...)$. We denote by $\tilde{\Lambda}_{\xi\odot \lambda}$ the integrable highest weight module of $\tilde{\mathbf{U}}$ defined in \ref{The module Lambda}. Then $\tilde{M}_{\xi\odot \lambda}$ and $\tilde{\Lambda}_{\xi\odot \lambda}$ are $\mathbf{U}$-modules via the natural embedding $\mathbf{U}\hookrightarrow \tilde{\mathbf{U}}$.

\begin{lemma}\label{submodules}
Let $\mathbf{f}\theta_{\lambda}\mathbf{f}$ be the subspace of $\tilde{\mathbf{f}}$ spanned by $\{x\theta_{\lambda}y\mid x,y\in \mathbf{f}\}$. Then it is a $\mathbf{U}$-submodule of $\tilde{M}_{\xi\odot \lambda}$.
\end{lemma}
\begin{proof}
It is well-known that any submodule of a weight module is also a weight module. It is clear that $\mathbf{f}\theta_{\lambda}\mathbf{f}$ is stable under the actions of $F_i$ and $K_\mu$ for any $i\in I$ and $\mu\in Y$. We need to prove that it is stable under the actions of $E_i$ for any $i\in I$. For any homogeneous $x,y\in \mathbf{f}$, we use induction on $\mathrm{tr}\,|x|\in \mathbb{N}$ to prove that $E_ix\theta_{\lambda}y\in \mathbf{f}\theta_{\lambda}\mathbf{f}$.\\
$\bullet$ If $\mathrm{tr}\,|x|=0$, we have 
\begin{align*}
E_i\theta_{\lambda}y=E_i\prod_{j\in I}F_{j'}^{(\langle j,\lambda\rangle)}y=\prod_{j\in I}F_{j'}^{(\langle j,\lambda\rangle)}E_iy=\theta_{\lambda}E_iy.
\end{align*}
Since $\mathbf{f}$ has a $\mathbf{U}$-module structure, it is claear that $E_iy\in \mathbf{f}$, and so $E_i\theta_{\lambda}y\in \mathbf{f}\theta_{\lambda}\mathbf{f}$.\\
$\bullet$ If $\mathrm{tr}\,|x|>0$, we may suppose that $x=\theta_kx'$ for some $k\in I$ and homogeneous $x'\in \mathbf{f}$. Notice that $\mathrm{tr}\,|x'|=\mathrm{tr}\,|x|-1$. By the inductive hypothesis, we have $E_ix'\theta_{\lambda}y\in \mathbf{f}\theta_{\lambda}\mathbf{f}$.
By Lemma \ref{3.4.2}, we have 
\begin{align*}
E_ix\theta_{\lambda}y=E_i\theta_kx'\theta_{\lambda}y=E_iF_kx'\theta_{\lambda}y
=F_kE_ix'\theta_{\lambda}y+\delta_{i,k}[\langle i,\xi\odot \lambda-|x'|-|\theta_{\lambda}|-|y|\rangle]x'\theta_{\lambda}y
\end{align*}
which belongs to $F_k\mathbf{f}\theta_{\lambda}\mathbf{f}+\mathbf{f}\theta_{\lambda}\mathbf{f}\subset \mathbf{f}\theta_{\lambda}\mathbf{f}$, as desired. Hence $\mathbf{f}\theta_{\lambda}\mathbf{f}$ is a submodule of $\tilde{M}_{\xi\odot \lambda}$.
\end{proof}

By \cite[Proposition 3.5.6(a)]{Lusztig-1993}, the left ideal $\sum_{i\in I}\tilde{\mathbf{f}}\theta_{i}^{\langle i,\xi\odot \lambda\rangle+1}+\sum_{i\in I}\tilde{\mathbf{f}}\theta_{i'}^{\langle i',\xi\odot \lambda\rangle+1}$ of $\tilde{\mathbf{f}}$ is a submodule of $\tilde{M}_{\xi\odot \lambda}$. By definitions, $\tilde{\Lambda}_{\xi\odot \lambda}$ is the quotient of $\tilde{M}_{\xi\odot \lambda}$ with respect to this submodule. Note that
$$\sum_{i\in I}\tilde{\mathbf{f}}\theta_{i}^{\langle i,\xi\odot \lambda\rangle+1}+\sum_{i\in I}\tilde{\mathbf{f}}\theta_{i'}^{\langle i',\xi\odot \lambda\rangle+1}=\sum_{i\in I}\tilde{\mathbf{f}}\theta_{i}^{\langle i,\xi\rangle+1}+\sum_{i\in I}\tilde{\mathbf{f}}\theta_{i'}^{\langle i,\lambda\rangle+1}.$$

\begin{lemma}\label{cap+=cap}
We have $$\mathbf{f}\theta_{\lambda}\mathbf{f}\cap(\sum_{i\in I}\tilde{\mathbf{f}}\theta_{i}^{\langle i,\xi\rangle+1}+\sum_{i\in I}\tilde{\mathbf{f}}\theta_{i'}^{\langle i, \lambda\rangle+1})=\mathbf{f}\theta_{\lambda}\mathbf{f}\cap\sum_{i\in I}\tilde{\mathbf{f}}\theta_{i}^{\langle i,\xi\rangle+1}.$$
\end{lemma}
\begin{proof}
The proof will be given after Proposition \ref{B(fthetaf)}.  
\end{proof}

\begin{definition}\label{Lambdaxilambda}
We define the $\mathbf{U}$-module $\Lambda_{\xi,\lambda}\subset \tilde{\Lambda}_{\xi\odot \lambda}$ to be the image of $\mathbf{f}\theta_{\lambda}\mathbf{f}\subset \tilde{M}_{\xi\odot \lambda}$ under the natural projection $\pi_{\xi\odot \lambda}:\tilde{M}_{\xi\odot \lambda}\rightarrow \tilde{\Lambda}_{\xi\odot \lambda}$. More precisely, 
\begin{align*}
\Lambda_{\xi,\lambda}=&(\mathbf{f}\theta_{\lambda}\mathbf{f}+\sum_{i\in I}\tilde{\mathbf{f}}\theta_{i}^{\langle i,\xi\odot \lambda\rangle+1}+\sum_{i\in I}\tilde{\mathbf{f}}\theta_{i'}^{\langle i',\xi\odot \lambda\rangle+1})/(\sum_{i\in I}\tilde{\mathbf{f}}\theta_{i}^{\langle i,\xi\odot \lambda\rangle+1}+\sum_{i\in I}\tilde{\mathbf{f}}\theta_{i'}^{\langle i',\xi\odot \lambda\rangle+1})\\
\cong&\mathbf{f}\theta_{\lambda}\mathbf{f}/(\mathbf{f}\theta_{\lambda}\mathbf{f}\cap(\sum_{i\in I}\tilde{\mathbf{f}}\theta_{i}^{\langle i,\xi\odot \lambda\rangle+1}+\sum_{i\in I}\tilde{\mathbf{f}}\theta_{i'}^{\langle i',\xi\odot \lambda\rangle+1}))=\mathbf{f}\theta_{\lambda}\mathbf{f}/(\mathbf{f}\theta_{\lambda}\mathbf{f}\cap\sum_{i\in I}\tilde{\mathbf{f}}\theta_{i}^{\langle i,\xi\rangle+1}).
\end{align*}
In particular, we have the $\mathbf{U}$-module $\Lambda_{0,\lambda}\subset \tilde{\Lambda}_{0\odot \lambda}$ such that $\Lambda_{0,\lambda}\cong \mathbf{f}\theta_{\lambda}\mathbf{f}/(\mathbf{f}\theta_{\lambda}\mathbf{f}\cap \sum_{i\in I}\tilde{\mathbf{f}}\theta_i)$. For convenience, we identify $\Lambda_{\xi,\lambda}$ with $\mathbf{f}\theta_{\lambda}\mathbf{f}/(\mathbf{f}\theta_{\lambda}\mathbf{f}\cap\sum_{i\in I}\tilde{\mathbf{f}}\theta_{i}^{\langle i,\xi\rangle+1})$. 
\end{definition}

\begin{proposition}\label{psi}
The right multiplication by $\theta_{\lambda}$ defines a $\mathbf{U}$-module homomorphism 
$$\cdot\theta_{\lambda}:M_{\lambda}=\mathbf{f}\rightarrow \mathbf{f}\theta_{\lambda}\mathbf{f},\,x\mapsto x\theta_{\lambda},$$
where we regard $\mathbf{f}\theta_{\lambda}\mathbf{f}\subset \tilde{M}_{0\odot\lambda}$. Moreover, it induces a $\mathbf{U}$-module isomorphism
$$\psi:\Lambda_{\lambda}=M_{\lambda}/\sum_{i\in I}\mathbf{f}\theta_i^{\langle i,\lambda\rangle+1}\rightarrow \mathbf{f}\theta_{\lambda}\mathbf{f}/(\mathbf{f}\theta_{\lambda}\mathbf{f}\cap \sum_{i\in I}\tilde{\mathbf{f}}\theta_i)=\Lambda_{0,\lambda}.$$
\end{proposition} 
\begin{proof}
It is clear that $\cdot\theta_{\lambda}$ is linear, and it commutes with the actions of $F_i$ for any $i\in I$. For any $\mu\in Y$ and homogeneous $x\in \mathbf{f}$, we have 
$$(K_{\mu} x)\theta_\lambda=v^{\langle \mu,\lambda-|x|\rangle}x\theta_{\lambda}=v^{\langle\mu,\lambda+|\theta_{\lambda}|-|x\theta_\lambda|\rangle}x\theta_{\lambda}=v^{\langle \mu,0\odot \lambda-|x\theta_\lambda|\rangle}=K_{\mu}(x\theta_{\lambda}).$$
For any $i\in I$, we use induction on $\mathrm{tr}\,|x|\in \mathbb{N}$ to prove that $(E_ix)\theta_{\lambda}=E_i(x\theta_{\lambda})$.\\
$\bullet$ If $\mathrm{tr}\,|x|=0$, we have $(E_i1)\theta_{\lambda}=0,\ E_i\theta_{\lambda}=\prod_{i\in I}F_{i'}^{(\langle i,\lambda\rangle)}E_i1=0$.\\
$\bullet$ If $\mathrm{tr}\,|x|>0$, we may suppose that $x=\theta_jx'$ for some $j\in I$ and homogeneous $x'\in \mathbf{f}$. Notice that $\mathrm{tr}\,|x'|=\mathrm{tr}\,|x|-1$. By the inductive hypothesis, we have $(E_ix')\theta_{\lambda}=E_i(x'\theta_{\lambda})$. By Lemma \ref{3.4.2}, we have 
\begin{align*}
&(E_ix)\theta_{\lambda}=(E_i\theta_jx')\theta_{\lambda}=(E_iF_jx')\theta_{\lambda}
=(F_jE_ix'+\delta_{i,j}[\langle i,\lambda-|x'|\rangle]x')\theta_{\lambda},\\
&E_i(x\theta_{\lambda})=E_i(\theta_jx'\theta_{\lambda})=E_iF_j(x'\theta_{\lambda})=F_jE_i(x'\theta_{\lambda})+\delta_{i,j}[\langle i,0\odot \lambda-|x'|-|\theta_{\lambda}|\rangle]x'\theta_{\lambda}.
\end{align*}
Notice that $\langle i,0\odot \lambda-|x'|-|\theta_{\lambda}|\rangle=\langle i,\lambda-|x'|\rangle$. Since $(E_ix')\theta_{\lambda}=E_i(x'\theta_{\lambda})$ and $\cdot\theta_{\lambda}$ commutes the action of $F_i$, we have $(E_ix)\theta_{\lambda}=E_i(x\theta_{\lambda})$, as desired. Hence the map $\cdot\theta_{\lambda}$ is a $\mathbf{U}$-module homomorphism. Moreover, since $i\cdot j'=-\delta_{i,j}$ for any $i,j\in I$, we have 
\begin{align*}
&\theta_i^{\langle i,\lambda\rangle+1}\theta_{j'}^{(\langle j,\lambda\rangle)}=\theta_{j'}^{(\langle j,\lambda\rangle)}\theta_i^{\langle i,\lambda\rangle+1}\ \textrm{if $j\not=i$};\\
&\theta_i^{(\langle i,\lambda\rangle+1)}\theta_{i'}^{(\langle i,\lambda\rangle)}=\sum_{n=0}^{\langle i,\lambda\rangle}\frac{v^{-(\langle i,\lambda\rangle+1-n)(\langle i,\lambda\rangle-n)}}{[n]!}\theta_{i'}^{(\langle i,\lambda\rangle-n)}(\theta_i\theta_{i'}-v^{-1}\theta_{i'}\theta_i)^n\theta_i^{(\langle i,\lambda\rangle+1-n)}
\end{align*}
by the quantum Serre relations and \cite[Lemma 42.1.2(b)]{Lusztig-1993}. As a result, we have $$(\sum_{i\in I}\mathbf{f}\theta^{\langle i,\lambda\rangle+1})\theta_\lambda\subset \mathbf{f}\theta_{\lambda}\mathbf{f}\cap \sum_{i\in I}\tilde{\mathbf{f}}\theta_i,$$ and so $\cdot\theta_{\lambda}$ induces a well-defined $\mathbf{U}$-module homomorphism $\psi:\Lambda_{\lambda}\rightarrow \Lambda_{0,\lambda}$. It is clear that $\psi$ is surjective, and $\Lambda_{0,\lambda}\not=0$. Since $\Lambda_{\lambda}$ is an integrable highest weight module, $\psi$ is an isomorphism.
\end{proof}

For any $\nu\in \mathbb{N}[I]$, we regard it as an element of $\mathbb{N}[\tilde{I}]$ via the natural embedding $I\hookrightarrow \tilde{I}$. Consider the following composition of linear maps
\begin{align*}
\varphi'':&\mathbf{f}\theta_{\lambda}\mathbf{f}\hookrightarrow\tilde{\mathbf{f}}\xrightarrow{r}\tilde{\mathbf{f}}\otimes \tilde{\mathbf{f}}\xrightarrow{p\otimes \mathrm{Id}}\bigoplus_{\nu\in \mathbb{N}[I]}\tilde{\mathbf{f}}_{\nu+|\theta_{\lambda}|}\otimes \tilde{\mathbf{f}}\xrightarrow{\tau}\tilde{\mathbf{f}}\otimes \bigoplus_{\nu\in \mathbb{N}[I]}\tilde{\mathbf{f}}_{\nu+|\theta_{\lambda}|}\\
&\hookrightarrow \tilde{\mathbf{f}}\otimes \tilde{\mathbf{f}}=\tilde{\mathbf{f}}\otimes \tilde{M}_{0\odot \lambda}\xrightarrow{\mathrm{Id}\otimes \pi_{0\odot \lambda}} \tilde{\mathbf{f}}\otimes \tilde{\Lambda}_{0\odot \lambda},
\end{align*}
where $r$ is the comultiplication of $\tilde{\mathbf{f}}$ defined in \ref{The algebra f}, $p$ is the natural projection from $\tilde{\mathbf{f}}$ to the direct sum of its certain homogeneous components, $\tau$ is the natural isomorphism swapping two factors, and $\pi_{0\odot \lambda}$ is the natural projection.

\begin{lemma}\label{imvarphi''=fLambda0lambda}
The image of the linear map $\varphi''$ is contained in $\mathbf{f}\otimes \Lambda_{0,\lambda}$, and so there is a linear map $\varphi'':\mathbf{f}\theta_{\lambda}\mathbf{f}\rightarrow \mathbf{f}\otimes \Lambda_{0,\lambda}$.
\end{lemma}
\begin{proof}
For any homogeneous $x,y\in \mathbf{f}$, we calculate the image $\varphi''(x\theta_{\lambda}y)$ as follows. Suppose that $r(x)=\sum x_1\otimes x_2,r(y)=\sum y_1\otimes y_2$, where $x_1,x_2,y_1,y_2\in \mathbf{f}$ are homogeneous. By definitions, we have  
\begin{align*}
\tau(p\otimes \mathrm{Id})r(x\theta_{\lambda}y)=&\tau(\sum (x_1\otimes x_2)(\theta_{\lambda}\otimes 1)(y_1\otimes y_2))\\
=&\tau(\sum v^{|x_2|\cdot|\theta_{\lambda}|+|x_2|\cdot |y_1|}x_1\theta_{\lambda}y_1\otimes x_2y_2)\\
=&\sum v^{|x_2|\cdot|\theta_{\lambda}|+|x_2|\cdot |y_1|} x_2y_2\otimes x_1\theta_{\lambda}y_1.
\end{align*}
Notice that $\pi_{0\odot \lambda}(x_1\theta_{\lambda}y_1)=0$ whenever $|y_1|\not=0$. Hence we have
\begin{equation}\label{Im varphi''}
\varphi''(x\theta_{\lambda}y)=\sum v^{|x_2|\cdot|\theta_{\lambda}|}x_2y\otimes \pi_{0\odot \lambda}(x_1\theta_{\lambda})\in \mathbf{f}\otimes \Lambda_{0,\lambda}.
\end{equation}
Hence we have $\varphi''(\mathbf{f}\theta_{\lambda}\mathbf{f})\subset \mathbf{f}\otimes \Lambda_{0,\lambda}$.
\end{proof}

Moreover, consider the composition of linear maps
$$\varphi':\mathbf{f}\theta_{\lambda}\mathbf{f}\xrightarrow{\varphi''}  \mathbf{f}\otimes \Lambda_{0,\lambda}=M_{\xi}\otimes \Lambda_{0,\lambda}\xrightarrow{\pi_{\xi}\otimes \psi^{-1}}\Lambda_{\xi}\otimes \Lambda_{\lambda},$$
where $\pi_{\xi}$ is the natural projection, and $\psi$ is the $\mathbf{U}$-module isomorphism in Proposition \ref{psi}. 

\begin{corollary}\label{Im varphi'}
For any $x,y\in \mathbf{f}$, suppose that $r(x)=\sum x_1\otimes x_2$, where $x_1,x_2\in \mathbf{f}$ are homogeneous, then we have 
\begin{align*}
\varphi'(x\theta_{\lambda}y)=\sum v^{|x_2|\cdot|\theta_{\lambda}|}x_2^-y^-\eta_{\xi}\otimes x_1^-\eta_{\lambda}.
\end{align*}
\end{corollary}
\begin{proof}
It follows from (\ref{Im varphi''}) and definitions.
\end{proof}

\begin{lemma}\label{surjections are isomorphisms}
Let $\mathbf{U}^0$ be the subalgebra $\mathbf{U}$ generated by $K_{\mu}$ for any $\mu\in Y$. Then any surjective $\mathbf{U}^0$-module homomorphisms from $\Lambda_{\xi,\lambda}$ to $\Lambda_{\xi}\otimes \Lambda_{\lambda}$ are $\mathbf{U}^0$-module isomorphisms.
\end{lemma}
\begin{proof}
The proof will be given after Corollary \ref{framed basis}.
\end{proof}

\begin{theorem}\label{framed construction of tensor}
The linear map $\varphi':\mathbf{f}\theta_{\lambda}\mathbf{f}\rightarrow \Lambda_{\xi}\otimes \Lambda_{\lambda}$ is a $\mathbf{U}$-module homomorphism, and it induces a $\mathbf{U}$-module isomorphism $$\varphi:\Lambda_{\xi,\lambda}=\mathbf{f}\theta_{\lambda}\mathbf{f}/(\mathbf{f}\theta_{\lambda}\mathbf{f}\cap \sum_{i\in I}\tilde{\mathbf{f}}\theta_i^{\langle i,\xi\rangle+1})\rightarrow \Lambda_{\xi}\otimes \Lambda_{\lambda}$$
such that
\begin{align*}
\varphi\pi_{\xi\odot \lambda}(x\theta_{\lambda}y)=\sum v^{|x_2|\cdot|\theta_{\lambda}|}x_2^-y^-\eta_{\xi}\otimes x_1^-\eta_{\lambda}
\end{align*}
for any $x,y\in \mathbf{f}$, where $r(x)=\sum x_1\otimes x_2$ and $x_1,x_2\in \mathbf{f}$ are homogeneous.
\end{theorem}
\begin{proof}
Firstly, we prove that $\varphi'$ is a $\mathbf{U}$-module homomorphism. For any homogeneous $x,y\in \mathbf{f}$, suppose that $r(x)=\sum x_1\otimes x_2$, where $x_1,x_2\in \mathbf{f}$ are homogeneous. Notice that $|x|=|x_1|+|x_2|$. For any $\mu\in Y$, by Corollary \ref{Im varphi'}, we have
\begin{align*}
\varphi'(K_\mu x\theta_{\lambda}y)=&\varphi'(v^{\langle\mu,\xi\odot \lambda-|x|-|\theta_{\lambda}|-|y|\rangle}x\theta_{\lambda}y)\\
=&\varphi'(v^{\langle\mu,\xi+\lambda+|\theta_{\lambda}|-|x|-|\theta_{\lambda}|-|y|\rangle}x\theta_{\lambda}y)\\
=&\sum v^{|x_2|\cdot|\theta_{\lambda}|}v^{\langle \mu,\xi-|x_2|-|y|\rangle}x_2^-y^-\eta_{\xi}\otimes v^{\langle \mu,\lambda-|x_1|\rangle}x_1^-\eta_{\lambda}\\
=&\sum v^{|x_2|\cdot|\theta_{\lambda}|}K_\mu x_2^-y^-\eta_{\xi}\otimes K_\mu x_1^-\eta_{\lambda}=K_\mu \varphi'(x\theta_{\lambda}y).
\end{align*}
For any $i\in I$, we have $r(\theta_ix)=(\theta_i\otimes 1+1\otimes \theta_i)\sum x_1\otimes x_2=\sum \theta_ix_1\otimes x_2+\sum v^{i\cdot |x_1|}x_1\otimes \theta_ix_2$. By Corollary \ref{Im varphi'}, we have
\begin{align*}
\varphi'(F_ix\theta_{\lambda}y)=&\varphi'(\theta_ix\theta_{\lambda}y)\\
=&\sum v^{|x_2|\cdot|\theta_{\lambda}|}x_2^-y^-\eta_{\xi}\otimes F_ix_1^-\eta_{\lambda}
+\sum v^{i\cdot |x_1|}v^{|\theta_ix_2|\cdot|\theta_{\lambda}|}F_ix_2^-y^-\eta_{\xi}\otimes x_1^-\eta_{\lambda}\\
=&\sum v^{|x_2|\cdot|\theta_{\lambda}|}x_2^-y^-\eta_{\xi}\otimes F_ix_1^-\eta_{\lambda}
+\sum v^{|x_2|\cdot|\theta_{\lambda}|}F_ix_2^-y^-\eta_{\xi}\otimes v^{\langle -i,\lambda-|x_1|\rangle}x_1^-\eta_{\lambda}\\
=&\sum v^{|x_2|\cdot|\theta_{\lambda}|}x_2^-y^-\eta_{\xi}\otimes F_ix_1^-\eta_{\lambda}
+\sum v^{|x_2|\cdot|\theta_{\lambda}|}F_ix_2^-y^-\eta_{\xi}\otimes K_{-i}x_1^-\eta_{\lambda}\\
=&F_i\varphi'(x\theta_{\lambda}y).
\end{align*}
For any $i\in I$, we use induction on $\mathrm{tr}\,|x|\in \mathbb{N}$ to prove that $\varphi'(E_ix\theta_{\lambda}y)=E_i\varphi'(x\theta_{\lambda}y)$.\\
$\bullet$ If $\mathrm{tr}\,|x|=0$, recall that we have proved that $E_iy\in \mathbf{f}$ in the proof of Lemma \ref{submodules}. By Corollary \ref{Im varphi'}, we have
\begin{align*}
\varphi'(E_i\theta_{\lambda}y)=&\varphi'(E_i\prod_{i\in I}F_{i'}^{(\langle i,\lambda\rangle)}y)=\varphi'(\prod_{i\in I}F_{i'}^{(\langle i,\lambda\rangle)}E_iy)=\varphi'(\theta_{\lambda}(E_iy))=(E_iy)^-\eta_{\xi}\otimes \eta_{\lambda}.
\end{align*}
Since $E_i1=0$ in $\tilde{M}_{\xi\odot\lambda}$ and $E_i\eta_{\xi}=0$ in $\Lambda_\xi$, by \cite[Proposition 3.1.6]{Lusztig-1993}, we have
\begin{align*}
&E_iy=E_iy^-1=\frac{K_i{(_ir(y))}^-1-(r_i(y))^-K_{-i}1}{v-v^{-1}}=\frac{v^{\langle i,\xi\odot\lambda-|y|\rangle+2}{_ir(y)}-v^{\langle -i,\xi\odot\lambda \rangle}r_i(y)}{v-v^{-1}},\\
&E_iy^-\eta_{\xi}=\frac{K_i{(_ir(y))}^--(r_i(y))^-K_{-i}}{v-v^{-1}}\eta_{\xi}=\frac{v^{\langle i,\xi-|y|\rangle+2}({_ir(y)})^--v^{\langle -i,\xi \rangle}(r_i(y))^-}{v-v^{-1}}\eta_{\xi}.
\end{align*}
Since $\langle i,\xi\odot\lambda\rangle=\langle i,\xi\rangle$, we have $(E_iy)^-\eta_{\xi}=E_iy^-\eta_{\xi}$, and so
$$\varphi'(E_i\theta_{\lambda}y)=(E_iy)^-\eta_{\xi}\otimes \eta_{\lambda}=E_iy^-\eta_{\xi}\otimes \eta_{\lambda}=E_i(y^-\eta_{\xi}\otimes \eta_{\lambda})=E_i\varphi'(\theta_{\lambda}y).$$
$\bullet$ If $\mathrm{tr}\,|x|>0$, we may suppose that $x=\theta_jx'$ for some $j\in I$ and homogeneous $x'\in \mathbf{f}$. Notice that $\mathrm{tr}\,|x'|=\mathrm{tr}\,|x|-1$. By the inductive hypothesis, we have $\varphi'(E_ix'\theta_{\lambda}y)=E_i\varphi'(x'\theta_{\lambda}y)$. By Lemma \ref{3.4.2} and Corollary \ref{Im varphi'}, we have
\begin{align*}
\varphi'(E_ix\theta_{\lambda}y)=&\varphi'(E_iF_jx'\theta_{\lambda}y)\\
=&\varphi'(F_jE_ix'\theta_{\lambda}y+\delta_{i,j}[\langle i,\xi\odot \lambda-|x'|-|\theta_{\lambda}|-|y|\rangle]x'\theta_{\lambda}y)\\
=&F_j\varphi'(E_ix'\theta_{\lambda}y)+\delta_{i,j}[\langle i,\xi-|x'|-|\theta_{\lambda}|-|y|\rangle]\varphi'(x'\theta_{\lambda}y)\\
=&F_jE_i\varphi'(x'\theta_{\lambda}y)+\delta_{i,j}[\langle i,\xi+\lambda-|x'|-|y|\rangle]\varphi'(x'\theta_{\lambda}y)\\
=&E_iF_j\varphi'(x'\theta_{\lambda}y)=E_i\varphi'(F_jx'\theta_{\lambda}y)=E_i\varphi'(x\theta_{\lambda}y),
\end{align*}
as desired. Hence the map $\varphi'$ is a $\mathbf{U}$-module homomorphism. 

Secondly, we prove that $\varphi'(\mathbf{f}\theta_{\lambda}\mathbf{f}\cap \sum_{i\in I}\tilde{\mathbf{f}}\theta_i^{\langle i,\xi\rangle+1})=0$. It suffices to prove that $$\varphi''(\mathbf{f}\theta_{\lambda}\mathbf{f}\cap \sum_{i\in I}\tilde{\mathbf{f}}\theta_i^{(\langle i,\xi\rangle+1)})\subset \sum_{i\in I}\mathbf{f}\theta_i^{\langle i,\xi\rangle+1}\otimes \Lambda_{0,\lambda}.$$
For any $\sum_{i\in I} x_i\theta_i^{(\langle i,\xi\rangle+1)}\in \mathbf{f}\theta_{\lambda}\mathbf{f}\cap \sum_{i\in I}\tilde{\mathbf{f}}\theta_i^{(\langle i,\xi\rangle+1)}$, where $x_i\in \tilde{\mathbf{f}}$, we calculate the image $\varphi''(\sum_{i\in I} x_i\theta_i^{(\langle i,\xi\rangle+1)})$ as follows. Suppose that $r(x_i)=\sum x_{i1}\otimes x_{i2}$, where $x_{i1},x_{i2}\in \tilde{\mathbf{f}}$ are homogeneous, by definitions and \cite[Lemma 1.4.2]{Lusztig-1993}, we have
\begin{align*}
&\tau(p\otimes \mathrm{Id})r(\sum_{i\in I}x_i\theta_i^{(\langle i,\xi\rangle+1)})\\
=&\tau(p\otimes \mathrm{Id})(\sum_{i\in I}\sum (x_{i1}\otimes x_{i2})\sum_{n=0}^{\langle i,\xi\rangle+1}v^{n(\langle i,\xi\rangle+1-n)}(\theta_i^{(n)}\otimes\theta_i^{(\langle i,\xi\rangle+1-n)}))\\
=&\tau(\sum_{i\in I}\sum\sum_{n=0}^{\langle i,\xi\rangle+1}v^{n(\langle i,\xi\rangle+1-n)+|x_{i2}|\cdot ni}x_{i1}\theta_i^{(n)}\otimes x_{i2}\theta_i^{(\langle i,\xi\rangle+1-n)})\\
=&\sum_{i\in I}\sum\sum_{n=0}^{\langle i,\xi\rangle+1}v^{n(\langle i,\xi\rangle+1-n)+|x_{i2}|\cdot ni}x_{i2}\theta_i^{(\langle i,\xi\rangle+1-n)}\otimes x_{i1}\theta_i^{(n)},
\end{align*}
where the second sum is taken over $x_{i1},x_{i2}\in \tilde{\mathbf{f}}$ such that $|x_{i1}|-|\theta_\lambda|,|x_{i2}|\in \mathbb{N}[I]$. In particular, we have $x_{i2}\in \mathbf{f}$. Notice that $\pi_{0\odot\lambda}(x_{i1}\theta_i^{(n)})=0$ whenever $n\not=0$, and so
\begin{align*}
\varphi''(\sum_{i\in I} x_i\theta_i^{(\langle i,\xi\rangle+1)})=\sum_{i\in I}\sum x_{i2}\theta_i^{(\langle i,\xi\rangle+1)}\otimes \pi_{0\odot\lambda}(x_{i1})\in \sum_{i\in I}\mathbf{f}\theta_i^{\langle i,\xi\rangle+1}\otimes \Lambda_{0,\lambda}.
\end{align*}
Hence $\varphi'$ induces a well-defined $\mathbf{U}$-module homomorphism $\varphi:\Lambda_{\xi,\lambda}\rightarrow \Lambda_{\xi}\otimes \Lambda_{\lambda}$. 

Thirdly, we prove that $\varphi$ is an isomorphism. Notice that the vector space $\Lambda_{\xi}\otimes \Lambda_{\lambda}$ is spanned by $\{y^-\eta_{\xi}\otimes x^-\eta_{\lambda}\mid x,y\in \mathbf{f}\ \textrm{are homogeneous}\}$. We use induction on $\mathrm{tr}\,|x|\in \mathbb{N}$ to prove that $y^-\eta_{\xi}\otimes x^-\eta_{\lambda}\in \mathrm{Im}\,\varphi$.\\
$\bullet$ If $\mathrm{tr}\,|x|=0$, by Corollary \ref{Im varphi'}, we have
$y^-\eta_{\xi}\otimes \eta_{\lambda}=\varphi\pi_{\xi\odot\lambda}(\theta_{\lambda}y)\in\mathrm{Im}\,\varphi$.\\
$\bullet$ If $\mathrm{tr}\,|x|>0$, suppose that $r(x)=x\otimes 1+\sum x_1\otimes x_2$, where $x_1,x_2\in \mathbf{f}$ are homogeneous such that $(x_1,x_2)\not=(x,1)$. Notice that such $x_1$ satisfies $\mathrm{tr}\,|x_1|<\mathrm{tr}\,|x|$. By the inductive hypothesis, we have $x_2^-y^-\eta_{\xi}\otimes x_1^-\eta_{\lambda}\in \mathrm{Im}\,\varphi$. By Corollary \ref{Im varphi'}, we have
\begin{align*}
\varphi\pi_{\xi\odot\lambda}(x\theta_{\lambda}y)=y^-\eta_{\xi}\otimes x^-\eta_{\lambda}+\sum v^{|x_2|\cdot|\theta_\lambda|}x_2^-y^-\eta_{\xi}\otimes x_1^-\eta_{\lambda}\in \mathrm{Im}\,\varphi,
\end{align*}
and so $y^-\eta_{\xi}\otimes x^-\eta_{\lambda}\in \mathrm{Im}\,\varphi$, as desired. Hence $\varphi$ is surjective. By Lemma \ref{surjections are isomorphisms}, $\varphi$ is an isomorphism.
\end{proof}

We remark that Theorem \ref{framed construction of tensor} still hold for any symmetrizable Cartan datum.

It is clear that the bar-involution of $\tilde{\mathbf{f}}$ can be restricted to be a bar-involution of the subspace $\mathbf{f}\theta_{\lambda}\mathbf{f}$, and the bar-involution of $\tilde{\Lambda}_{\xi\odot\lambda}$ can be restricted to be a bar-involution of the submodule $\Lambda_{\xi,\lambda}$ such that $\pi_{\xi\odot\lambda}(\overline{x})=\overline{\pi_{\xi\odot\lambda}(x)}$ for any $x\in \mathbf{f}\theta_{\lambda}\mathbf{f}$. 

\begin{lemma}\label{varphi commutes with involutions}
For any $m\in \Lambda_{\xi,\lambda}$, we have $\varphi(\overline{m})=\Psi\varphi(m)$.
\end{lemma}
\begin{proof}
It suffices to prove that $\varphi'(\overline{x})=\Psi\varphi'(x)$ for any $x\in \mathbf{f}\theta_{\lambda}\mathbf{f}$. Notice that $\mathbf{f}\theta_{\lambda}\mathbf{f}$ is spanned by $\{\theta_{\bnu_1}\theta_{\lambda}\theta_{\bnu_2}\mid \bnu_1\in \mathcal{V}_{\nu_1},\bnu_2\in \mathcal{V}_{\nu_2},\nu_1,\nu_2\in \mathbb{N}[I]\}$. We use induction on $\mathrm{tr}\,\nu_1\in \mathbb{N}$ to prove that $\varphi'(\overline{\theta_{\bnu_1}\theta_{\lambda}\theta_{\bnu_2}})=\Psi\varphi'(\theta_{\bnu_1}\theta_{\lambda}\theta_{\bnu_2})$.\\
$\bullet$ If $\mathrm{tr}\,\nu_1=0$, by definitions and Corollary \ref{Im varphi'}, we have
\begin{align*}
\Psi\varphi'(\theta_{\lambda}\theta_{\bnu_2})=&\Theta(\overline{\varphi'(\theta_{\lambda}\theta_{\bnu_2})})=\Theta(\overline{\theta_{\bnu_2}^-\eta_{\xi}\otimes \eta_{\lambda}})
=\Theta(\theta_{\bnu_2}^-\eta_{\xi}\otimes \eta_{\lambda})\\=&
\theta_{\bnu_2}^-\eta_{\xi}\otimes \eta_{\lambda}
=\varphi'(\theta_{\lambda}\theta_{\bnu_2})=\varphi'(\overline{\theta_{\lambda}\theta_{\bnu_2}}),
\end{align*}
where we use $b^{*+}\eta_{\lambda}=0$ for any $b\in \mathbf{B}\setminus\{1\}$ to calculate $\Theta(\theta_{\bnu_2}^-\eta_{\xi}\otimes \eta_{\lambda})$.\\
$\bullet$ If $\mathrm{tr}\,\nu_1>0$, suppose that $\bnu_1=(a_1i_1,...,a_ni_n)\in \mathcal{V}_{\nu_1}$ with $a_1\geqslant 1$. By definitions, we have $\theta_{\bnu_1}=\theta_{i_1}^{(a_1)}\theta_{\bnu}$, where $\bnu=(a_2i_2,...,a_ni_n)\in \mathcal{V}_{\nu_1-a_1i_1}$. By the inductive hypothesis, we have $\varphi'(\overline{\theta_{\bnu}\theta_{\lambda}\theta_{\bnu_2}})=\Psi\varphi'(\theta_{\bnu}\theta_{\lambda}\theta_{\bnu_2})$. By Corollary \ref{Im varphi'}, we have
\begin{align*}
\varphi'(\overline{\theta_{\bnu_1}\theta_{\lambda}\theta_{\bnu_2}})=&\varphi'(F_{i_1}^{(a_1)}\overline{\theta_{\bnu}\theta_{\lambda}\theta_{\bnu_2}})=F_{i_1}^{(a_1)}\varphi'(\overline{\theta_{\bnu}\theta_{\lambda}\theta_{\bnu_2}})=\overline{F_{i_1}^{(a_1)}}\Psi\varphi'(\theta_{\bnu}\theta_{\lambda}\theta_{\bnu_2})\\
=&\Psi(F_{i_1}^{(a_1)}\varphi'(\theta_{\bnu}\theta_{\lambda}\theta_{\bnu_2}))=\Psi\varphi'(F_{i_1}^{(a_1)}\theta_{\bnu}\theta_{\lambda}\theta_{\bnu_2})=\Psi\varphi'(\theta_{\bnu_1}\theta_{\lambda}\theta_{\bnu_2}),
\end{align*}
where we use $\Psi(um)=\overline{u}\Psi(m)$ for any $u\in \mathbf{U}$ and $m\in \Lambda_{\xi}\otimes \Lambda_{\lambda}$, as desired. It is clear that both $\varphi'(\overline{x})$ and $\Psi\varphi'(x)$ are $\mathbb{Q}$-linear in $x$, and swap $v,v^{-1}$. Hence $\varphi'(\overline{x})=\Psi\varphi'(x)$ for any $x\in \mathbf{f}\theta_{\lambda}\mathbf{f}$, and so $\varphi(\overline{m})=\Psi\varphi(m)$ for any $m\in \Lambda_{\xi,\lambda}$.
\end{proof}

By Proposition \ref{19.1.2}, there is a unique admissible (with respect to the $\tilde{\mathbf{U}}$-action) symmetric bilinear form $(,)_{\xi\odot\lambda}:\tilde{\Lambda}_{\xi\odot\lambda}\times \tilde{\Lambda}_{\xi\odot\lambda}\rightarrow \mathbb{Q}(v)$ such that $(\eta_{\xi\odot\lambda},\eta_{\xi\odot\lambda})_{\xi\odot\lambda}=1$.

\begin{lemma}\label{two pairings}
For any $\nu=\sum_{i\in I}\nu_ii\in \mathbb{N}[I]$ and homogeneous $y,y'\in \mathbf{f}_\nu$, we have 
$$(\pi_{\xi\odot\lambda}(\theta_{\lambda}y),\pi_{\xi\odot\lambda}(\theta_{\lambda}y'))_{\xi\odot\lambda}=(\varphi\pi_{\xi\odot\lambda}(\theta_{\lambda}y),\varphi\pi_{\xi\odot\lambda}(\theta_{\lambda}y'))_{\xi,\lambda}\prod_{i\in I}v^{-\langle i,\lambda\rangle \nu_i}\begin{bmatrix}\langle i,\lambda\rangle+\nu_i\\\langle i,\lambda\rangle\end{bmatrix}.$$ 
\end{lemma}
\begin{proof}
It is clear that the subspace $\mathbf{f}\subset \tilde{\mathbf{f}}$ is a $\mathbf{U}$-submodule of $\tilde{M}_{\xi\odot\lambda}$, see the proof of Lemma \ref{submodules}, and  the image of $\mathbf{f}\subset \tilde{M}_{\xi\odot\lambda}$ under the natural projection  $\pi_{\xi\odot\lambda}:\tilde{M}_{\xi\odot\lambda}\rightarrow \tilde{\Lambda}_{\xi\odot\lambda}$ can be identified with $\Lambda_{\xi}$ such that $\eta_{\xi\odot\lambda}$ is identified with $\eta_{\xi}$. By Proposition \ref{19.1.2} and the uniqueness of $(,)_{\xi}$, we have 
$(y^-\eta_{\xi\odot\lambda},{y'}^-\eta_{\xi\odot\lambda})_{\xi\odot\lambda}=(y^-\eta_{\xi},{y'}^-\eta_{\xi})_{\xi}$.
By definitions, we have $(y^-\eta_{\xi},{y'}^-\eta_{\xi})_{\xi}=(y^-\eta_{\xi}\otimes \eta_{\lambda},{y'}^-\eta_{\xi}\otimes \eta_{\lambda})_{\xi,\lambda}=(\varphi\pi_{\xi\odot\lambda}(\theta_{\lambda}y),\varphi\pi_{\xi\odot\lambda}(\theta_{\lambda}y'))_{\xi,\lambda}$,
and so 
$$(y^-\eta_{\xi\odot\lambda},{y'}^-\eta_{\xi\odot\lambda})_{\xi\odot\lambda}=(\varphi\pi_{\xi\odot\lambda}(\theta_{\lambda}y),\varphi\pi_{\xi\odot\lambda}(\theta_{\lambda}y'))_{\xi,\lambda}.$$
By definitions, we have
$$(\pi_{\xi\odot\lambda}(\theta_{\lambda}y),\pi_{\xi\odot\lambda}(\theta_{\lambda}y'))_{\xi\odot\lambda}=(\prod_{j\in I}F_{j'}^{(\langle j,\lambda\rangle)}y^-\eta_{\xi\odot\lambda},\prod_{j\in I}F_{j'}^{(\langle j,\lambda\rangle)}{y'}^-\eta_{\xi\odot\lambda})_{\xi\odot \lambda}.$$
For any $i\in I$, by definitions, we have $\rho(F_{i'}^{(\langle i,\lambda\rangle)})=v^{\langle i,\lambda\rangle^2}K_{-\langle i,\lambda\rangle i'}E_{i'}^{(\langle i,\lambda\rangle)}$. We denote $m=\prod_{j\in I,j\not=i}F_{j'}^{(\langle j,\lambda\rangle)}y^-\eta_{\xi\odot\lambda}$ and $m'=\prod_{j\in I,j\not=i}F_{j'}^{(\langle j,\lambda\rangle)}{y'}^-\eta_{\xi\odot\lambda}$. By definitions and Lemma \ref{3.4.2}, we have
\begin{align*}
E_{i'}^{(\langle i,\lambda\rangle)}\prod_{j\in I}F_{j'}^{(\langle j,\lambda\rangle)}{y'}^-\eta_{\xi\odot\lambda}=&\prod_{j\in I,j\not=i}F_{j'}^{(\langle j,\lambda\rangle)}E_{i'}^{(\langle i,\lambda\rangle)}F_{i'}^{(\langle i,\lambda\rangle)}{y'}^-\eta_{\xi\odot\lambda}\\
=&\prod_{j\in I,j\not=i}F_{j'}^{(\langle j,\lambda\rangle)}\sum_{k=0}^{\langle i,\lambda\rangle}\begin{bmatrix}\langle i',\xi\odot\lambda-\nu\rangle\\k\end{bmatrix}F_{i'}^{(\langle i,\lambda\rangle-k)}E_{i'}^{(\langle i,\lambda\rangle-k)}{y'}^-\eta_{\xi\odot\lambda}
\end{align*}
Notice that $E_{i'}{y'}^-\eta_{\xi\odot\lambda}={y'}^-E_{i'}\eta_{\xi\odot\lambda}=0$, and so $E_{i'}^{(\langle i,\lambda\rangle)}\pi_{\xi\odot\lambda}(\theta_{\lambda}y')=\begin{bmatrix}\langle i,\lambda\rangle+\nu_i\\\langle i,\lambda\rangle\end{bmatrix}m'$. By definitions, we have
\begin{align*}
K_{-\langle i,\lambda\rangle i'}m'=&K_{-\langle i,\lambda\rangle i'}\prod_{j\in I,j\not=i}F_{j'}^{(\langle j,\lambda\rangle)}{y'}^-\eta_{\xi\odot\lambda}=\prod_{j\in I,j\not=i}F_{j'}^{(\langle j,\lambda\rangle)}K_{-\langle i,\lambda\rangle i'}{y'}^-\eta_{\xi\odot\lambda}\\
=&\prod_{j\in I,j\not=i}F_{j'}^{(\langle j,\lambda\rangle)}v^{-\langle i,\lambda\rangle\langle i',\xi\odot\lambda-\nu\rangle}{y'}^-\eta_{\xi\odot\lambda}=v^{-\langle i,\lambda\rangle (\langle i,\lambda\rangle+\nu_i)}m'.
\end{align*}
Hence $(\prod_{j\in I}F_{j'}^{(\langle j,\lambda\rangle)}y^-\eta_{\xi\odot\lambda},\prod_{j\in I}F_{j'}^{(\langle j,\lambda\rangle)}{y'}^-\eta_{\xi\odot\lambda})_{\xi\odot \lambda}$ is equal to
\begin{align*}
(F_{i'}^{(\langle i,\lambda\rangle)}m,\prod_{j\in I}F_{j'}^{(\langle j,\lambda\rangle)}{y'}^-\eta_{\xi\odot\lambda})_{\xi\odot\lambda}
=&(m,v^{\langle i,\lambda\rangle^2}K_{-\langle i,\lambda\rangle i'}E_{i'}^{(\langle i,\lambda\rangle)}\pi_{\xi\odot\lambda}(\theta_{\lambda}y'))_{\xi\odot\lambda}\\
=&v^{\langle i,\lambda\rangle^2}v^{-\langle i,\lambda\rangle (\langle i,\lambda\rangle+\nu_i)}\begin{bmatrix}\langle i,\lambda\rangle+\nu_i\\\langle i,\lambda\rangle\end{bmatrix}(m,m')_{\xi\odot\lambda}.
\end{align*}
By repeating above calculation, we obtain 
\begin{align*}
(\pi_{\xi\odot\lambda}(\theta_{\lambda}y),\pi_{\xi\odot\lambda}(\theta_{\lambda}y'))_{\xi\odot\lambda}=&(y^-\eta_{\xi\odot\lambda},{y'}^-\eta_{\xi\odot\lambda})_{\xi\odot\lambda}\prod_{i\in I}v^{-\langle i,\lambda\rangle \nu_i}\begin{bmatrix}\langle i,\lambda\rangle+\nu_i\\\langle i,\lambda\rangle\end{bmatrix}\\
=&(\varphi\pi_{\xi\odot\lambda}(\theta_{\lambda}y),\varphi\pi_{\xi\odot\lambda}(\theta_{\lambda}y'))_{\xi,\lambda}\prod_{i\in I}v^{-\langle i,\lambda\rangle \nu_i}\begin{bmatrix}\langle i,\lambda\rangle+\nu_i\\\langle i,\lambda\rangle\end{bmatrix},
\end{align*}
as desired.
\end{proof}

\section{Framed construction of the canonical basis of tensor product}\label{Framed construction of the canonical basis of tensor product}

\subsection{Framed construction of $\mathbf{B}(\Lambda_{\xi}\otimes \Lambda_{\lambda})$}

Let $\tilde{\mathbf{B}}$ be the canonical basis of $\tilde{\mathbf{f}}$ defined in \cite[Chapter 14]{Lusztig-1993}.

\begin{definition}\label{part of basis}
For any $b\in \tilde{\mathbf{B}}$ and $x\in \tilde{\mathbf{f}}$, suppose that $x=\sum_{b'\in \tilde{\mathbf{B}}}c_{b'}b'$. Then we say $b$ appears in $x$ if $c_b\not=0$. We define the subset $\tilde{\mathbf{B}}(\mathbf{f}\theta_{\lambda}\mathbf{f})\subset \tilde{\mathbf{B}}$ consisting of $b\in \tilde{\mathbf{B}}$ which appear in $\theta_{\bnu_1}\theta_{\lambda}\theta_{\bnu_2}$ for some $\bnu_1\in \mathcal{V}_{\nu_1},\bnu_2\in \mathcal{V}_{\nu_2}$ and $\nu_1,\nu_2\in \mathbb{N}[I]$.
\end{definition}

For any $i\in I$ and $n\in \mathbb{N}$, by Theorem \ref{14.3.2}, there are bijections
$\pi_{i,n}:\tilde{\mathbf{B}}_{i,0}\rightarrow \tilde{\mathbf{B}}_{i,n}$ and $\pi_{i,n}^{\sigma}:\tilde{\mathbf{B}}_{i,0}^{\sigma}\rightarrow \tilde{\mathbf{B}}_{i,n}^{\sigma}$ associated to the canonical basis $\tilde{\mathbf{B}}$ of $\tilde{\mathbf{f}}$.

\begin{lemma}\label{restricted bijection}
The bijections
$\pi_{i,n}:\tilde{\mathbf{B}}_{i,0}\rightarrow \tilde{\mathbf{B}}_{i,n}$ and $\pi_{i,n}^{\sigma}:\tilde{\mathbf{B}}_{i,0}^{\sigma}\rightarrow \tilde{\mathbf{B}}_{i,n}^{\sigma}$ can be restricted to be the bijections
\begin{align*}
\pi_{i,n}:\tilde{\mathbf{B}}_{i,0}\cap \tilde{\mathbf{B}}(\mathbf{f}\theta_{\lambda}\mathbf{f})\rightarrow \tilde{\mathbf{B}}_{i,n}\cap \tilde{\mathbf{B}}(\mathbf{f}\theta_{\lambda}\mathbf{f}),\ \pi_{i,n}^{\sigma}:\tilde{\mathbf{B}}_{i,0}^{\sigma}\cap \tilde{\mathbf{B}}(\mathbf{f}\theta_{\lambda}\mathbf{f})\rightarrow \tilde{\mathbf{B}}_{i,n}^{\sigma}\cap \tilde{\mathbf{B}}(\mathbf{f}\theta_{\lambda}\mathbf{f}).
\end{align*}
Moreover, for any $b\in \tilde{\mathbf{B}}_{i,0}\cap \tilde{\mathbf{B}}(\mathbf{f}\theta_{\lambda}\mathbf{f})$, we have 
\begin{align*}
&\theta_i^{(n)}b=\pi_{i,n}(b)+\sum_{b'\in \tilde{\mathbf{B}}_{i,\geqslant n+1}\cap\tilde{\mathbf{B}}(\mathbf{f}\theta_{\lambda}\mathbf{f})}c_{b'}b',\\
&(p_{n i}\otimes \mathrm{Id})r\pi_{i,n}(b)=\theta_i^{(n)}\otimes b+\sum_{b''\in \tilde{\mathbf{B}}_{i,\geqslant 1}\cap\tilde{\mathbf{B}}(\mathbf{f}\theta_{\lambda}\mathbf{f})}d_{b''}\theta_i^{(n)}\otimes b'',
\end{align*}
where $c_{b'},d_{b''}\in \mathbb{N}[v,v^{-1}]$; for any $b\in \tilde{\mathbf{B}}_{i,0}^{\sigma}\cap \tilde{\mathbf{B}}(\mathbf{f}\theta_{\lambda}\mathbf{f})$, we have 
\begin{align*}
&b\theta_i^{(n)}=\pi_{i,n}^{\sigma}(b)+\sum_{b'\in \tilde{\mathbf{B}}_{i,\geqslant n+1}^{\sigma}\cap \tilde{\mathbf{B}}(\mathbf{f}\theta_{\lambda}\mathbf{f})}c'_{b'}b',\\ 
&(\mathrm{Id}\otimes p_{ni})r\pi_{i,n}^{\sigma}(b)=b\otimes \theta_i^{(n)}+\sum_{b''\in \tilde{\mathbf{B}}_{i,\geqslant 1}^{\sigma}\cap \tilde{\mathbf{B}}(\mathbf{f}\theta_{\lambda}\mathbf{f})}d'_{b''}b''\otimes \theta_i^{(n)}
\end{align*}
where $c_{b'},d_{b''}\in \mathbb{N}[v,v^{-1}]$.
\end{lemma}
\begin{proof}
For any $b\in \tilde{\mathbf{B}}_{i,0}$, by Theorem \ref{14.3.2}, we have 
\begin{align*}
\theta_i^{(n)}b=\pi_{i,n}(b)+\sum_{b'\in \tilde{\mathbf{B}}_{i,\geqslant n+1}}c_{b'}b',\
(p_{ni}\otimes \mathrm{Id})r\pi_{i,n}(b)=\theta_i^{(n)}\otimes b+\sum_{b''\in \tilde{\mathbf{B}}_{i,\geqslant 1}}d_{b''}\theta_i^{(n)}\otimes b'',
\end{align*}
where $c_{b'},d_{b''}\in \mathbb{N}[v,v^{-1}]$. If $b\in \tilde{\mathbf{B}}_{i,0}\cap \tilde{\mathbf{B}}(\mathbf{f}\theta_{\lambda}\mathbf{f})$, suppose that $b$ appears in $\theta_{\bnu_1}\theta_{\lambda}\theta_{\bnu_2}$ for some $\bnu_1\in \mathcal{V}_{\nu_1},\bnu_2\in \mathcal{V}_{\nu_2}$ and $\nu_1,\nu_2\in \mathbb{N}[I]$, that is, if we write $\theta_{\bnu_1}\theta_{\lambda}\theta_{\bnu_2}=\sum_{b_0\in \tilde{\mathbf{B}}}k_{b_0}b_0$, then $k_b\not=0$. Thus we have
$$\theta_i^{(n)}\theta_{\bnu_1}\theta_{\lambda}\theta_{\bnu_2}=\theta_i^{(n)}(k_bb+\sum_{b_0\in \tilde{\mathbf{B}},b_0\not=b}k_{b_0}b_0)=k_b\pi_{i,n}(b)+\sum_{b'\in \tilde{\mathbf{B}}_{i,\geqslant n+1}}k_bc_{b'}b'+\sum_{b_0\in \tilde{\mathbf{B}},b_0\not=b}k_{b_0}\theta_i^{(n)}b_0.$$
Since the monomial $\theta_{\bnu_1}\theta_{\lambda}\theta_{\bnu_2}$ is a product of divided powers which are canonical basis elements, by the positivity of $\tilde{\mathbf{B}}$ with respect to the multiplication, we have $k_b,k_{b_0}\in \mathbb{N}[v,v^{-1}]$, and $\theta_i^{(n)}b_0$ is a $\mathbb{N}[v,v^{-1}]$-linear combination of canonical basis elements in $\tilde{\mathbf{B}}$. Hence $\pi_{i,n}(b)$ and those $b'\in \tilde{\mathbf{B}}_{i,\geqslant n+1}$ such that $c_{b'}\not=0$, appear in $\theta_{(ni,\bnu_1)}\theta_{\lambda}\theta_{\bnu_2}$, where $(ni,\bnu_1)\in \mathcal{V}_{ni+\nu_1}$ is the sequence obtained by adding $ni$ before the sequence $\bnu_1$, and so we have $\pi_{i,n}(b)\in \tilde{\mathbf{B}}_{i,n}\cap \tilde{\mathbf{B}}(\mathbf{f}\theta_{\lambda}\mathbf{f})$ and $b'\in \tilde{\mathbf{B}}_{i,\geqslant n+1}\cap\tilde{\mathbf{B}}(\mathbf{f}\theta_{\lambda}\mathbf{f})$. Conversely, if $b\in \tilde{\mathbf{B}}_{i,0}$ and $\pi_{i,n}(b)\in \tilde{\mathbf{B}}_{i,n}\cap \tilde{\mathbf{B}}(\mathbf{f}\theta_{\lambda}\mathbf{f})$, suppose that $\pi_{i,n}(b)$ appears in $\theta_{\bnu_1}\theta_{\lambda}\theta_{\bnu_2}$ for some $\bnu_1=(a_1i_1,...,a_ni_n)\in \mathcal{V}_{\nu_1},\bnu_2=(a'_1j_1,...,a'_mj_m)\in \mathcal{V}_{\nu_2}$ and $\nu_1,\nu_2\in \mathbb{N}[I]$. We identify $\theta_{\lambda}$ with $\theta_{(\langle i,\lambda\rangle i')_{i\in I}}$ for $(\langle i,\lambda\rangle)i')_{i\in I}\in \mathcal{V}_{|\theta_{\lambda}|}$ in arbitrary order. By \cite[Section 9.2.11]{Lusztig-1993}, we have 
\begin{equation}\label{pniIdr}
(p_{ni}\otimes \mathrm{Id})r(\theta_{\bnu_1}\theta_{\lambda}\theta_{\bnu_2})=\sum v^{d(\bnu,\bnu')}\theta_{\bnu}\otimes \theta_{\bnu'},
\end{equation}
where the sum is taken over $\bnu=(x_1,...,x_{n+m+|I|})\in \mathcal{V}_{ni},\bnu'=(y_1,...,y_{n+m+|I|})\in \mathcal{V}_{\nu_1+|\theta_{\lambda}|+\nu_2-ni}$ such that $$(x_1+y_1,...,x_{n+m+|I|}+y_{n+m+|I|})=(a_1i_1,...,a_ni_n,(\langle i,\lambda\rangle)i')_{i\in I},a'_1j_1,...,a'_mj_m),$$
and $d(\bnu,\bnu')\in \mathbb{Z}$. Notice that such $\bnu'$ must have the form $(\bnu'_1,(\langle i,\lambda\rangle)i')_{i\in I},\bnu'_2)$, where $\bnu'_1\in \mathcal{V}_{\nu'_1},\bnu'_2\in \mathcal{V}_{\nu'_2}$ and $\nu'_1,\nu'_2\in \mathbb{N}[I]$ such that $\nu'_1+\nu'_2=\nu_1+\nu_2-ni$, thus $\theta_{\bnu'}=\theta_{\bnu'_1}\theta_{\lambda}\theta_{\bnu'_2}$. Notice that $\theta_{\bnu}=c_{\bnu}\theta_i^{(n)}$ for some $c_{\bnu}\in \mathbb{N}[v,v^{-1}]$, so $b$ and those $b''\in \tilde{\mathbf{B}}_{i,\geqslant 1}$ such that $d_{b''}\not=0$ appear in $\sum c_{\bnu}\theta_{\bnu'_1}\theta_{\lambda}\theta_{\bnu'_2}$, and moreover, they appear in some $\theta_{\bnu'_1}\theta_{\lambda}\theta_{\bnu'_2}$, and so we have  $b\in \tilde{\mathbf{B}}_{i,0}\cap \tilde{\mathbf{B}}(\mathbf{f}\theta_{\lambda}\mathbf{f})$ and $b''\in \tilde{\mathbf{B}}_{i,\geqslant 1}\cap\tilde{\mathbf{B}}(\mathbf{f}\theta_{\lambda}\mathbf{f})$. Hence $\pi_{i,n}$ can be restricted to be the bijection from $\tilde{\mathbf{B}}_{i,0}\cap \tilde{\mathbf{B}}(\mathbf{f}\theta_{\lambda}\mathbf{f})$ to $\tilde{\mathbf{B}}_{i,n}\cap \tilde{\mathbf{B}}(\mathbf{f}\theta_{\lambda}\mathbf{f})$. The other statements can be proved dually.
\end{proof}

\begin{proposition}\label{B(fthetaf)}
We have $\tilde{\mathbf{B}}(\mathbf{f}\theta_{\lambda}\mathbf{f})\subset \mathbf{f}\theta_{\lambda}\mathbf{f}$, and so $\tilde{\mathbf{B}}(\mathbf{f}\theta_{\lambda}\mathbf{f})$ is a basis of $\mathbf{f}\theta_{\lambda}\mathbf{f}$.
\end{proposition}
\begin{proof}
For any $b\in \tilde{\mathbf{B}}(\mathbf{f}\theta_{\lambda}\mathbf{f})$, suppose that $b$ appears in $\theta_{\bnu_1}\theta_{\lambda}\theta_{\bnu_2}$, where 
$$\bnu_1=(a_1i_1,...,a_ni_n)\in \mathcal{V}_{\nu_1},\bnu_2=(a'_1j_1,...,a'_mj_m)\in \mathcal{V}_{\nu_2},\nu_1,\nu_2\in \mathbb{N}[I].$$ 
We use induction on $\mathrm{tr}\,\nu_1+\mathrm{tr}\,\nu_2\in \mathbb{N}$ and descending inductions on 
$$t_{i_1}(b)\in [0,(\nu_1)_{i_1}+(\nu_2)_{i_1}],t_{j_m}^{\sigma}(b)\in [0,(\nu_1)_{j_m}+(\nu_2)_{j_m}]$$ 
to prove that $b\in \mathbf{f}\theta_{\lambda}\mathbf{f}$.\\
$\bullet$ If $\mathrm{tr}\,\nu_1+\mathrm{tr}\,\nu_2=0$, then $b$ appears in $\theta_{\lambda}$. We claim that $\theta_{\lambda}\in \tilde{\mathbf{B}}_{i',0}$. Indeed, it is clear that $1\in \tilde{\mathbf{B}}$. For any $i\in I$, by Theorem \ref{14.3.2}, we have $\theta_{i'}^{(\langle i,\lambda\rangle)}=\theta_{i'}^{(\langle i,\lambda\rangle)}1=\pi_{i',\langle i,\lambda\rangle}(1)\in \tilde{\mathbf{B}}$. Notice that we have $\theta_{i'}^{(\langle i,\lambda\rangle)}\in \bigcap_{j\in I,j\not=i}\tilde{\mathbf{B}}_{j',0}$, by repeating above process, we obtain $\theta_{\lambda}\in \tilde{\mathbf{B}}$. Hence $b=\theta_{\lambda}\in \mathbf{f}\theta_{\lambda}\mathbf{f}$.\\
$\bullet$ If $\mathrm{tr}\,\nu_1+\mathrm{tr}\,\nu_2>0$ and $\mathrm{tr}\,\nu_1>0$, then we have $a_1\geqslant 1$ and $\theta_{\bnu_1}\theta_{\lambda}\theta_{\bnu_2}=\theta_{i_1}^{(a_1)}\theta_{\bnu'_1}\theta_{\lambda}\theta_{\bnu_2}$, where $\bnu'_1=(a_2i_2,...,a_ni_n)\in \mathcal{V}_{\nu_1-a_1i_1}$. By Corollary \ref{t_i=s_i}, we have $t_{i_1}(b)\geqslant a_1$. By Theorem \ref{14.3.2} and Lemma \ref{restricted bijection}, there exists a unique $b_1\in \tilde{\mathbf{B}}_{i_1,0}\cap \tilde{\mathbf{B}}(\mathbf{f}\theta_{\lambda}\mathbf{f})$ such that
\begin{equation}\label{thetab1}
\theta_{i_1}^{(t_{i_1}(b))}b_1=b+\sum_{b'_1\in \tilde{\mathbf{B}}_{i_1,\geqslant t_{i_1}(b)+1}\cap \tilde{\mathbf{B}}(\mathbf{f}\theta_{\lambda}\mathbf{f})}c_{b'_1}b'_1.
\end{equation}
Notice that $\mathrm{tr}\,|b_1|=\mathrm{tr}\,|b|-t_{i_1}(b)$, by the inductive hypothesis, we have $b_1\in \mathbf{f}\theta_{\lambda}\mathbf{f}$, and moreover, $\theta_{i_1}^{(t_{i_1}(b))}b_1\in \mathbf{f}\theta_{\lambda}\mathbf{f}$.\\
$\bullet'$ If $t_{i_1}(b)=(\nu_1)_{i_1}+(\nu_2)_{i_1}$, then the right hand side of (\ref{thetab1}) is $b$, and so $b\in \mathbf{f}\theta_{\lambda}\mathbf{f}$.\\
$\bullet'$ If $t_{i_1}(b)<(\nu_1)_{i_1}+(\nu_2)_{i_1}$, by the inductive hypothesis, we have $b'_1\in \mathbf{f}\theta_{\lambda}\mathbf{f}$, and so $b\in \mathbf{f}\theta_{\lambda}\mathbf{f}$, as desired.\\
$\bullet$ If $\mathrm{tr}\,\nu_1+\mathrm{tr}\,\nu_2>0$ and $\mathrm{tr}\,\nu_2>0$, then we have $a'_m\geqslant 1$ and $\theta_{\bnu_1}\theta_{\lambda}\theta_{\bnu_2}\!=\!\theta_{\bnu_1}\theta_{\lambda}\theta_{\bnu'_2}\theta_{j_m}^{(a'_m)}$, where $\bnu'_2=(a'_1j_1,...,a'_{m-1}j_{m-1})\in \mathcal{V}_{\nu_2-a'_mj_m}$. By Corollary \ref{t_i=s_i}, we have $t_{j_m}^{\sigma}(b)\geqslant a'_m$. By Theorem \ref{14.3.2} and Lemma \ref{restricted bijection}, there exists a unique $b_2\in \tilde{\mathbf{B}}_{j_m,0}^{\sigma}\cap \tilde{\mathbf{B}}(\mathbf{f}\theta_{\lambda}\mathbf{f})$ such that 
\begin{equation}\label{b2theta}
b_2\theta_{j_m}^{(t_{j_m}^{\sigma}(b))}=b+\sum_{b'_2\in \tilde{\mathbf{B}}_{j_m,\geqslant t_{j_m}^{\sigma}(b)+1}^{\sigma}\cap \tilde{\mathbf{B}}(\mathbf{f}\theta_{\lambda}\mathbf{f})}c'_{b'_2}b'_2.
\end{equation}
Notice that $\mathrm{tr}\,|b_2|=\mathrm{tr}\,|b|-t_{j_m}^{\sigma}(b)$, by the inductive hypothesis, we have $b_2\in \mathbf{f}\theta_{\lambda}\mathbf{f}$, and moreover, $b_2\theta_{j_m}^{(t_{j_m}^{\sigma}(b))}\in \mathbf{f}\theta_{\lambda}\mathbf{f}$.\\
$\bullet'$ If $t_{j_m}^{\sigma}(b)=(\nu_1)_{j_m}+(\nu_2)_{j_m}$, then the right hand side of (\ref{b2theta}) is $b$, and so $b\in \mathbf{f}\theta_{\lambda}\mathbf{f}$. \\
$\bullet'$ If $t_{j_m}^{\sigma}(b)<(\nu_1)_{j_m}+(\nu_2)_{j_m}$, by the inductive hypothesis, we have $b'_2\in \mathbf{f}\theta_{\lambda}\mathbf{f}$, and so $b\in \mathbf{f}\theta_{\lambda}\mathbf{f}$, as desired. Hence we have $\tilde{\mathbf{B}}(\mathbf{f}\theta_{\lambda}\mathbf{f})\subset \mathbf{f}\theta_{\lambda}\mathbf{f}$. By definitions, $\tilde{\mathbf{B}}(\mathbf{f}\theta_{\lambda}\mathbf{f})$ consists of canonical basis elements of $\tilde{\mathbf{f}}$ which can span $\mathbf{f}\theta_{\lambda}\mathbf{f}$, and so it is a basis of $\mathbf{f}\theta_{\lambda}\mathbf{f}$.
\end{proof}

\begin{proof}[Proof of Lemma {\rm{\ref{cap+=cap}}}]
By Theorem \ref{14.3.2}, we know that $\bigcup_{i\in I}\tilde{\mathbf{B}}_{i,\geqslant \langle i,\xi\rangle+1}^{\sigma}$ is a basis of $\sum_{i\in I}\tilde{\mathbf{f}}\theta_i^{\langle i,\xi\rangle+1}$, and $\bigcup_{i\in I}(\tilde{\mathbf{B}}_{i,\geqslant \langle i,\xi\rangle+1}^{\sigma}\cup\tilde{\mathbf{B}}_{i',\geqslant \langle i,\lambda\rangle+1}^{\sigma})$ is a basis of $\sum_{i\in I}\tilde{\mathbf{f}}\theta_i^{\langle i,\xi\rangle+1}+\sum_{i\in I}\tilde{\mathbf{f}}\theta_{i'}^{\langle i,\lambda\rangle+1}$. Moreover, by Proposition \ref{B(fthetaf)}, we know $\tilde{\mathbf{B}}(\mathbf{f}\theta_{\lambda}\mathbf{f})\cap\bigcup_{i\in I}\tilde{\mathbf{B}}_{i,\geqslant \langle i,\xi\rangle+1}^{\sigma}$ is a basis of $\mathbf{f}\theta_{\lambda}\mathbf{f}\cap\sum_{i\in I}\tilde{\mathbf{f}}\theta_i^{\langle i,\xi\rangle+1}$, and $\tilde{\mathbf{B}}(\mathbf{f}\theta_{\lambda}\mathbf{f})\cap\bigcup_{i\in I}(\tilde{\mathbf{B}}_{i,\geqslant \langle i,\xi\rangle+1}^{\sigma}\cup\tilde{\mathbf{B}}_{i',\geqslant \langle i,\lambda\rangle+1}^{\sigma})$ is a basis of $\mathbf{f}\theta_{\lambda}\mathbf{f}\cap(\sum_{i\in I}\tilde{\mathbf{f}}\theta_i^{\langle i,\xi\rangle+1}+\sum_{i\in I}\tilde{\mathbf{f}}\theta_{i'}^{\langle i,\lambda\rangle+1})$. Notice that $\tilde{\mathbf{B}}(\mathbf{f}\theta_{\lambda}\mathbf{f})\cap\bigcup_{i\in I}\tilde{\mathbf{B}}_{i',\geqslant \langle i,\lambda\rangle+1}^{\sigma}=\varnothing$, and so $$\mathbf{f}\theta_{\lambda}\mathbf{f}\cap(\sum_{i\in I}\tilde{\mathbf{f}}\theta_i^{\langle i,\xi\rangle+1}+\sum_{i\in I}\tilde{\mathbf{f}}\theta_{i'}^{\langle i,\lambda\rangle+1})=\mathbf{f}\theta_{\lambda}\mathbf{f}\cap\sum_{i\in I}\tilde{\mathbf{f}}\theta_i^{\langle i,\xi\rangle+1},$$
as desired.
\end{proof}

\begin{definition}\label{basis definition}
We define the set
\begin{align*}
\tilde{\mathbf{B}}(\xi,\lambda)=\tilde{\mathbf{B}}(\mathbf{f}\theta_{\lambda}\mathbf{f})\cap \bigcap_{i\in I}\bigcup_{0\leqslant n\leqslant \langle i,\xi\rangle}\tilde{\mathbf{B}}_{i,n}^{\sigma}\subset \tilde{\mathbf{B}}(\mathbf{f}\theta_{\lambda}\mathbf{f}),
\end{align*}
and define the set 
\begin{align*}
\tilde{\mathbf{B}}(\Lambda_{\xi,\lambda})=\{\pi_{\xi\odot\lambda}(b)\mid b\in \tilde{\mathbf{B}}(\xi,\lambda)\}\subset \Lambda_{\xi,\lambda},
\end{align*}
where $\pi_{\xi\odot\lambda}:\mathbf{f}\theta_{\lambda}\mathbf{f}\rightarrow \Lambda_{\xi,\lambda}$ is the natural projection. In particular, we have the sets
$$\tilde{\mathbf{B}}(0,\lambda)\subset \tilde{\mathbf{B}}(\mathbf{f}\theta_{\lambda}\mathbf{f}),\ \tilde{\mathbf{B}}(\Lambda_{0,\lambda})\subset \Lambda_{0,\lambda}.$$
\end{definition}

Let $\tilde{\mathbf{B}}(\tilde{\Lambda}_{\xi\odot\lambda})=\{\pi_{\xi\odot\lambda}(b)\mid b\in \tilde{\mathbf{B}}(\xi\odot\lambda)\}$ be the canonical basis of $\tilde{\Lambda}_{\xi\odot\lambda}$, see \ref{Canonical basis of Lambdalambda}. By definitions, we have 
\begin{align*}
\tilde{\mathbf{B}}(\xi\odot\lambda)=&(\bigcap_{i\in I}\bigcup_{0\leqslant n\leqslant \langle i,\xi\odot\lambda\rangle}\tilde{\mathbf{B}}_{i,n}^{\sigma})\cap (\bigcap_{i\in I}\bigcup_{0\leqslant n\leqslant \langle i',\xi\odot\lambda\rangle}\tilde{\mathbf{B}}_{i',n}^{\sigma})\\
=&(\bigcap_{i\in I}\bigcup_{0\leqslant n\leqslant \langle i,\xi\rangle}\tilde{\mathbf{B}}_{i,n}^{\sigma})\cap (\bigcap_{i\in I}\bigcup_{0\leqslant n\leqslant \langle i,\lambda\rangle}\tilde{\mathbf{B}}_{i',n}^{\sigma}),
\end{align*}
Note that $\tilde{\mathbf{B}}(\mathbf{f}\theta_{\lambda}\mathbf{f})\subset \bigcap_{i\in I}\bigcup_{0\leqslant n\leqslant \langle i,\lambda\rangle}\tilde{\mathbf{B}}_{i',n}^{\sigma}$, and so 
$$\tilde{\mathbf{B}}(\xi,\lambda)\subset \tilde{\mathbf{B}}(\xi\odot\lambda),\ \tilde{\mathbf{B}}(\Lambda_{\xi,\lambda})\subset \tilde{\mathbf{B}}(\tilde{\Lambda}_{\xi\odot\lambda}).$$

\begin{corollary}\label{framed basis}
{\rm{(a)}} The set $\tilde{\mathbf{B}}(\Lambda_{\xi,\lambda})$ is a basis of $\Lambda_{\xi,\lambda}$. In particular, the set $\tilde{\mathbf{B}}(\Lambda_{0,\lambda})$ is a basis of $\Lambda_{0,\lambda}$.\\
{\rm{(b)}} For any $b\in \tilde{\mathbf{B}}(\mathbf{f}\theta_{\lambda}\mathbf{f})\setminus \tilde{\mathbf{B}}(\xi,\lambda)$, we have $\pi_{\xi\odot\lambda}(b)=0$.
\end{corollary}
\begin{proof}
By Theorem \ref{14.3.2}, $\bigcup_{i\in I}\tilde{\mathbf{B}}_{i,\geqslant \langle i,\xi\rangle+1}^{\sigma}$ is a basis of $\sum_{i\in I}\tilde{\mathbf{f}}\theta_i^{\langle i,\xi\rangle+1}$. Hence $\pi_{\xi\odot\lambda}(b)=0$ for any $b\in \tilde{\mathbf{B}}(\mathbf{f}\theta_{\lambda}\mathbf{f})\setminus \tilde{\mathbf{B}}(\xi,\lambda)$. By Proposition \ref{B(fthetaf)}, we know that $\tilde{\mathbf{B}}(\mathbf{f}\theta_{\lambda}\mathbf{f})\cap \bigcup_{i\in I}\tilde{\mathbf{B}}_{i,\geqslant \langle i,\xi\rangle+1}^{\sigma}$ is a basis of $\mathbf{f}\theta_{\lambda}\mathbf{f}\cap\sum_{i\in I}\tilde{\mathbf{f}}\theta_i^{\langle i,\xi\rangle+1}$. Hence $\tilde{\mathbf{B}}(\Lambda_{\xi,\lambda})$ is a basis of $\Lambda_{\xi,\lambda}$.
\end{proof}

\begin{proof}[Proof of Lemma {\rm{\ref{surjections are isomorphisms}}}]
Consider the following composition of linear maps
\begin{align*}
\chi':M_{\xi}\otimes \mathbf{f}\theta_{\lambda}\mathbf{f}=\mathbf{f}\otimes \mathbf{f}\theta_{\lambda}\mathbf{f}\xrightarrow{\tau}\mathbf{f}\theta_{\lambda}\mathbf{f}\otimes \mathbf{f}\xrightarrow{\mathrm{multiplication\ of}\ \tilde{\mathbf{f}}}\mathbf{f}\theta_{\lambda}\mathbf{f}\xrightarrow{\pi_{\xi\odot\lambda}}\Lambda_{\xi,\lambda},
\end{align*}
where $\tau$ is the isomorphism swapping two factors, and $\pi_{\xi\odot\lambda}$ is the natural projection. It is clear that $\chi'$ is surjective, and $\chi'((\sum_{i\in I}\mathbf{f}\theta_i^{\langle i,\xi\rangle+1})\otimes \mathbf{f}\theta_{\lambda}\mathbf{f})=0$, so $\chi'$ induces a well-defined surjective linear map
$\chi:\Lambda_{\xi}\otimes \mathbf{f}\theta_{\lambda}\mathbf{f}\rightarrow \Lambda_{\xi,\lambda}$, that is, $\chi(\Lambda_{\xi}\otimes \mathbf{f}\theta_{\lambda}\mathbf{f})=\Lambda_{\xi,\lambda}$. Let $S$ be the subspace of $\mathbf{f}\theta_{\lambda}\mathbf{f}$ spanned by $\tilde{\mathbf{B}}(0,\lambda)$. We claim that 
\begin{align}\label{chilambdaxiotimesS=}
\chi(\Lambda_{\xi}\otimes S)=\chi(\Lambda_{\xi}\otimes \mathbf{f}\theta_{\lambda}\mathbf{f})=\Lambda_{\xi,\lambda}.
\end{align}
It is clear that $\chi(\Lambda_{\xi}\otimes S)\subset \chi(\Lambda_{\xi}\otimes \mathbf{f}\theta_{\lambda}\mathbf{f})$. For any $b\in \tilde{\mathbf{B}}(\mathbf{f}\theta_{\lambda}\mathbf{f})\setminus \tilde{\mathbf{B}}(0,\lambda)$, we have $b\in \bigcup_{i\in I}\tilde{\mathbf{B}}_{i,\geqslant 1}^{\sigma}$. We use induction on $\mathrm{tr}\,|b|$ and a descending induction on $t_i^{\sigma}(b)\in [1,|b|_i]$ for any $i\in I$ to prove that $\chi(m\otimes b)\in \chi(\Lambda_{\xi}\otimes S)$ for any $m\in \Lambda_{\xi}$. Suppose that $b\in \tilde{\mathbf{B}}_{i,\geqslant 1}^{\sigma}$, and so $t_i^{\sigma}(b)\geqslant 1$. By Theorem \ref{14.3.2} and Lemma \ref{restricted bijection}, there exists a unique $b'\in \tilde{\mathbf{B}}_{i,0}^{\sigma}\cap \tilde{\mathbf{B}}(\mathbf{f}\theta_{\lambda}\mathbf{f})$ such that 
\begin{equation}\label{b'theta}
b'\theta_i^{(t_i^{\sigma}(b))}=b+\sum_{b''\in \tilde{\mathbf{B}}_{i,\geqslant t_i^{\sigma}(b)+1}^{\sigma}\cap \tilde{\mathbf{B}}(\mathbf{f}\theta_{\lambda}\mathbf{f})}c_{b''}b''.
\end{equation}
If $b'\in \tilde{\mathbf{B}}(0,\lambda)$, then 
$\chi(m\otimes b'\theta_i^{(t_i^{\sigma}(b))})=\chi(F_i^{(t_i^{\sigma}(b))}m\otimes b')\in \chi(\Lambda_{\xi}\otimes S)$; if $b'\in \tilde{\mathbf{B}}(\mathbf{f}\theta_{\lambda}\mathbf{f})\setminus \tilde{\mathbf{B}}(0,\lambda)$, notice that $\mathrm{tr}\,|b'|<\mathrm{tr}\,|b|$, then we have 
$\chi(m\otimes b'\theta_i^{(t_i^{\sigma}(b))})=\chi(F_i^{(t_i^{\sigma}(b))}m\otimes b')\in \chi(\Lambda_{\xi}\otimes S)$ by the inductive hypothesis. If $t_i^{\sigma}(b)=|b|_i$, then the right hand side of (\ref{b'theta}) is $b$, and so $\chi(m\otimes b)\in \chi(\Lambda_{\xi}\otimes S)$. If $t_i^{\sigma}(b)<|b|_i$, for those $b''\in \tilde{\mathbf{B}}(0,\lambda)$, we have $\chi(m\otimes b'')\in \chi(\Lambda_{\xi}\otimes S)$ by definitions; for those $b''\in \tilde{\mathbf{B}}(\mathbf{f}\theta_{\lambda}\mathbf{f})\setminus \tilde{\mathbf{B}}(0,\lambda)$, we have $\chi(m\otimes b'')\in \chi(\Lambda_{\xi}\otimes S)$ by the inductive hypothesis. Hence $\chi(m\otimes b)\in \chi(\Lambda_{\xi}\otimes S)$, as desired. This finishes the proof of (\ref{chilambdaxiotimesS=}). As a result, the surjective linear map $\chi:\Lambda_{\xi}\otimes \mathbf{f}\theta_{\lambda}\mathbf{f}\rightarrow \Lambda_{\xi,\lambda}$ can be restricted to be a surjective linear map $\chi:\Lambda_{\xi}\otimes S\rightarrow \Lambda_{\xi,\lambda}$. We regard $\mathbf{f}\theta_{\lambda}\mathbf{f}\subset \tilde{M}_{0\odot\lambda}$. Since $S$ is spanned by homogeneous elements, $S$ is compatible with the weight space decomposition of $\tilde{M}_{0\odot\lambda}$. For any $\mu\in Y$ and homogeneous $x\in \mathbf{f},x'\in S$, we have
\begin{align*}
\chi(K_{\mu}(x^-\eta_{\xi}\otimes x'))=&\chi(v^{\langle \mu,\xi-|x|\rangle+\langle \mu,0\odot\lambda-|x'|\rangle}x^-\eta_{\xi}\otimes x')=v^{\langle \mu,\xi+\lambda+|\theta_{\lambda}|-|x|-|x'|\rangle}\pi_{\xi\odot\lambda}(x'x)\\
=&v^{\langle \mu,\xi\odot\lambda-|x|-|x'|\rangle}\pi_{\xi\odot\lambda}(x'x)=K_{\mu}\pi_{\xi\odot\lambda}(x'x).
\end{align*}
Hence $\chi$ commutes with the actions of $K_{\mu}$ for any $\mu\in Y$. By Corollary \ref{framed basis}, $\Lambda_{0,\lambda}$ has a basis $\tilde{\mathbf{B}}(\Lambda_{0,\lambda})$, and so $\dim\,\Lambda_{0,\lambda}^{\lambda'}=\dim\,S^{\lambda'}$ for any $\lambda'\in X$, where $S^{\lambda'}=S\cap \tilde{M}_{0\odot\lambda}^{\lambda'}$. By Proposition \ref{psi}, we have
$$\mathrm{dim}\, \Lambda_{\xi}^{\xi'}\otimes \Lambda_{\lambda}^{\lambda'}=\mathrm{dim}\, \Lambda_{\xi}^{\xi'}\otimes \Lambda_{0,\lambda}^{\lambda'}=\mathrm{dim}\, \Lambda_{\xi}^{\xi'}\otimes S^{\lambda'}\geqslant \mathrm{dim}\, \Lambda_{\xi,\lambda}^{\xi'+\lambda'}$$
for any $\xi',\lambda'\in X$. Hence any surjective $\mathbf{U}^0$-module homomorphisms from $\Lambda_{\xi,\lambda}$ to $\Lambda_{\xi}\otimes \Lambda_{\lambda}$ are isomorphisms.
\end{proof}

\begin{lemma}\label{almost orthonormal}
For any $b\in \tilde{\mathbf{B}}(\xi,\lambda)$, we have 
$(\varphi\pi_{\xi\odot\lambda}(b),\varphi\pi_{\xi\odot\lambda}(b))_{\xi,\lambda}\in 1+v^{-1}\mathbf{A}$.
\end{lemma}
\begin{proof}
The proof will be given after Lemma \ref{(L,L)inA}.
\end{proof}

We remark that the isomorphism $\varphi:\Lambda_{\xi,\lambda}\rightarrow \Lambda_{\xi}\otimes \Lambda_{\lambda}$ does not preserve the bilinear forms $(,)_{\xi\odot\lambda}$ and $(,)_{\xi,\lambda}$, see Lemma \ref{two pairings}. As a result, we can not use the almost orthonormal property of the canonical basis $\tilde{\mathbf{B}}$ to prove Lemma \ref{almost orthonormal} directly. 

\begin{theorem}\label{framed construction of canonical basis}
The image of $\tilde{\mathbf{B}}(\Lambda_{\xi,\lambda})$ under the $\mathbf{U}$-module isomorphism $\varphi:\Lambda_{\xi,\lambda}\rightarrow \Lambda_{\xi}\otimes \Lambda_{\lambda}$ coincides with the canonical basis $\mathbf{B}(\Lambda_{\xi}\otimes \Lambda_{\lambda})$ of $\Lambda_{\xi}\otimes \Lambda_{\lambda}$, that is,
$$\mathbf{B}(\Lambda_{\xi}\otimes \Lambda_{\lambda})=\{\varphi\pi_{\xi\odot\lambda}(b)\mid b\in \tilde{\mathbf{B}}(\xi,\lambda)\}.$$
\end{theorem}
\begin{proof}
We use Lemma \ref{14.2.2} and Theorem \ref{unique canonical basis element}. Recall that $\mathcal{L}_{\xi,\lambda}$ is the $\mathbb{Z}[v^{-1}]$-submodule of $\Lambda_{\xi}\otimes \Lambda_{\lambda}$ generated by $\mathbf{B}(\Lambda_{\xi})\otimes \mathbf{B}(\Lambda_{\lambda})$, and so $\mathcal{A}\otimes_{\mathbb{Z}[v^{-1}]}\mathcal{L}_{\xi,\lambda}$ is the $\mathcal{A}$-submodule of $\Lambda_{\xi}\otimes \Lambda_{\lambda}$ generated by $\mathbf{B}(\Lambda_{\xi})\otimes \mathbf{B}(\Lambda_{\lambda})$.

For any $b\in \tilde{\mathbf{B}}(\xi,\lambda)\subset \tilde{\mathbf{B}}$, suppose that $r(b)=\sum_{b'_1,b'_2\in \tilde{\mathbf{B}}}c_{b'_1,b'_2} b'_1\otimes b'_2$, where $c_{b'_1,b'_2}\in \mathbb{N}[v,v^{-1}]$ by the positivity property of $\tilde{\mathbf{B}}$ with respect to the comultiplication. By the definition of $\varphi$, we have 
$$\varphi\pi_{\xi\odot\lambda}(b)=\sum c_{b'_1,b'_2}{b'_2}^-\eta_{\xi}\otimes \phi_{\lambda}^{-1}\pi_{0\odot\lambda}(b'_1),$$
where the summation is taken over $b'_1,b'_2\in \tilde{\mathbf{B}}$ such that $|b'_1|-|\theta_{\lambda}|\in \mathbb{N}[I]$. By the proof of Lemma \ref{imvarphi''=fLambda0lambda} and Corollary \ref{framed basis}, we have $\sum c_{b'_1,b'_2}{b'_2}^-\eta_{\xi}\otimes \pi_{0\odot\lambda}(b'_1)\in \Lambda_{\xi}\otimes \Lambda_{0,\lambda}$, where $\Lambda_{0,\lambda}=\pi_{0\odot\lambda}((\mathbf{f}\theta_{\lambda}\mathbf{f})_{0\odot\lambda})\subset \Lambda_{0\odot\lambda}$, and then ${b'_2}^-\eta_{\xi}\in \mathbf{B}(\Lambda_{\xi}),b'_1\in \mathbf{B}((\mathbf{f}\theta_{\lambda}\mathbf{f})_{0\odot\lambda})$. If $\pi_{0\odot\lambda}(b'_1)\not=0$, by definitions and \eqref{Fang-Lan-2025-5.3*}, we have $b'_1\in \tilde{\mathbf{B}}(0\odot\lambda)\subset \bigcap_{i\in I}\tilde{\mathbf{B}}_{i,0}^{\sigma}$, and there exists a unique $b''_1\in\mathbf{B}(\lambda)\subset \mathbf{B}$ such that $b'_1=b''_1\theta_{\lambda}$. Hence
\begin{align}\label{varphib integral}
\varphi\pi_{\xi\odot\lambda}(b)=\sum c_{b'_1,b'_2}{b'_2}^-\eta_{\xi}\otimes {b''_1}^-\eta_{\lambda}\in \mathbb{N}[v,v^{-1}][\mathbf{B}(\Lambda_{\xi})\otimes \mathbf{B}(\Lambda_{\lambda})]\subset \mathcal{A}\otimes_{\mathbb{Z}[v^{-1}]}\mathcal{L}_{\xi,\lambda}.   
\end{align}
By Lemma \ref{almost orthonormal}, $(\varphi\pi_{\xi\odot\lambda}(b),\varphi\pi_{\xi\odot\lambda}(b))_{\xi,\lambda}\in 1+v^{-1}\mathbf{A}$. By Lemma \ref{14.2.2}, there exists $b_1^-\eta_{\xi}\otimes b_2^-\eta_{\lambda}\in \mathbf{B}(\Lambda_{\xi})\otimes \mathbf{B}(\Lambda_{\lambda})$ such that 
$$\varphi\pi_{\xi\odot\lambda}(b)=\pm b_1^-\eta_{\xi}\otimes b_2^-\eta_{\lambda}\ \mathrm{mod}\, v^{-1}L(\Lambda_{\xi}\otimes \Lambda_{\lambda}),$$
where $L(\Lambda_{\xi}\otimes \Lambda_{\lambda})$ is the $\mathbf{A}$-submodule of $\Lambda_{\xi}\otimes \Lambda_{\lambda}$ generated by $\mathbf{B}(\Lambda_{\xi})\otimes \mathbf{B}(\Lambda_{\lambda})$. By \eqref{varphib integral} and $\mathcal{A}\cap\mathbf{A}=\mathbb{Z}[v^{-1}]$,
\begin{align*}
\varphi\pi_{\xi\odot\lambda}(b)=\pm b_1^-\eta_{\xi}\otimes b_2^-\eta_{\lambda}\ \mathrm{mod}\,v^{-1}\mathcal{L}_{\xi,\lambda}=b_1^-\eta_{\xi}\otimes b_2^-\eta_{\lambda}\ \mathrm{mod}\,v^{-1}\mathcal{L}_{\xi,\lambda},
\end{align*}
that is, $\varphi\pi_{\xi\odot\lambda}(b)-b_1^-\eta_{\xi}\otimes b_2^-\eta_{\lambda}\in v^{-1}\mathcal{L}_{\xi,\lambda}$. Since $b\in\tilde{\mathbf{B}}(\xi,\lambda) \subset \tilde{\mathbf{B}}(\xi\odot\lambda)$, we know that $\pi_{\xi\odot\lambda}(b)$ is bar-invariant with respect to the bar-involution of $\Lambda_{\xi,\lambda}$, see Theorem \ref{properties of B(Lambda)}. By Lemma \ref{varphi commutes with involutions}, we have
$$\Psi\varphi\pi_{\xi\odot\lambda}(b)=\varphi\pi_{\xi\odot\lambda}(b).$$
By the uniqueness of $b_1^-\eta_{\xi}\diamondsuit b_2^-\eta_{\lambda}$, we have  $\varphi\pi_{\xi\odot\lambda}(b)=b_1^-\eta_{\xi}\diamondsuit b_2^-\eta_{\lambda}$. Since $\varphi:\Lambda_{\xi,\lambda}\rightarrow \Lambda_{\xi}\otimes \Lambda_{\lambda}$ is an isomorphism, we conclude that the image of the basis $\tilde{\mathbf{B}}(\Lambda_{\xi,\lambda})$ of $\Lambda_{\xi,\lambda}$ is the canonical basis $\mathbf{B}(\Lambda_{\xi}\otimes \Lambda_{\lambda})$ of $\Lambda_{\xi}\otimes \Lambda_{\lambda}$.
\end{proof}

\begin{corollary}\label{positivity of transition matrix}
For any $b_1\in \mathbf{B}(\Lambda_{\xi})$ and $b_2\in \mathbf{B}(\Lambda_{\lambda})$, we have 
$$b_1\diamondsuit b_2\in b_1\otimes b_2+\sum v^{-1}\mathbb{N}[v^{-1}]b'_1\otimes b'_2,$$
where the sum is taken over $b'_1\otimes b'_2\in \mathbf{B}(\Lambda_{\xi})\otimes \mathbf{B}(\Lambda_{\lambda})\setminus \{b_1\otimes b_2\}$
\end{corollary}
\begin{proof}
By Theorem \ref{unique canonical basis element}, we have $b_1\diamondsuit b_2\in b_1\otimes b_2+\sum v^{-1}\mathbb{Z}[v^{-1}]b'_1\otimes b'_2$. Suppose that $b_1\diamondsuit b_2=\varphi\pi_{\xi\odot\lambda}(b)$, where $b\in \tilde{\mathbf{B}}(\xi,\lambda)$. By \eqref{varphib integral}, we have $\varphi\pi_{\xi\odot\lambda}(b)\in \sum\mathbb{N}[v,v^{-1}]b''_1\otimes b''_2$, where the sum is taken over $b''_1\otimes b''_2\in \mathbf{B}(\Lambda_{\xi})\otimes \mathbf{B}(\Lambda_{\lambda})$. Since $v^{-1}\mathbb{Z}[v^{-1}]\cap \mathbb{N}[v,v^{-1}]=v^{-1}\mathbb{N}[v^{-1}]$, we have $b_1\diamondsuit b_2\in b_1\otimes b_2+\sum v^{-1}\mathbb{N}[v^{-1}]b'_1\otimes b'_2$.
\end{proof}

\subsection{Example}\label{Example}

We present an example for Theorem \ref{framed construction of canonical basis}. Let $(I,\cdot)$ be the symmetric Cartan datum of type $A_1$ and $(Y,X,\langle,\rangle,...)$ be the simply connected root datum of type $(I,\cdot)$, that is, $I=\{i\}, i\cdot i=2$ and $Y=\mathbb{Z}[I],X=\mathrm{Hom}(Y,\mathbb{Z})$ with the natural pairing $\langle,\rangle:Y\times X\rightarrow \mathbb{Z}$ and natural embeddings $I\hookrightarrow Y,I\hookrightarrow X$. Then the canonical basis of $\mathbf{f}$ is 
$$\mathbf{B}=\{\theta_i^{(k)}\mid k\in \mathbb{N}\}.$$ 
For the dominant weight $n\in \mathbb{N}$, we have $\mathbf{B}(n)=\{\theta_i^{(k)}\mid 0\leqslant k\leqslant n\}$, and the canonical basis of $\Lambda_{n}$ is $$\mathbf{B}(\Lambda_n)=\{F_i^{(k)}\eta_n\mid 0\leqslant k\leqslant n\}.$$ 
For the dominant weights $m,n\in \mathbb{N}$, the canonical basis of $\Lambda_m\otimes \Lambda_n$ is 
\begin{align*}
\mathbf{B}(\Lambda_m\otimes \Lambda_n)=&\{\alpha_{k,l}\mid 0\leqslant k\leqslant m,0\leqslant l\leqslant n,k-l\leqslant m-n\}\\
\cup&\{\beta_{k,l}\mid 0\leqslant k\leqslant m,0\leqslant l\leqslant n,k-l\geqslant m-n\},
\end{align*}
where 
\begin{align*}
\alpha_{k,l}=\sum_{s=0}^{\mathrm{min}(k,l)}v^{s(k-m-s)}\begin{bmatrix}n-l+s\\s\end{bmatrix}F_i^{(m-k+s)}\eta_m\otimes F_i^{(l-s)}\eta_n,\\
\beta_{k,l}=\sum_{s=0}^{\mathrm{min}(k,l)}v^{s(l-n-s)}\begin{bmatrix}m-k+s\\s\end{bmatrix}F_i^{(m-k+s)}\eta_m\otimes F_i^{(l-s)}\eta_n
\end{align*}
with the identification $\alpha_{k,l}=\beta_{k,l}$ for $k-l=m-n$, see \cite[Section 6]{Lusztig-1992}.

The framed Cartan datum $(\tilde{I},\cdot)$ is of type $A_2$. The canonical basis of $\tilde{\mathbf{f}}$ is
$$\tilde{\mathbf{B}}=\{\theta_i^{(p)}\theta_{i'}^{(q)}\theta_i^{(r)}\mid p,q,r\in \mathbb{N},q\geqslant p+r\}\cup\{\theta_{i'}^{(r)}\theta_i^{(q)}\theta_{i'}^{(p)}\mid p,q,r\in \mathbb{N},q\geqslant p+r\}$$
with the identification $\theta_i^{(p)}\theta_{i'}^{(q)}\theta_i^{(r)}=\theta_{i'}^{(r)}\theta_i^{(q)}\theta_{i'}^{(p)}$ for $q=p+r$, see \cite[Section 14.5.4]{Lusztig-1993}. 

\begin{proposition}
We have 
\begin{equation}
\begin{aligned}
\tilde{\mathbf{B}}(m,n)=&\{\theta_{i'}^{(n-l)}\theta_i^{(m-k+l)}\theta_{i'}^{(l)}\mid 0\leqslant k\leqslant m,0\leqslant l\leqslant n,k-l\leqslant m-n\}\\
\cup&\{\theta_i^{(l)}\theta_{i'}^{(n)}\theta_i^{(m-k)}\mid 0\leqslant k\leqslant m,0\leqslant l\leqslant n,k-l\geqslant m-n\}
\end{aligned}
\end{equation}
with the identification $\theta_{i'}^{(n-l)}\theta_i^{(m-k+l)}\theta_{i'}^{(l)}=\theta_i^{(l)}\theta_{i'}^{(n)}\theta_i^{(m-k)}$ for $k-l=m-n$. Moreover, 
\begin{align*}
\varphi\pi_{m\odot n}(\theta_{i'}^{(n-l)}\theta_i^{(m-k+l)}\theta_{i'}^{(l)})=\alpha_{k,l},\ 
\varphi\pi_{m\odot n}(\theta_i^{(l)}\theta_{i'}^{(n)}\theta_i^{(m-k)})=\beta_{k,l}.
\end{align*}
\end{proposition}
\begin{proof}
We only need to prove $\theta_{i'}^{(n-l)}\theta_i^{(m-k+l)}\theta_{i'}^{(l)},\theta_i^{(l)}\theta_{i'}^{(n)}\theta_i^{(m-k)}\in \tilde{\mathbf{B}}(m,n)$ and calculate their images $\varphi\pi_{m\odot n}(\theta_{i'}^{(n-l)}\theta_i^{(m-k+l)}\theta_{i'}^{(l)}), \varphi\pi_{m\odot n}(\theta_i^{(l)}\theta_{i'}^{(n)}\theta_i^{(m-k)})$. 

For any $0\leqslant k\leqslant m,0\leqslant l\leqslant n,k-l\leqslant m-n$, by \cite[Lemma 42.1.2(d)]{Lusztig-1993}, we have 
\begin{align}\label{42.1.2}
\theta_i^{(l)}\theta_{i'}^{(n)}\theta_i^{(m-k)}=\sum_{s=m-n-k}^{m-n-k+l}\begin{bmatrix}m-n-k+l\\s\end{bmatrix}\theta_{i'}^{(m-k-s)}\theta_i^{(m-k+l)}\theta_{i'}^{(-m+n+k+s)}
\end{align}
such that $\theta_{i'}^{(n-l)}\theta_i^{(m-k+l)}\theta_{i'}^{(l)}$ appears in $\theta_i^{(l)}\theta_{i'}^{(n)}\theta_i^{(m-k)}$, and so $\theta_{i'}^{(n-k)}\theta_i^{(m-k+l)}\theta_{i'}^{(l)}\in \tilde{\mathbf{B}}(\mathbf{f}\theta_n\mathbf{f})$. We claim that $t_i^{\sigma}(\theta_{i'}^{(n-l)}\theta_i^{(m-k+l)}\theta_{i'}^{(l)})=m-k$. Indeed, let $t_i^{\sigma}(\theta_{i'}^{(n-l)}\theta_i^{(m-k+l)}\theta_{i'}^{(l)})=t$. By (\ref{42.1.2}) and Corollary \ref{t_i=s_i}, we have $t\geqslant m-k$. By Theorem \ref{14.3.2}, there exists a unique $b_1\in \tilde{\mathbf{B}}_{i,0}^{\sigma}$ such that 
\begin{equation}\label{btheta}
b_1\theta_i^{(t)}=\theta_{i'}^{(n-l)}\theta_i^{(m-k+l)}\theta_{i'}^{(l)}+\sum_{b'_1\in \tilde{\mathbf{B}}_{i,\geqslant t+1}^{\sigma}}c_{b'_1}b'_1.
\end{equation}
By the conditions $|b_1|=|\theta_{i'}^{(n-l)}\theta_i^{(m-k+l)}\theta_{i'}^{(l)}|-ti$ and $t_i^{\sigma}(b_1)=0$, we know that $b_1$ has the form $\theta_{i'}^{(n_1)}\theta_i^{(m-k+l-t)}\theta_{i'}^{(n_2)}$ satisfying $ n=n_1+n_2<m-k+l-t, n_2\not=0$, or is $\theta_i^{(m-k+l-t)}\theta_{i'}^{(n)}$. In the first case, we have $t<m-n-k+l\leqslant m-k$, a contradiction. In the second case, we have $b_1\theta_i^{(t)}=\theta_i^{(m-k+l-t)}\theta_{i'}^{(n)}\theta_i^{(t)}$. By \cite[Lemma 42.1.2(d)]{Lusztig-1993} and (\ref{btheta}), we have 
$$\theta_{i'}^{(n-l)}\theta_i^{(m-k+l)}\theta_{i'}^{(l)}+\sum_{b'_1\in \tilde{\mathbf{B}}_{i,\geqslant t+1}^{\sigma}}c_{b'_1}b'_1=\sum_{s=-n+t}^{m-n-k+l}\begin{bmatrix}m-n-k+l\\s\end{bmatrix}\theta_{i'}^{(t-s)}\theta_i^{(m-k+l)}\theta_{i'}^{(n-t+s)}.$$
Notice that $\theta_{i'}^{(t-s)}\theta_i^{(m-k+l)}\theta_{i'}^{(n-t+s)}\in \tilde{\mathbf{B}}$. Comparing the coefficient of $\theta_{i'}^{(n-l)}\theta_i^{(m-k+l)}\theta_{i'}^{(l)}$, we obtain $t=m-k$ or $t=n-l\leqslant m-k$. Hence we have $t=m-k\leqslant m$, and so $\theta_{i'}^{(n-l)}\theta_i^{(m-k+l)}\theta_{i'}^{(l)}\in\tilde{\mathbf{B}}(m,n)$. Moreover, we have 
\begin{align*}
\tau(p\otimes\mathrm{Id})r(\theta_{i'}^{(n-l)}\theta_i^{(m-k+l)}\theta_{i'}^{(l)})=\sum_{s=-m+k}^lv^{s(k-m-s)}\theta_i^{(m-k+s)}\otimes \theta_{i'}^{(n-l)}\theta_{i}^{(l-s)}\theta_{i'}^{(l)}.
\end{align*}
For any $-m+k\leqslant s<0$, by similarly argument of $t_i^{\sigma}(\theta_{i'}^{(n-l)}\theta_i^{(m-k+l)}\theta_{i'}^{(l)})=m-k$, we have $t_i^{\sigma}(\theta_{i'}^{(n-l)}\theta_{i}^{(l-s)}\theta_{i'}^{(l)})=-s>0$, and so $\pi_{0\odot n}(\theta_{i'}^{(n-l)}\theta_{i}^{(l-s)}\theta_{i'}^{(l)})=0\in \Lambda_{0,\lambda}$. For any $0\leqslant s\leqslant l$, by \cite[Lemma 42.1.2(d)]{Lusztig-1993}, we have 
$$\theta_{i'}^{(n-l)}\theta_{i}^{(l-s)}\theta_{i'}^{(l)}=\sum_{t=0}^{n-l+s}\begin{bmatrix}n-l+s\\t\end{bmatrix}\theta_i^{(m-k+s)}\otimes \theta_i^{(l-t)}\theta_{i'}^{(n)}\theta_{i}^{(-s+t)},$$
and so 
\begin{align*}
\varphi\pi_{m\odot n}(\theta_{i'}^{(n-l)}\theta_i^{(m-k+l)}\theta_{i'}^{(l)})=&\sum_{s=0}^lv^{s(k-m-s)}\begin{bmatrix}n-l+s\\s\end{bmatrix}\eta_{m}\otimes F_i^{(l-s)}\eta_n\\
=&\sum_{s=0}^{\mathrm{min}(k,l)}v^{s(k-m-s)}\begin{bmatrix}n-l+s\\s\end{bmatrix}\eta_{m}\otimes F_i^{(l-s)}\eta_n=\alpha_{k,l},
\end{align*}
where the condition $s\leqslant k$ is given by $F_i^{(m-k+s)}\eta_m\not=0$, as desired.

For any $0\leqslant k\leqslant m,0\leqslant l\leqslant n,k-l\geqslant m-n$, it is clear that $\theta_i^{(l)}\theta_{i'}^{(n)}\theta_i^{(m-k)}\in \tilde{\mathbf{B}}(\mathbf{f}\theta_n\mathbf{f})$. We claim that $t_i^{\sigma}(\theta_i^{(l)}\theta_{i'}^{(n)}\theta_i^{(m-k)})=m-k$. Indeed, let $t_i^{\sigma}(\theta_i^{(l)}\theta_{i'}^{(n)}\theta_i^{(m-k)})=t$, then it is clear that $t\geqslant m-k$. By Theorem \ref{14.3.2}, there exists a unique $b_2\in \tilde{\mathbf{B}}_{i,0}^{\sigma}$ such that 
\begin{equation}\label{b2thetait}
b_2\theta_i^{(t)}=\theta_i^{(l)}\theta_{i'}^{(n)}\theta_i^{(m-k)}+\sum_{b'_2\in \tilde{\mathbf{B}}_{i,\geqslant t+1}^{\sigma}}c_{b'_2}b'_2.
\end{equation}
By the conditions $|b_2|=|\theta_i^{(l)}\theta_{i'}^{(n)}\theta_i^{(m-k)}|-ti$ and $t_i^{\sigma}(b_2)=0$, we know that $b$ is of the form $\theta_{i'}^{(n_1)}\theta_i^{(m-k+l-t)}\theta_{i'}^{(n_2)}$ satisfying $ n=n_1+n_2<m-k+l-t, n_2\not=0$, or is $\theta_i^{(m-k+l-t)}\theta_{i'}^{(n)}$. In the first case, we have $t<m-n-k+l\leqslant m-k$, a contradiction. In the second case, notice that $b\theta_i^{(t)}=\theta_i^{(m-k+l-t)}\theta_{i'}^{(n)}\theta_i^{(t)}\in \tilde{\mathbf{B}}$, and the right hand side of (\ref{btheta}) is $\theta_i^{(l)}\theta_{i'}^{(n)}\theta_i^{(m-k)}$. Hence we have $t=m-k\leqslant m$, and so $\theta_i^{(l)}\theta_{i'}^{(n)}\theta_i^{(m-k)}\in \tilde{\mathbf{B}}(m,n)$. Moreover, by Theorem \ref{framed construction of tensor}, we have 
\begin{align*}
\varphi\pi_{m\odot n}(\theta_i^{(l)}\theta_{i'}^{(n)}\theta_i^{(m-k)})=&\sum_{s=0}^lv^{s(l-s)-sn}F_i^{(s)}F_i^{(m-k)}\eta_m\otimes F_i^{(l-s)}\eta_n\\
=&\sum_{s=0}^{\mathrm{min}(k,l)}v^{s(l-n-s)}\begin{bmatrix}m-k+s\\s\end{bmatrix}\eta_m\otimes F_i^{(l-s)}\eta_n=\beta_{k,l},
\end{align*}
where the condition $s\leqslant k$ is given by $F_i^{(m-k+s)}\eta_m\not=0$, as desired. 
\end{proof}

\subsection{Positivity}\label{Positivity}

For any $i\in I$, let ${_ir}:\tilde{\mathbf{f}}\rightarrow \tilde{\mathbf{f}} ,r_i:\tilde{\mathbf{f}}\rightarrow \tilde{\mathbf{f}}$ be the linear maps defined in \cite[Section 1.2.13]{Lusztig-1993} associated to the algebra $\mathbf{\tilde{f}}$. They are the unique linear maps such that 
\begin{equation}\label{Lusztig-1.2.13}
\begin{aligned}
&{_ir}(1)=0,\ {_ir}(\theta_j)=\delta_{i,j},\ {_ir}(xy)={_ir}(x)y+v^{|x|\cdot i}x{_ir}(y),\\
&r_i(1)=0,\ r_i(\theta_j)=\delta_{i,j},\ r_i(xy)=v^{|y|\cdot i}r_i(x)y+xr_i(y)
\end{aligned}
\end{equation}
for any $j\in \tilde{I}$ and homogeneous $x,y\in \mathbf{\tilde{f}}$. Moreover, we have
\begin{equation}\label{randriir}
\begin{aligned}
r(x)=\theta_i\otimes {_ir}(x)+\sum x_1\otimes x_2=r_i(x)\otimes \theta_i+\sum x'_1\otimes x'_2,
\end{aligned}
\end{equation}
where the summations are finite, and $x_1,x_2,x'_1,x'_2\in \mathbf{f}$ are homogeneous such that $|x_1|\not=i$ and $|x'_2|\not=i$.

\begin{lemma}\label{Ei formula}
For any $i\in I$ and homogeneous $x\in \mathbf{f}\theta_{\lambda}\mathbf{f}\subset \tilde{M}_{\xi\odot \lambda}$, we have 
\begin{align*}
E_ix=&\frac{v^{\langle i,\xi\odot \lambda-|x|\rangle+2}{_ir}(x)-v^{-\langle i,\xi\odot \lambda\rangle}r_i(x)}{v-v^{-1}}\\
=&\frac{v^{\langle i,\xi-|x|\rangle+2}{_ir(x)}-v^{-\langle i,\xi-|x|\rangle-2}\overline{{_ir}(\overline{x})}}{v-v^{-1}}.
\end{align*}
\end{lemma}
\begin{proof}
Since $E_ix=E_ix^-1$ and $E_i1=0$ in $\tilde{M}_{\xi\odot \lambda}$, by \cite[Proposition 3.1.6]{Lusztig-1993}, we have 
\begin{align*}
E_ix=\frac{K_i({_ir(x)})^-1-(r_i(x))^-K_{-i}1}{v-v^{-1}}=\frac{v^{\langle i,\xi\odot \lambda-|x|\rangle+2}{_ir}(x)-v^{-\langle i,\xi\odot \lambda\rangle}r_i(x)}{v-v^{-1}}.
\end{align*}
The second formula follows from \cite[Lemma 1.2.14]{Lusztig-1993} and $\langle i,\xi\odot\lambda\rangle=\langle i,\xi\rangle$.
\end{proof}

\begin{lemma}\label{nonzero term power}
For any $i\in I$ and $b\in \tilde{\mathbf{B}}(\mathbf{f}\theta_{\lambda}\mathbf{f})$, suppose that ${_ir}(b)=\sum_{b'\in \tilde{\mathbf{B}}}\sum_{n\in \mathbb{Z}}d_{b',n}^bv^nb'$ for any $b'\in \tilde{\mathbf{B}}$ and $n\in \mathbb{Z}$, we have $d_{b',n}^b\in \mathbb{N}$. Moreover, for any $b'\in \tilde{\mathbf{B}}(\xi,\lambda)$ and $n\in \mathbb{Z}$ such that $d_{b',n}^b\not=0$, we have 
$$\langle i,\xi-|b|\rangle+2+n\geqslant 0.$$
\end{lemma}
\begin{proof}
Suppose $r(b)=\sum_{b_1,b_2\in \tilde{\mathbf{B}}}c_{b_1,b_2}^bb_1\otimes b_2$. By the positivity of the canonical basis $\tilde{\mathbf{B}}$ with respect to the comultiplication of $\tilde{\mathbf{f}}$, see Theorem \ref{properties of B}, we have $c_{b_1,b_2}^bb_1\in \mathbb{N}[v,v^{-1}]$. By (\ref{randriir}), we have $\sum_{n\in \mathbb{Z}}d_{b',n}^bv^n=c_{\theta_i,b'}^b$, and so $d_{b',n}^b\in \mathbb{N}$. The second statement follows from $\tilde{\mathbf{B}}(\xi,\lambda)\subset \tilde{\mathbf{B}}(\xi\odot\lambda)$ and \cite[Corollary 22.1.5]{Lusztig-1993}.
\end{proof}

\begin{theorem}\label{positive}
The basis $\tilde{\mathbf{B}}(\Lambda_{\xi,\lambda})$ of $\Lambda_{\xi,\lambda}$ and the canonical basis $\mathbf{B}(\Lambda_{\xi}\otimes \Lambda_{\lambda})$ of $\Lambda_{\xi}\otimes \Lambda_{\lambda}$ have the positivity, that is, for any $i\in I$, we have
\begin{align*}
&E_ib, F_ib\in\sum_{b'\in \tilde{\mathbf{B}}(\Lambda_{\xi,\lambda})}\mathbb{N}[v,v^{-1}]b'\ \textrm{for any $b\in \tilde{\mathbf{B}}(\Lambda_{\xi,\lambda})$},\\ 
&E_ib, F_ib\in\sum_{b'\in \mathbf{B}(\Lambda_{\xi}\otimes \Lambda_{\lambda})}\mathbb{N}[v,v^{-1}]b'\ \textrm{for any $b\in \mathbf{B}(\Lambda_{\xi}\otimes \Lambda_{\lambda})$}.
\end{align*}
\end{theorem}
\begin{proof}
Recall that $\tilde{\mathbf{B}}(\Lambda_{\xi,\lambda})=\{\pi_{\xi\odot\lambda}(b)\mid b\in \tilde{\mathbf{B}}(\xi,\lambda)\}$, see Definition \ref{basis definition}.
For any $b\in \tilde{\mathbf{B}}(\xi,\lambda)$, by the positivity of the canonical basis $\tilde{\mathbf{B}}$ with respect to the multiplication of $\tilde{\mathbf{f}}$, see Theorem \ref{properties of B}, we have $\theta_ib\in \sum_{b'\in \tilde{\mathbf{B}}}\mathbb{N}[v,v^{-1}]b'$. By Corollary \ref{framed basis}, we have 
$$F_i\pi_{\xi\odot\lambda}(b)=\pi_{\xi\odot\lambda}(\theta_ib)\in \sum_{b'\in \tilde{\mathbf{B}}(\xi,\lambda)}\mathbb{N}[v,v^{-1}]\pi_{\xi\odot\lambda}(b').$$
For any $b\in \tilde{\mathbf{B}}(\xi,\lambda)$, since $b$ is bar-invariant, by Lemma \ref{Ei formula}, we have
\begin{align*}
E_ib=&\frac{v^{\langle i,\xi-|b|\rangle+2}{_ir(b)}-v^{-\langle i,\xi-|b|\rangle-2}\overline{{_ir}(\overline{b})}}{v-v^{-1}}\\
=&\sum_{b'\in \tilde{\mathbf{B}}}\sum_{n\in \mathbb{Z}}\frac{v^{\langle i,\xi-|b|\rangle+2}d_{b',n}^bv^nb'-v^{-\langle i,\xi-|b|\rangle-2}d_{b',n}^bv^{-n}b'}{v-v^{-1}}\\
=&\sum_{b'\in \tilde{\mathbf{B}}}\sum_{n\in \mathbb{Z}}\frac{v^{\langle i,\xi-|b|\rangle+2+n}-v^{-\langle i,\xi-|b|\rangle-2-n}}{v-v^{-1}}d_{b',n}^bb'.
\end{align*}
By Corollary \ref{framed basis} and Lemma \ref{nonzero term power}, we have
\begin{align*}
E_i\pi_{\xi\odot\lambda}(b)=\pi_{\xi\odot\lambda}(E_ib)=\sum_{b'\in \tilde{\mathbf{B}}(\xi,\lambda)}\sum_{n\in \mathbb{Z}}\frac{v^{\langle i,\xi-|b|\rangle+2+n}-v^{-\langle i,\xi-|b|\rangle-2-n}}{v-v^{-1}}d_{b',n}^b\pi_{\xi\odot\lambda}(b'),
\end{align*}
where $d_{b',n}^b\geqslant 0$, and $\langle i,\xi-|b|\rangle+2+n\geqslant 0$ whenever $d_{b',n}^b\not=0$. Notice that 
$$\frac{v^N-v^{-N}}{v-v^{-1}}=v^{N-1}+v^{N-3}+...+v^{1-N}\in \mathbb{N}[v,v^{-1}],\ \textrm{if $N\geqslant 0$},$$
and so 
$$E_i\pi_{\xi\odot\lambda}(b)\in \sum_{b'\in \tilde{\mathbf{B}}(\xi,\lambda)}\mathbb{N}[v,v^{-1}]\pi_{\xi\odot\lambda}(b').$$
Hence the basis $\tilde{\mathbf{B}}(\Lambda_{\xi,\lambda})$ of $\Lambda_{\xi,\lambda}$ has the positivity. Moreover, by Theorem \ref{framed construction of tensor} and Theorem \ref{framed construction of canonical basis}, the canonical basis $\mathbf{B}(\Lambda_{\xi}\otimes \Lambda_{\lambda})$ of $\Lambda_{\xi}\otimes \Lambda_{\lambda}$ also has the positivity.
\end{proof}

\subsection{Generalized case}\label{generalized case}

Our method can be generalized to the case for the tensor product of several integrable highest weight modules by further enlarging the Cartan datum. For example, we sketch the construction for the tensor product $\Lambda_{\xi}\otimes \Lambda_{\lambda}\otimes \Lambda_{\zeta}$, where $\zeta$ is another dominant weight with respect to the root datum $(Y,X,\langle,\rangle,...)$. For the Cartan datum $(I,\cdot)$, we enlarge its framed Cartan datum $(\tilde{I},\cdot)$ to be the $2$-framed Cartan datum $(\tilde{I}^2,\cdot)$, where $\tilde{I}^2=\tilde{I}\sqcup I''=\tilde{I}\sqcup\{i''\mid i\in I\}$ and 
$$i\cdot j''=-\delta_{i,j},\ i'\cdot j''=0,\ i''\cdot j''=2\delta_{i,j}\ \textrm{for any $i,j\in I$}.$$
Similarly, for the root datum $(Y,X,\langle,\rangle,...)$, we enlarge its framed root datum $(\tilde{Y},\tilde{X},\langle,\rangle,...)$ to be the $2$-framed root datum $(\tilde{Y}^2,\tilde{X}^2,\langle,\rangle,...)$. Then we have the corresponding algebras $\tilde{\mathbf{f}}^2$ with the canonical basis $\tilde{\mathbf{B}}^2$, and the quantized enveloping algebra $\tilde{\mathbf{U}}^2$. We define
$$\xi\odot\lambda\odot\zeta=\xi\odot\lambda+|\theta_{\zeta}|+\sum_{i\in I}(\langle i,\zeta\rangle-\langle i',\xi\odot\lambda+|\theta_{\zeta}|\rangle)\omega_{i''}\in \tilde{X}^2.$$
Then we have the corresponding Verma module $\tilde{M}_{\xi\odot\lambda\odot\zeta}^2$ and the simple highest weight module $\tilde{M}_{\xi\odot\lambda\odot\zeta}^2$ of $\tilde{\mathbf{U}}^2$ with the canonical basis $\tilde{\mathbf{B}}^2(\tilde{\Lambda}_{\xi\odot \lambda\odot \zeta}^2)$. We restrict them to be $\mathbf{U}$-modules, and can prove that the subspace $\mathbf{f}\theta_{\zeta}\mathbf{f}\theta_{\lambda}\mathbf{f}$ of $\tilde{\mathbf{f}}^2$ spanned by $\{x\theta_{\zeta}y\theta_{\lambda}z\mid x,y,z\in \mathbf{f}\}$, where $\theta_{\zeta}=\prod_{i\in I}\theta_{i''}^{(\langle i,\zeta\rangle)}\in \tilde{\mathbf{f}}^2$, is a $\mathbf{U}$-submodule $\mathbf{f}\theta_{\zeta}\mathbf{f}\theta_{\lambda}\mathbf{f}$ of $\tilde{M}_{\xi\odot\lambda\odot\zeta}^2$. This is analogue to the submodule $\mathbf{f}\theta_{\lambda}\mathbf{f}$ of $\tilde{M}_{\xi\odot\lambda}$. We define the $\mathbf{U}$-module $\Lambda_{\xi,\lambda,\zeta}$ to be the image of $\mathbf{f}\theta_{\zeta}\mathbf{f}\theta_{\lambda}\mathbf{f}$ under the natural projection $\pi_{\xi\odot\lambda\odot\zeta}:\tilde{M}_{\xi\odot\lambda\odot\zeta}^2\rightarrow \tilde{\Lambda}_{\xi\odot\lambda\odot\zeta}^2$. This is analogue to the module $\Lambda_{\xi,\lambda}$. Moreover, we can construct a subset $\tilde{\mathbf{B}}^2(\Lambda_{\xi,\lambda,\zeta})\subset \tilde{\mathbf{B}}^2(\tilde{\Lambda}_{\xi\odot \lambda\odot \zeta}^2)$, and prove that it is a basis of $\Lambda_{\xi,\lambda,\zeta}$. This is analogue to the basis $\tilde{\mathbf{B}}(\Lambda_{\xi,\lambda})$ of $\Lambda_{\xi,\lambda}$. We can construct a $\mathbf{U}$-module isomorphism $\varphi^2:\Lambda_{\xi,\lambda,\zeta}\rightarrow \Lambda_{\xi,\lambda}\otimes \Lambda_{\zeta}$ by using the comultiplication of $\tilde{\mathbf{f}}^2$, and prove that the image of $\tilde{\mathbf{B}}^2(\Lambda_{\xi,\lambda,\zeta})$ under $\varphi^2$ coincides with the canonical basis of the tensor product $\Lambda_{\xi,\lambda}\otimes \Lambda_{\zeta}$. This is analogue to Theorem \ref{framed construction of tensor} and \ref{framed construction of canonical basis}.

By Theorem \ref{framed construction of canonical basis}, the map $\varphi:\Lambda_{\xi,\lambda}\rightarrow \Lambda_{\xi}\otimes \Lambda_{\lambda}$ is a based module isomorphism, then so are $\varphi\otimes \mathrm{Id}:\Lambda_{\xi,\lambda}\otimes \Lambda_{\zeta}\rightarrow \Lambda_{\xi}\otimes \Lambda_{\lambda}\otimes \Lambda_{\zeta}$ and $(\varphi\otimes \mathrm{Id})\varphi^2:\Lambda_{\xi,\lambda,\zeta}\rightarrow \Lambda_{\xi}\otimes \Lambda_{\lambda}\otimes \Lambda_{\zeta}$. Hence the image of $\tilde{\mathbf{B}}^2(\Lambda_{\xi,\lambda,\zeta})$ under $(\varphi\otimes \mathrm{Id})\varphi^2$ coincides with the canonical basis of the tensor product $\Lambda_{\xi}\otimes \Lambda_{\lambda}\otimes \Lambda_{\zeta}$. Based on the positivity of $\tilde{\mathbf{B}}^2(\tilde{\Lambda}_{\xi\odot \lambda\odot \zeta}^2)$ due to Lusztig, we obtain the positivity of the canonical basis of the tensor product $\Lambda_{\xi}\otimes \Lambda_{\lambda}\otimes \Lambda_{\zeta}$. By repeating above process inductively, we can solve the case for the tensor product of several integrable highest weight modules.

\section{Kashiwara's operators}\label{Kashiwara's operators}

The purpose of this section is to prove Lemma \ref{almost orthonormal}. Our idea comes from the proof of \cite[Lemma 19.1.4]{Lusztig-1993} which proves the almost orthonormal property of the canonical basis $\mathbf{B}(\Lambda_{\lambda})$ of $\Lambda_{\lambda}$ for any $\lambda\in X^+$ by using Kashiwara's operators.

\subsection{Kashiwara's operators on $\mathbf{f}\theta_{\lambda}\mathbf{f}$}

For any $i\in I$, let $\tilde{\phi}_i:\tilde{\mathbf{f}}\rightarrow \tilde{\mathbf{f}},\tilde{\epsilon}_i:\tilde{\mathbf{f}}\rightarrow \tilde{\mathbf{f}}$ be the Kashiwara's operators associated to the algebra $\tilde{\mathbf{f}}$, see \cite[Section 17.3.1]{Lusztig-1993}. Recall that $\tilde{\phi}_i$ is induced by $\phi_i$ the left multiplication by $\theta_i$, and $\tilde{\epsilon}_i$ is induced by the linear map ${_ir}$, see \S \ref{Positivity} for the definition of ${_ir}$. It is clear that $\mathbf{f}\theta_{\lambda}\mathbf{f}$ is stable under $\phi_i$.
By (\ref{pniIdr}) for the case $n=1$, we know that $\mathbf{f}\theta_{\lambda}\mathbf{f}$ is also stable under ${_ir}$. Hence $\tilde{\phi}_i$ and $\tilde{\epsilon}_i$ can be restricted to $\tilde{\phi}_i:\mathbf{f}\theta_{\lambda}\mathbf{f}\rightarrow \mathbf{f}\theta_{\lambda}\mathbf{f},\ \tilde{\epsilon}_i:\mathbf{f}\theta_{\lambda}\mathbf{f}\rightarrow \mathbf{f}\theta_{\lambda}\mathbf{f}$.

Let $\mathcal{L}(\tilde{\mathbf{f}})=\{x\in {_{\mathcal{A}}}\tilde{\mathbf{f}}\mid (x,x)\in \mathbf{A}\}$. By Lemma \ref{14.2.2}, Theorem \ref{properties of B} and $\mathcal{A}\cap\mathbf{A}=\mathbb{Z}[v^{-1}]$, we know that $\mathcal{L}(\tilde{\mathbf{f}})$ is the $\mathbb{Z}[v^{-1}]$-submodule of $\tilde{\mathbf{f}}$ generated by $\tilde{\mathbf{B}}$. We have the direct sum decomposition $\mathcal{L}(\tilde{\mathbf{f}})=\bigoplus_{\tilde{\nu}\in \mathbb{N}[I]}\mathcal{L}(\tilde{\mathbf{f}})_{\tilde{\nu}}$, where $\mathcal{L}(\tilde{\mathbf{f}})_{\tilde{\nu}}$ is the $\mathbb{Z}[v^{-1}]$-submodule generated by $\tilde{\mathbf{B}}_{\tilde{\nu}}$, see \cite[Section 17.3.3, 18.1.5]{Lusztig-1993}. 

Let $\mathcal{L}(\mathbf{f}\theta_{\lambda}\mathbf{f})=\{x\in \mathbf{f}\theta_{\lambda}\mathbf{f}\cap{_{\mathcal{A}}}\tilde{\mathbf{f}}\mid (x,x)\in \mathbf{A}\}$. By Lemma \ref{14.2.2}, Proposition \ref{B(fthetaf)}, Theorem \ref{properties of B} and $\mathcal{A}\cap\mathbf{A}=\mathbb{Z}[v^{-1}]$, we know that $\mathcal{L}(\mathbf{f}\theta_{\lambda}\mathbf{f})$ is the $\mathbb{Z}[v^{-1}]$-submodule of $\mathbf{f}\theta_{\lambda}\mathbf{f}$ generated by $\tilde{\mathbf{B}}(\mathbf{f}\theta_{\lambda}\mathbf{f})$. We have the direct sum decomposition $\mathcal{L}(\mathbf{f}\theta_{\lambda}\mathbf{f})=\bigoplus_{\nu\in \mathbb{N}[I]}\mathcal{L}(\mathbf{f}\theta_{\lambda}\mathbf{f})_{\nu+|\theta_{\lambda}|}$, where $\mathcal{L}(\mathbf{f}\theta_{\lambda}\mathbf{f})_{\nu+|\theta_{\lambda}|}$ is the $\mathbb{Z}[v^{-1}]$-submodule generated by $\tilde{\mathbf{B}}(\mathbf{f}\theta_{\lambda}\mathbf{f})\cap\tilde{\mathbf{B}}_{\nu+|\theta_{\lambda}|}$. 

By \cite[Lemma 16.2.8]{Lusztig-1993}, we know that $\mathcal{L}(\mathbf{f}\theta_{\lambda}\mathbf{f})$ is stable under $\tilde{\phi}_i,\tilde{\epsilon}_i$ for any $i\in I$.

\begin{lemma}\label{17.3.7}
For any $i\in I,n\in \mathbb{N}$ and $b\in \tilde{\mathbf{B}}_{i,n}\cap \tilde{\mathbf{B}}(\mathbf{f}\theta_{\lambda}\mathbf{f})$, let $b'=\pi_{i,n}^{-1}(b)\in \tilde{\mathbf{B}}_{i,0}\cap \tilde{\mathbf{B}}(\mathbf{f}\theta_{\lambda}\mathbf{f})$, see Lemma {\rm{\ref{restricted bijection}}}. We have\\
{\rm{(a)}} $\tilde{\phi}_i(b)=\pi_{i,n+1}(b')\ \mathrm{mod}\, v^{-1}\mathcal{L}(\mathbf{f}\theta_{\lambda}\mathbf{f})$;\\
{\rm{(b)}} $\tilde{\epsilon}_i(b)=\pi_{i,n-1}(b')\ \mathrm{mod}\, v^{-1}\mathcal{L}(\mathbf{f}\theta_{\lambda}\mathbf{f})$, if $n\geqslant 1$; $\tilde{\epsilon}_i(b)=0\ \mathrm{mod}\, v^{-1}\mathcal{L}(\mathbf{f}\theta_{\lambda}\mathbf{f})$, if $n=0$.\\
As a result, for any $i\in I$ and $b\in \tilde{\mathbf{B}}(\mathbf{f}\theta_{\lambda}\mathbf{f})$, there exists a unique $b_1\in \tilde{\mathbf{B}}(\mathbf{f}\theta_{\lambda}\mathbf{f})$ such that $$\tilde{\phi}_i(b)=b_1\ \mathrm{mod}\, v^{-1}\mathcal{L}(\mathbf{f}\theta_{\lambda}\mathbf{f}),\ \tilde{\epsilon}_i(b_1)=b\ \mathrm{mod}\, v^{-1}\mathcal{L}(\mathbf{f}\theta_{\lambda}\mathbf{f}).$$
If in addition $b\in \tilde{\mathbf{B}}_{i,n}$ for some $n\geqslant 1$, then there exists a unique $b_2\in \tilde{\mathbf{B}}(\mathbf{f}\theta_{\lambda}\mathbf{f})$ such that $$\tilde{\epsilon}_i(b)=b_2\ \mathrm{mod}\, v^{-1}\mathcal{L}(\mathbf{f}\theta_{\lambda}\mathbf{f}),\ \tilde{\phi}_i(b_2)=b\ \mathrm{mod}\, v^{-1}\mathcal{L}(\mathbf{f}\theta_{\lambda}\mathbf{f}).$$
\end{lemma}
\begin{proof}
By \cite[Proposition 17.3.5, Corollary 17.3.7]{Lusztig-1993}, we have $\tilde{\phi}_i(b)=\pi_{i,n+1}(b')\ \mathrm{mod}\, v^{-1}\mathcal{L}(\tilde{\mathbf{f}})$.
By Lemma \ref{restricted bijection}, we have $\pi_{i,n+1}(b')\in \mathcal{L}(\mathbf{f}\theta_{\lambda}\mathbf{f})$. Since $\mathcal{L}(\mathbf{f}\theta_{\lambda}\mathbf{f})$ is stable under $\tilde{\phi}_i$, we have $\tilde{\phi}_i(b)\in \mathcal{L}(\mathbf{f}\theta_{\lambda}\mathbf{f})$. Hence $\tilde{\phi}_i(b)=\pi_{i,n+1}(b')\ \mathrm{mod}\, v^{-1}\mathcal{L}(\mathbf{f}\theta_{\lambda}\mathbf{f})$. The other statements can be proved similarly.
\end{proof}

For any $\nu\in \mathbb{N}[I]$, we define $\mathcal{L}(\mathbf{f}\theta_{\lambda}\mathbf{f})^{\nu}$ to be the $\mathbb{Z}[v^{-1}]$-submodule of $\mathbf{f}\theta_{\lambda}\mathbf{f}$ generated by 
$\{\tilde{\phi}_{i_1}...\tilde{\phi}_{i_{\mathrm{tr}\,\nu}}(b)\mid i_1,...,i_{\mathrm{tr}\,\nu}\in I,b\in \bigcap_{i\in I}\tilde{\mathbf{B}}_{i,0}\cap \tilde{\mathbf{B}}(\mathbf{f}\theta_{\lambda}\mathbf{f}),\sum_{k=1}^{\mathrm{tr}\,\nu} i_k=\nu\}$. Since $\mathcal{L}(\mathbf{f}\theta_{\lambda}\mathbf{f})$ is stable under $\tilde{\phi}_i$ for any $i\in I$, we have $\mathcal{L}(\mathbf{f}\theta_{\lambda}\mathbf{f})^{\nu}\subset \mathcal{L}(\mathbf{f}\theta_{\lambda}\mathbf{f})$.

\begin{corollary}\label{18.1.7}
We have $\mathcal{L}(\mathbf{f}\theta_{\lambda}\mathbf{f})=\sum_{\nu\in \mathbb{N}[I]}\mathcal{L}(\mathbf{f}\theta_{\lambda}\mathbf{f})^{\nu}$.
\end{corollary}
\begin{proof}
We denote $\mathcal{L}'=\sum_{\nu\in \mathbb{N}[I]}\mathcal{L}(\mathbf{f}\theta_{\lambda}\mathbf{f})^{\nu}$, then $\mathcal{L}'$ is a submodule of $\mathcal{L}$. For any $b\in \tilde{\mathbf{B}}(\mathbf{f}\theta_{\lambda}\mathbf{f})$, we use induction on $\mathrm{tr}\, |b|$ to prove that there exist $b'\in \bigcap_{i\in I}\tilde{\mathbf{B}}_{i,0}\cap \tilde{\mathbf{B}}(\mathbf{f}\theta_{\lambda}\mathbf{f})$ and $i_1,...,i_n\in I$ such that $b=\tilde{\phi}_{i_1}...\tilde{\phi}_{i_n}(b')\ \mathrm{mod}\ v^{-1}\mathcal{L}(\mathbf{f}\theta_{\lambda}\mathbf{f})$.\\
$\bullet$ If $b\in \bigcap_{i\in I}\tilde{\mathbf{B}}_{i,0}$, then it is clear.\\
$\bullet$ If $b\in \tilde{\mathbf{B}}_{i,\geqslant 1}$ for some $i\in I$, by Lemma \ref{17.3.7}, we have $b=\tilde{\phi}_i(b_2)\ \mathrm{mod}\, v^{-1}\mathcal{L}(\mathbf{f}\theta_{\lambda}\mathbf{f})$ for a unique $b_2\in \tilde{\mathbf{B}}(\mathbf{f}\theta_{\lambda}\mathbf{f})$. Notice that $\mathrm{tr}\,|b_2|=\mathrm{tr}\,|b|-1$, by the inductive hypothesis, there exist $b'\in \bigcap_{i\in I}\tilde{\mathbf{B}}_{i,0}\cap \tilde{\mathbf{B}}(\mathbf{f}\theta_{\lambda}\mathbf{f})$ and $i_2,...,i_n\in I$ such that $b_2=\tilde{\phi}_{i_2}...\tilde{\phi}_{i_n}(b')\ \mathrm{mod}\ v^{-1}\mathcal{L}(\mathbf{f}\theta_{\lambda}\mathbf{f})$. Since $\mathcal{L}(\mathbf{f}\theta_{\lambda}\mathbf{f})$ is stable under $\tilde{\phi}_i$, we have $b=\tilde{\phi}_i\tilde{\phi}_{i_2}...\tilde{\phi}_{i_n}(b')\ \mathrm{mod}\ v^{-1}\mathcal{L}(\mathbf{f}\theta_{\lambda}\mathbf{f})$, as desired. Hence $\mathcal{L}(\mathbf{f}\theta_{\lambda}\mathbf{f})\subset \mathcal{L}'+v^{-1}\mathcal{L}(\mathbf{f}\theta_{\lambda}\mathbf{f})$, and so the $\mathbb{Z}[v^{-1}]$-module $\mathcal{L}(\mathbf{f}\theta_{\lambda}\mathbf{f})/\mathcal{L}'$ is annihilated by $v^{-1}$. By Nakayama's lemma, it is zero, and so $\mathcal{L}(\mathbf{f}\theta_{\lambda}\mathbf{f})=\mathcal{L}'$.
\end{proof}

\begin{lemma}\label{thetalambdaB}
{\rm{(a)}} The left multiplication by $\theta_{\lambda}$ defines an injective map from $\mathbf{B}$ to $\tilde{\mathbf{B}}(\mathbf{f}\theta_{\lambda}\mathbf{f})$ such that $t_i^{\sigma}(b_0)=t_i^{\sigma}(\theta_{\lambda}b_0)$ for any $b_0\in \mathbf{B}$ and $i\in I$. As a result, $b_0\in \mathbf{B}(\xi)$ if and only if $\theta_{\lambda}b_0\in \tilde{\mathbf{B}}(\xi,\lambda)$.\\
{\rm{(b)}} This map can be restricted to be a bijection 
$$\theta_{\lambda}\cdot:\sigma(\mathbf{B}(\lambda))\rightarrow \bigcap_{i\in I}\tilde{\mathbf{B}}_{i,0}\cap\tilde{\mathbf{B}}(\mathbf{f}\theta_{\lambda}\mathbf{f}).$$
\end{lemma}
\begin{proof}
(a) Firstly, we prove that the left multiplication by $\theta_{\lambda}$ defines an injective linear map from $\mathbf{f}$ to $\tilde{\mathbf{f}}$. For any homogeneous $x\in \mathbf{f}$, we use induction on $\mathrm{tr}\,|x|\in \mathbb{N}$ to prove that if $\theta_{\lambda}x=0$, then $x=0$.\\
$\bullet$ If $\mathrm{tr}\,|x|=0$, then it is clear.\\
$\bullet$ If $\mathrm{tr}\,|x|>0$, then for any $i\in I$, by (\ref{Lusztig-1.2.13}), we have $${_ir}(\theta_{\lambda}x)={_ir}(\theta_{\lambda})x+v^{|\theta_{\lambda}|\cdot i}\theta_{\lambda}{_ir}(x)=v^{-\langle i\,\lambda\rangle}\theta_{\lambda}{_ir}(x)=0$$
and then $\theta_{\lambda}{_ir}(x)=0$. Notice that $\mathrm{tr}\,|{_ir}(x)|=\mathrm{tr}\,|x|-1$, by the inductive hypothesis, we have ${_ir}(x)=0$, and so $x=0$ by \cite[Lemma 1.2.15]{Lusztig-1993}, as desired. Hence the map $\theta_{\lambda}\cdot:\mathbf{f}\rightarrow \tilde{\mathbf{f}}$ is injective.

Secondly, for any $b_0\in \mathbf{B}$, it is clear that $b_0\in \bigcap_{i\in I}\tilde{\mathbf{B}}_{i',0}$. For any $j\in I$, by Theorem \ref{14.3.2}, we have $\theta_{j'}^{(\langle j,\lambda\rangle)}b_0=\pi_{j',\langle j,\lambda\rangle}(b_0)\in \tilde{\mathbf{B}}$. Notice that we have $\theta_{j'}^{(\langle j,\lambda\rangle)}b_0\in \bigcap_{i\in I,i\not=j}\tilde{\mathbf{B}}_{i',0}$, by repeating above process, we obtain $\theta_{\lambda}b_0\in \tilde{\mathbf{B}}$. By definitions, it is clear that $\theta_{\lambda}b_0\in \tilde{\mathbf{B}}(\mathbf{f}\theta_{\lambda}\mathbf{f})$. It remains to prove that $t_i^{\sigma}(b_0)=t_i^{\sigma}(\theta_{\lambda}b_0)$ for any $i\in I$. By \cite[Section 14.3.1]{Lusztig-1993}, it is equivalent to
$t_i(\sigma(b_0))=t_i(\sigma(\theta_{\lambda}b_0))=t_i(\sigma(b_0)\theta_{\lambda})$ for any $i\in I$. For any $x\in \mathbf{f}$ and $i\in I$, it is clear that $\phi_i(x)\theta_{\lambda}=\phi_i(x\theta_{\lambda})$. By \eqref{Lusztig-1.2.13}, we have ${_ir}(x)\theta_{\lambda}={_ir}(x\theta_{\lambda})$. Hence the right multiplication by $\theta_{\lambda}$ commutes with the actions of $\phi_i,\epsilon_i$ for any $i\in I$, and so it commutes with the Kashiwara's operators $\tilde{\phi_i},\tilde{\epsilon}_i$ on $\mathbf{f},\mathbf{f}\theta_{\lambda}\mathbf{f}$ for any $i\in I$. By \cite[Corollary 17.3.7]{Lusztig-1993} and Lemma \ref{17.3.7}, we have $t_i(\sigma(b_0))=t_i(\sigma(b_0)\theta_{\lambda})$ for any $i\in I$, as desired.

(b) Firstly, we prove that $\theta_{\lambda}b_0\in \bigcap_{i\in I}\tilde{\mathbf{B}}_{i,0}$ for any $b_0\in \sigma(\mathbf{B}(\lambda))$. Suppose $\theta_{\lambda}b_0\in \tilde{\mathbf{B}}_{j,\geqslant 1}$ for some $j\in I$. By \cite[Section 14.3.1]{Lusztig-1993}, we have $t_j^{\sigma}(\sigma(\theta_{\lambda}b_0))=t_j(\theta_{\lambda}b_0)\geqslant 1$, and so
$\sigma(b_0)\theta_{\lambda}=\sigma(b_0)\sigma(\theta_{\lambda})=\sigma(\theta_{\lambda}b_0)\in \tilde{\mathbf{f}}\theta_j\subset\sum_{i\in I}\tilde{\mathbf{f}}\theta_i$. By Proposition \ref{psi}, we know that $\sigma(b_0)\in \sum_{i\in I}\mathbf{f}\theta_i^{\langle i,\lambda\rangle+1}$, a contradiction to $b_0\in \sigma(\mathbf{B}(\lambda))$. Hence $\theta_{\lambda}\cdot$ can be restricted to be an injective map $\sigma(\mathbf{B}(\lambda))\rightarrow \bigcap_{i\in I}\tilde{\mathbf{B}}_{i,0}\cap\tilde{\mathbf{B}}(\mathbf{f}\theta_{\lambda}\mathbf{f})$. Secondly, for any $b\in \bigcap_{i\in I}\tilde{\mathbf{B}}_{i,0}\cap\tilde{\mathbf{B}}(\mathbf{f}\theta_{\lambda}\mathbf{f})$, we use induction on $\mathrm{tr}\, |b|$ to prove that there exists $b_0\in \mathbf{B}$ such that $b=\theta_{\lambda}b_0$.\\
$\bullet$ If $t_i^{\sigma}(b)=0$ for any $i\in I$, by Corollary \ref{t_i=s_i}, $b$ can only appear in $\theta_{\lambda}$, and so $b=\theta_{\lambda}$.\\
$\bullet$ If $t_i^{\sigma}(b)>0$ for some $i\in I$, by Theorem \ref{14.3.2} and Lemma \ref{restricted bijection}, there exists a unique $b'\in\tilde{\mathbf{B}}_{i,0}^{\sigma}\cap \tilde{\mathbf{B}}
(\mathbf{f}\theta_{\lambda}\mathbf{f})$ such that $b'\theta_i^{(t_i^{\sigma}(b))}=b+\sum_{b''\in \tilde{\mathbf{B}}_{i,\geqslant t_i^{\sigma}(b)+1}^{\sigma}\cap \tilde{\mathbf{B}}(\mathbf{f}\theta_{\lambda}\mathbf{f})}c_{b''}b''$. We claim that $b'\in \bigcap_{j\in I}\tilde{\mathbf{B}}_{j,0}$. Indeed, suppose that $t_j(b')>0$ for some $j\in I$. By Theorem \ref{14.3.2} and Lemma \ref{restricted bijection}, there exists a unique $b'_1\in \tilde{\mathbf{B}}_{j,0}\cap \tilde{\mathbf{B}}
(\mathbf{f}\theta_{\lambda}\mathbf{f})$ such that $\theta_j^{t_j(b')}b'_1=b'+\sum_{b''_1\in\tilde{\mathbf{B}}_{j,\geqslant t_j(b')}\cap \tilde{\mathbf{B}}
(\mathbf{f}\theta_{\lambda}\mathbf{f})}c'_{b''_1}b''_1$. Then we have 
\begin{align*}
\theta_j^{t_j(b')}b'_1\theta_i^{(t_i^{\sigma}(b))}=&b'\theta_i^{(t_i^{\sigma}(b))}+\sum_{b''_1\in\tilde{\mathbf{B}}_{j,\geqslant t_j(b')}\cap \tilde{\mathbf{B}}
(\mathbf{f}\theta_{\lambda}\mathbf{f})}c'_{b''_1}b''_1\theta_i^{(t_i^{\sigma}(b))}\\
=&b+\sum_{b''\in \tilde{\mathbf{B}}_{i,\geqslant t_i^{\sigma}(b)+1}^{\sigma}\cap \tilde{\mathbf{B}}(\mathbf{f}\theta_{\lambda}\mathbf{f})}c_{b''}b''+\sum_{b''_1\in\tilde{\mathbf{B}}_{j,\geqslant t_j(b')}\cap \tilde{\mathbf{B}}
(\mathbf{f}\theta_{\lambda}\mathbf{f})}c'_{b''_1}b''_1\theta_i^{(t_i^{\sigma}(b))}.
\end{align*}
Notice that $c_{b''},c'_{b''_1}\in \mathbb{N}[v,v^{-1}]$ and $b''_1\theta_i^{(t_i^{\sigma}(b))}$ is a $\mathbb{N}[v,v^{-1}]$-linear combination of $\tilde{\mathbf{B}}$. Hence $b$ appears in $\theta_j^{t_j(b')}b'_1\theta_i^{(t_i^{\sigma}(b))}$. By Corollary \ref{t_i=s_i}, we have $t_j(b)\geqslant t_j(b')>0$, a contradiction. Notice that $\mathrm{tr}\,|b'|=\mathrm{tr}\,|b|-t_i^{\sigma}(b)$, by the inductive hypothesis, there exists $b'_0\in \mathbf{B}$ such that $b'=\theta_{\lambda}b'_0$. By (a), we have $t_i^{\sigma}(b'_0)=t_i^{\sigma}(b')=0$. By Theorem \ref{14.3.2}, there exists a unique $b_0\in \mathbf{B}_{i,t_i^{\sigma}(b)}^{\sigma}$ such that $b'_0\theta_i^{(t_i^{\sigma}(b))}=b_0+\sum_{b''_0\in \mathbf{B}_{i,\geqslant t_i^{\sigma}(b)+1}^{\sigma}}c''_{b''_0}b''_0$.
Hence 
$$b-\theta_{\lambda}b_0=\sum_{b''_0 \in \mathbf{B}_{i,\geqslant t_i^{\sigma}(b)+1}^{\sigma}}c''_{b''_0}\theta_{\lambda}b''_0-\sum_{b''\in \tilde{\mathbf{B}}_{i,\geqslant t_i^{\sigma}(b)+1}^{\sigma}\cap \tilde{\mathbf{B}}(\mathbf{f}\theta_{\lambda}\mathbf{f})}c_{b''}b''.$$
By (a), we have $\theta_{\lambda}b_0\in \tilde{\mathbf{B}}_{i,t_i^{\sigma}(b)}^{\sigma}\cap \tilde{\mathbf{B}}(\mathbf{f}\theta_{\lambda}\mathbf{f})$ and $\theta_{\lambda}b''_0\in \tilde{\mathbf{B}}_{i,\geqslant t_i^{\sigma}(b)+1}^{\sigma}\cap \tilde{\mathbf{B}}(\mathbf{f}\theta_{\lambda}\mathbf{f})$. Hence we have $b=\theta_{\lambda}b_0$, as desired. It remains to prove that $b_0\in \sigma(\mathbf{B}(\lambda))$. Suppose that $t_j(b_0)\geqslant \langle j,\lambda\rangle+1$ for some $j\in I$, since $i'\cdot j=-\delta_{i,j}$ for any $i\in I$, by the quantum Serre relations and \cite[Lemma 42.1.2(b)]{Lusztig-1993}, we have 
\begin{align*}
&\theta_{i'}^{\langle i,\lambda\rangle}\theta_{j}^{(\langle j,\lambda\rangle)+1}=\theta_{j}^{(\langle j,\lambda\rangle)+1}\theta_{i'}^{\langle i,\lambda\rangle}\ \textrm{if $j\not=j$};\\
&\theta_{j'}^{\langle i,\lambda\rangle}\theta_{j}^{(\langle j,\lambda\rangle)+1}=\sum_{n=0}^{\langle j,\lambda\rangle}\frac{v^{-(\langle j,\lambda\rangle+1-n)(\langle j,\lambda\rangle-n)}}{[n]!}\theta_{j}^{(\langle j,\lambda\rangle+1-n)}(\theta_{j'}\theta_j-v^{-1}\theta_j\theta_{j'})^n\theta_{j'}^{(\langle j,\lambda\rangle-n)}.
\end{align*}
such that $b=\theta_{\lambda}b_0\in \theta_j\tilde{\mathbf{f}}$, and so $t_j(b)\geqslant 1$, a contradiction. Hence by \cite[Section 14.3.1]{Lusztig-1993}, we have $t_i^{\sigma}(\sigma(b_0))=t_i(b_0)\leqslant \langle i,\lambda\rangle$ for any $i\in I$, and so $b_0\in \sigma(\mathbf{B}(\lambda))$. Therefore, the map $\theta_{\lambda}\cdot:\sigma(\mathbf{B}(\lambda))\rightarrow \bigcap_{i\in I}\tilde{\mathbf{B}}_{i,0}\cap\tilde{\mathbf{B}}(\mathbf{f}\theta_{\lambda}\mathbf{f})$ is surjective, and so it is a bijection.
\end{proof}

Similarly, there is a bijection
\begin{equation}\label{Fang-Lan-2025-5.3*}
\cdot\theta_{\lambda}:\mathbf{B}(\lambda)\rightarrow \bigcap_{i\in I}\tilde{\mathbf{B}}_{i,0}^{\sigma}\cap\tilde{\mathbf{B}}(\mathbf{f}\theta_{\lambda}\mathbf{f}).
\end{equation}

For any $\nu\in \mathbb{N}[I]$ and $b_0\in \sigma(\mathbf{B}(\lambda))$, we define $\mathcal{L}(\mathbf{f}\theta_{\lambda}\mathbf{f})^{\nu,b_0}$ to be the $\mathbb{Z}[v^{-1}]$-submodule of $\mathbf{f}\theta_{\lambda}\mathbf{f}$ generated by 
$\{\tilde{\phi}_{i_1}...\tilde{\phi}_{i_{\mathrm{tr}\,\nu}}(\theta_{\lambda}b_0)\mid i_1,...,i_{\mathrm{tr}\,\nu}\in I,\sum_{k=1}^{\mathrm{tr}\,\nu} i_k=\nu\}$. By Lemma \ref{thetalambdaB}, we have $\mathcal{L}(\mathbf{f}\theta_{\lambda}\mathbf{f})^{\nu}=\sum_{b_0\in \sigma(\mathbf{B}(\lambda))}\mathcal{L}(\mathbf{f}\theta_{\lambda}\mathbf{f})^{\nu,b_0}$.

\subsection{Kashiwara's operators on $\Lambda_{\xi,\lambda}$}

By definitions, the $\mathbf{U}$-modules $\Lambda_{\xi},\Lambda_{\lambda}$ and $\tilde{\Lambda}_{\xi\odot\lambda}$ are integrable. By \cite[Section 3.5.2]{Lusztig-1993} and Theorem \ref{framed construction of tensor}, the $\mathbf{U}$-modules $\Lambda_{\xi}\otimes \Lambda_{\lambda}$ and $\Lambda_{\xi,\lambda}$ are integrable. For any $i\in I$, let $\tilde{F}_i,\tilde{E}_i$ be the Kashiwara's operators associated to these integrable modules, see \cite[Lemma 16.1.4]{Lusztig-1993}.

Let $L(\tilde{\Lambda}_{\xi\odot\lambda})=\{m\in \tilde{\Lambda}_{\xi\odot\lambda}\mid (m,m)_{\xi\odot\lambda}\in \mathbf{A}\}$ . By Lemma \ref{14.2.2} and Theorem \ref{properties of B(Lambda)}, we know that $L(\tilde{\Lambda}_{\xi\odot\lambda})$ is the $\mathbf{A}$-submodule of $\tilde{\Lambda}_{\xi\odot\lambda}$ generated by $\tilde{\mathbf{B}}(\tilde{\Lambda}_{\xi\odot\lambda})$. We have the direct sum decomposition $L(\tilde{\Lambda}_{\xi\odot\lambda})=\bigoplus_{\tilde{\nu}\in \mathbb{N}[\tilde{I}]}L(\tilde{\Lambda}_{\xi\odot\lambda})_{\tilde{\nu}}$, where $L(\tilde{\Lambda}_{\xi\odot\lambda})_{\tilde{\nu}}$ is the $\mathbf{A}$-submodule generated by $\{\pi_{\xi\odot\lambda}(b)\mid b\in \tilde{\mathbf{B}}(\xi\odot\lambda)\cap\tilde{\mathbf{B}}_{\tilde{\nu}}\}$, see \cite[Section 18.1.8]{Lusztig-1993}. 

Let $L(\Lambda_{\xi,\lambda})=\{m\in \Lambda_{\xi,\lambda}\mid (m,m)_{\xi\odot\lambda}\in \mathbf{A}\}$. By Lemma \ref{14.2.2}, Corollary \ref{framed basis} and Theorem \ref{properties of B(Lambda)}, we know that $L(\Lambda_{\xi,\lambda})$ is the $\mathbf{A}$-submodule of $\Lambda_{\xi,\lambda}$ generated by $\tilde{\mathbf{B}}(\Lambda_{\xi,\lambda})$. We have the direct sum decomposition $L(\Lambda_{\xi,\lambda})=\bigoplus_{\nu\in \mathbb{N}[I]}L(\Lambda_{\xi,\lambda})_{\nu+|\theta_{\lambda}|}$, where $L(\Lambda_{\xi,\lambda})_{\nu+|\theta_{\lambda}|}$ is the $\mathbf{A}$-submodule generated by $\{\pi_{\xi\odot\lambda}(b)\mid b\in \tilde{\mathbf{B}}(\xi,\lambda)\cap\tilde{\mathbf{B}}_{\nu+|\theta_{\lambda}|}\}$. 

For any $m\in L(\Lambda_{\xi,\lambda})$ and $i\in I$, by \cite[Proposition 19.1.3]{Lusztig-1993}, we have
\begin{align*}
&(\tilde{F_i}(m),\tilde{F_i}(m))_{\xi\odot\lambda}=(m,\tilde{E}_i\tilde{F}_i(m))_{\xi\odot\lambda}\ \mathrm{mod}\,v^{-1}\mathbf{A}=(m,m)_{\xi\odot\lambda}\ \mathrm{mod}\,v^{-1}\mathbf{A}\in \mathbf{A},\\
&(\tilde{E_i}(m),\tilde{E_i}(m))_{\xi\odot\lambda}=(\tilde{F}_i\tilde{E}_i(m),m)_{\xi\odot\lambda}\ \mathrm{mod}\,v^{-1}\mathbf{A}=(m,m)_{\xi\odot\lambda}\ \mathrm{mod}\,v^{-1}\mathbf{A}\in \mathbf{A},
\end{align*}
and so $\tilde{F_i}(m),\tilde{E_i}(m)\in L(\Lambda_{\xi,\lambda})$. Hence $L(\Lambda_{\xi,\lambda})$ is stable under $\tilde{F_i},\tilde{E_i}$ for any $i\in I$.

\begin{lemma}\label{18.3.8}
For any $\nu\in \mathbb{N}[I]$ and  $i\in I$, we have \\
{\rm{(a)}} $\pi_{\xi\odot\lambda}\tilde{\phi}_i(x)=\tilde{F}_i\pi_{\xi\odot\lambda}(x)\ \mathrm{mod}\,v^{-1}L(\Lambda_{\xi,\lambda})$ for any $x\in \mathcal{L}(\mathbf{f}\theta_{\lambda}\mathbf{f})_{\nu+|\theta_{\lambda}|}$;\\
{\rm{(b)}} $\pi_{\xi\odot\lambda}\tilde{\epsilon}_i(b)=\tilde{E}_i\pi_{\xi\odot\lambda}(b)\ \mathrm{mod}\,v^{-1}L(\Lambda_{\xi,\lambda})$ for any $b\in \tilde{\mathbf{B}}(\xi,\lambda)\cap\tilde{\mathbf{B}}_{\nu+|\theta_{\lambda}|}$, if $\nu_i>0$.
\end{lemma}
\begin{proof}
By \cite[Theorem 18.3.8]{Lusztig-1993}, we have $\pi_{\xi\odot\lambda}\tilde{\phi}_i(x)=\tilde{F}_i\pi_{\xi\odot\lambda}(x)\ \mathrm{mod}\,v^{-1}L(\tilde{\Lambda}_{\xi\odot\lambda})$. Since $L(\Lambda_{\xi,\lambda})$ is stable under $\tilde{F_i}$ and $\mathcal{L}(\mathbf{f}\theta_{\lambda}\mathbf{f})$ is stable under $\tilde{\phi}_i$, we have $\tilde{F}_i\pi_{\xi\odot\lambda}(x)\in L(\Lambda_{\xi,\lambda})$ and $\tilde{\phi}_i(x)\in \mathcal{L}(\mathbf{f}\theta_{\lambda}\mathbf{f}), \pi_{\xi\odot\lambda}\tilde{\phi}_i(x)\in L(\Lambda_{\xi,\lambda})$. Hence $\pi_{\xi\odot\lambda}\tilde{\phi}_i(x)=\tilde{F}_i\pi_{\xi\odot\lambda}(x)\ \mathrm{mod}\,v^{-1}L(\Lambda_{\xi,\lambda})$. The other statements can be proved similarly.
\end{proof}

For any $\nu\in \mathbb{N}[I]$ and $b_0\in \sigma(\mathbf{B}(\lambda))$, let $L(\Lambda_{\xi,\lambda})^{\nu,b_0}$ be the $\mathbf{A}$-submodule of $\Lambda_{\xi,\lambda}$ generated by $\{\pi_{\xi\odot\lambda}\tilde{\phi}_{i_1}...\tilde{\phi}_{i_{\mathrm{tr}\,\nu}}(\theta_{\lambda}b_0)\mid i_1,...,i_{\mathrm{tr}\,\nu}\in I,\sum_{k=1}^{\mathrm{tr}\,\nu} i_k=\nu\}$. By definitions, Corollary \ref{18.1.7} and Lemma \ref{thetalambdaB}, we have $L(\Lambda_{\xi,\lambda})^{\nu,b_0}=\mathbf{A}\otimes_{\mathbb{Z}[v^{-1}]}\pi_{\xi\odot\lambda}(\mathcal{L}(\mathbf{f}\theta_{\lambda}\mathbf{f})^{\nu,b_0})$ and  
\begin{align*}
L(\Lambda_{\xi,\lambda})=\mathbf{A}\otimes_{\mathbb{Z}[v^{-1}]}\pi_{\xi\odot\lambda}(\mathcal{L}(\mathbf{f}\theta_{\lambda}\mathbf{f}))=\sum_{\nu\in \mathbb{N}[I]}\sum_{b_0\in \sigma(\mathbf{B}(\lambda))}L(\Lambda_{\xi,\lambda})^{\nu,b_0}.
\end{align*}
Note that $L(\Lambda_{\xi,\lambda})^{\nu,b_0}\subset L(\Lambda_{\xi,\lambda})_{\nu+|\theta_{\lambda}b_0|}$, and so $L(\Lambda_{\xi,\lambda})_{\nu+|\theta_{\lambda}|}=\sum L(\Lambda_{\xi,\lambda})^{\nu-|b_0|,b_0}$, where the sum is taken over $b_0\in \sigma(\mathbf{B}(\lambda))$ such that $\nu-|b_0|\in \mathbb{N}[I]$. 

For any $\nu\in \mathbb{N}[I]$ and $b_0\in \sigma(\mathbf{B}(\lambda))$, let $L'(\Lambda_{\xi,\lambda})^{\nu,b_0}$ be the $\mathbf{A}$-submodule of $\Lambda_{\xi,\lambda}$ generated by $\{\tilde{F}_{i_1}...\tilde{F}_{i_{\mathrm{tr}\,\nu}}\pi_{\xi\odot\lambda}(\theta_{\lambda}b_0)\mid i_1,...,i_{\mathrm{tr}\,\nu}\in I,\sum_{k=1}^{\mathrm{tr}\,\nu} i_k=\nu\}$, and let 
$$L'(\Lambda_{\xi,\lambda})=\sum_{\nu\in \mathbb{N}[I]}\sum_{b_0\in \sigma(\mathbf{B}(\lambda))}L'(\Lambda_{\xi,\lambda})^{\nu,b_0}.$$
Since $L(\Lambda_{\xi,\lambda})$ is stable under $\tilde{F_i}$ for any $i\in I$, we have $L'(\Lambda_{\xi,\lambda})^{\nu,b_0}\subset L(\Lambda_{\xi,\lambda})$.

\begin{lemma}\label{L=L'}
We have $L(\Lambda_{\xi,\lambda})=L'(\Lambda_{\xi,\lambda})$.
\end{lemma}
\begin{proof}
We know that $L'(\Lambda_{\xi,\lambda})$ is a submodule of $L(\Lambda_{\xi,\lambda})$. For any $\nu\in \mathbb{N}[I]$, we use induction on $\mathrm{tr}\,|\nu|\in \mathbb{N}$ to prove that 
$$L(\Lambda_{\xi,\lambda})_{\nu+|\theta_{\lambda}|}\subset L'(\Lambda_{\xi,\lambda})+v^{-1}L(\Lambda_{\xi,\lambda}).$$ 
$\bullet$ If $\mathrm{tr}\,|\nu|=0$, then $L(\Lambda_{\xi,\lambda})_{|\theta_{\lambda}|}$ is the $\mathbf{A}$-submodule of $\Lambda_{\xi,\lambda}$ generated by $\{\pi_{\xi\odot\lambda}(\theta_{\lambda})\}$, and it is clear that $L(\Lambda_{\xi,\lambda})_{|\theta_{\lambda}|}\subset L'(\Lambda_{\xi,\lambda})+v^{-1}L(\Lambda_{\xi,\lambda})$.\\
$\bullet$ If $\mathrm{tr}\,|\nu|>0$, recall that $L(\Lambda_{\xi,\lambda})_{\nu+|\theta_{\lambda}|}=\sum L(\Lambda_{\xi,\lambda})^{\nu-|b_0|,b_0}$, where the sum is taken over $b_0\in \sigma(\mathbf{B}(\lambda))$ such that $\nu-|b_0|\in \mathbb{N}[I]$. For those $b_0$ such that $\nu-|b_0|=0$, it is clear that $L(\Lambda_{\xi,\lambda})^{0,b_0}=L'(\Lambda_{\xi,\lambda})^{0,b_0}\subset L'(\Lambda_{\xi,\lambda})+v^{-1}L(\Lambda_{\xi,\lambda})$. For those $b_0$ such that $\nu-|b_0|\not=0$, by definitions, we have $L(\Lambda_{\xi,\lambda})^{\nu-|b_0|,b_0}=\sum\mathbf{A}\otimes_{\mathbb{Z}[v^{-1}]}\pi_{\xi\odot\lambda}\tilde{\phi}_i(\mathcal{L}(\mathbf{f}\theta_{\lambda}\mathbf{f})^{\nu-i-|b_0|,b_0})$,
where the sum is taken over $i\in I$ such that $\nu_i-|b_0|_i>0$. For any $x\in \mathcal{L}(\mathbf{f}\theta_{\lambda}\mathbf{f})^{\nu-i-|b_0|,b_0}$, since $\pi_{\xi\odot\lambda}(x)\in L(\Lambda_{\xi,\lambda})^{\nu-i-|b_0|,b_0}\subset L(\Lambda_{\xi,\lambda})_{\nu-i+|\theta_{\lambda}|}$,
by the inductive hypothesis, we have $\pi_{\xi\odot\lambda}(x)\in L'(\Lambda_{\xi,\lambda})+v^{-1}L(\Lambda_{\xi,\lambda})$. Since $L'(\Lambda_{\xi,\lambda})$ and $L(\Lambda_{\xi,\lambda})$ are stable under $\tilde{F}_i$, we have $\tilde{F}_i\pi_{\xi\odot\lambda}(x)\in L'(\Lambda_{\xi,\lambda})+v^{-1}L(\Lambda_{\xi,\lambda})$. By Lemma \ref{18.3.8}, we have 
$$\pi_{\xi\odot\lambda}\tilde{\phi}_i(x)=\tilde{F}_i\pi_{\xi\odot\lambda}(x)\ \mathrm{mod}\,v^{-1}L(\Lambda_{\xi,\lambda})\in L'(\Lambda_{\xi,\lambda})+v^{-1}L(\Lambda_{\xi,\lambda}).$$
So $L(\Lambda_{\xi,\lambda})^{\nu-|b_0|,b_0}\subset L'(\Lambda_{\xi,\lambda})+v^{-1}L(\Lambda_{\xi,\lambda})$, and then $ L(\Lambda_{\xi,\lambda})_{\nu+|\theta_{\lambda}|}\subset L'(\Lambda_{\xi,\lambda})+v^{-1}L(\Lambda_{\xi,\lambda})$, as desired. Hence $L(\Lambda_{\xi,\lambda})\subset L'(\Lambda_{\xi,\lambda})+v^{-1}L(\Lambda_{\xi,\lambda})$, and so the $\mathbf{A}$-module $L(\Lambda_{\xi,\lambda})/L'(\Lambda_{\xi,\lambda})$ is annihilated by $v^{-1}$. By Nakayama's lemma, it is zero, and so $L(\Lambda_{\xi,\lambda})=L'(\Lambda_{\xi,\lambda})$.
\end{proof}

Note that $L'(\Lambda_{\xi,\lambda})^{\nu,b_0}\subset L(\Lambda_{\xi,\lambda})_{\nu+|\theta_{\lambda}b_0|}$, and so $L(\Lambda_{\xi,\lambda})_{\nu+|\theta_{\lambda}|}=\sum L'(\Lambda_{\xi,\lambda})^{\nu-|b_0|,b_0}$, where the sum is taken over $b_0\in \sigma(\mathbf{B}(\lambda))$ such that $\nu-|b_0|\in \mathbb{N}[I]$.

Let $L(\Lambda_{\xi}\otimes \Lambda_{\lambda})=\{m\in \Lambda_{\xi}\otimes \Lambda_{\lambda}\mid (m,m)\in \mathbf{A}\}$. By Lemma \ref{14.2.2}, we know that $L(\Lambda_{\xi}\otimes \Lambda_{\lambda})$ is the $\mathbf{A}$-submodule of $\Lambda_{\xi}\otimes \Lambda_{\lambda}$ generated by $\mathbf{B}(\Lambda_{\xi})\otimes \mathbf{B}(\Lambda_{\lambda})$.

\begin{lemma}\label{adjointtensor}
For any $m,m'\in L(\Lambda_{\xi}\otimes \Lambda_{\lambda})$, we have 
$$(\tilde{F}_i(m),m')_{\xi,\lambda}=(m,\tilde{E}_i(m'))_{\xi,\lambda}\ \mathrm{mod}\,v^{-1}\mathbf{A}.$$
As a result, $L(\Lambda_{\xi}\otimes \Lambda_{\lambda})$ is stable under $\tilde{F_i},\tilde{E_i}$ for any $i\in I$.
\end{lemma}
\begin{proof}
It follows from \cite[Corollary 17.2.4, Proposition 19.1.3]{Lusztig-1993} and definitions.
\end{proof}

\begin{lemma}\label{(L,L)inA}
For any $\nu\in \mathbb{N}[I]$, we have\\
{\rm{(a)}} $\varphi(L(\Lambda_{\xi,\lambda})_{\nu+|\theta_{\lambda}|})\subset L(\Lambda_{\xi}\otimes \Lambda_{\lambda})$.\\
{\rm{(b)}} $(\varphi(L(\Lambda_{\xi,\lambda})_{\nu+|\theta_{\lambda}|}),\varphi(L(\Lambda_{\xi,\lambda})_{\nu+|\theta_{\lambda}|}))_{\xi,\lambda}\subset \mathbf{A}$.
\end{lemma}
\begin{proof}
We use induction on $\mathrm{tr}\,\nu$ to prove the statement.\\
(a) $\bullet$ If $\mathrm{tr}\,|\nu|=0$, then $L(\Lambda_{\xi,\lambda})_{|\theta_{\lambda}|}$ is the $\mathbf{A}$-submodule of $\Lambda_{\xi,\lambda}$ generated by $\{\pi_{\xi\odot\lambda}(\theta_{\lambda})\}$. By Theorem \ref{framed construction of tensor}, we have $\varphi\pi_{\xi\odot\lambda}(\theta_{\lambda})=\eta_{\xi}\otimes \eta_{\lambda}\in L(\Lambda_{\xi}\otimes \Lambda_{\lambda})$.\\
$\bullet$ If $\mathrm{tr}\,|\nu|>0$, recall that $L(\Lambda_{\xi,\lambda})_{\nu+|\theta_{\lambda}|}=\sum L'(\Lambda_{\xi,\lambda})^{\nu-|b_0|,b_0}$, where the sum is taken over $b_0\in \sigma(\mathbf{B}(\lambda))$ such that $\nu-|b_0|\in \mathbb{N}[I]$. For those $\nu-|b_0|=0$, by Theorem \ref{framed construction of tensor}, we have $\varphi\pi_{\xi\odot\lambda}(\theta_{\lambda}b_0)=b_0^-\eta_{\xi}\otimes \eta_{\lambda}\in L(\Lambda_{\xi}\otimes \Lambda_{\lambda})$, and so $\varphi(L'(\Lambda_{\xi,\lambda})^{0,b_0})\subset L(\Lambda_{\xi}\otimes \Lambda_{\lambda})$. For those $b_0$ such that $\nu-|b_0|\not=0$, by definitions, we have $L'(\Lambda_{\xi,\lambda})^{\nu-|b_0|,b_0}=\sum \tilde{F}_i(L'(\Lambda_{\xi,\lambda})^{\nu-i-|b_0|,b_0})$, where the sum is taken over $i\in I$ such that $\nu_i-|b_0|_i>0$. For such $i$, notice that $L'(\Lambda_{\xi,\lambda})^{\nu-i-|b_0|,b_0}\subset L(\Lambda_{\xi,\lambda})_{\nu-i+|\theta_{\lambda}|}$, by the inductive hypothesis and Lemma \ref{adjointtensor}, we have 
$$\varphi(L'(\Lambda_{\xi,\lambda})^{\nu-i-|b_0|,b_0})\subset L(\Lambda_{\xi}\otimes \Lambda_{\lambda}),\tilde{F}_i\varphi(L'(\Lambda_{\xi,\lambda})^{\nu-i-|b_0|,b_0})\subset L(\Lambda_{\xi}\otimes \Lambda_{\lambda}).$$ 
By Theorem \ref{framed construction of tensor}, $\varphi$ is a $\mathbf{U}$-module isomorphism, and so it commutes with the Kashiwara's operators on $\Lambda_{\xi,\lambda},\Lambda_{\xi}\otimes \Lambda_{\lambda}$. Hence we have 
$$\varphi(\tilde{F}_i(L'(\Lambda_{\xi,\lambda})^{\nu-i-|b_0|,b_0}))=\tilde{F}_i\varphi(L'(\Lambda_{\xi,\lambda})^{\nu-i-|b_0|,b_0})\subset L(\Lambda_{\xi}\otimes \Lambda_{\lambda}),$$
and so $\varphi(L(\Lambda_{\xi,\lambda})_{\nu+|\theta_{\lambda}|})\subset L(\Lambda_{\xi}\otimes \Lambda_{\lambda})$, as desired.

(b) $\bullet$ If $\mathrm{tr}\,|\nu|=0$, then $L(\Lambda_{\xi,\lambda})_{|\theta_{\lambda}|}$ is the $\mathbf{A}$-submodule of $\Lambda_{\xi,\lambda}$ generated by $\{\pi_{\xi\odot\lambda}(\theta_{\lambda})\}$. By Theorem \ref{framed construction of tensor} and definitions, we have
$$(\varphi\pi_{\xi\odot\lambda}(\theta_{\lambda}),\varphi\pi_{\xi\odot\lambda}(\theta_{\lambda}))_{\xi,\lambda}=(\eta_{\xi}\otimes \eta_{\lambda},\eta_{\xi}\otimes \eta_{\lambda})_{\xi,\lambda}=(\eta_{\xi},\eta_{\xi})_\xi(\eta_{\lambda},\eta_{\lambda})_{\lambda}=1\in \mathbf{A}.$$
$\bullet$ If $\mathrm{tr}\,|\nu|>0$, recall that $L(\Lambda_{\xi,\lambda})_{\nu+|\theta_{\lambda}|}=\sum L'(\Lambda_{\xi,\lambda})^{\nu-|b_0|,b_0}$, where the sum is taken over $b_0\in \sigma(\mathbf{B}(\lambda))$ such that $\nu-|b_0|\in \mathbb{N}[I]$. For those $b_0,b'_0$ such that $\nu-|b_0|=\nu-|b'_0|=0$, by Theorem \ref{framed construction of tensor} and \cite[Proposition 19.1.3]{Lusztig-1993}, we have 
$$(\varphi\pi_{\xi\odot\lambda}(\theta_{\lambda}b_0),\varphi\pi_{\xi\odot\lambda}(\theta_{\lambda}b'_0))_{\xi,\lambda}=(b_0^-\eta_{\xi}\otimes \eta_{\lambda},{b'_0}^-\eta_{\xi}\otimes \eta_{\lambda})_{\xi,\lambda}=(b_0^-\eta_{\xi},{b'_0}^-\eta_{\xi})_\xi(\eta_{\lambda},\eta_{\lambda})_{\lambda}\in \mathbf{A},$$ 
and so $(\varphi(L'(\Lambda_{\xi,\lambda})^{0,b_0}),\varphi(L'(\Lambda_{\xi,\lambda})^{0,b'_0}))_{\xi,\lambda}\in \mathbf{A}$. It remains to prove that $$(\varphi(L'(\Lambda_{\xi,\lambda})^{\nu-|b_0|,b_0}),\varphi(L'(\Lambda_{\xi,\lambda})^{\nu-|b'_0|,b'_0}))_{\xi,\lambda}\subset \mathbf{A}$$ 
for those $b_0,b'_0$ such that $\nu-|b_0|\not=0$ or $\nu-|b'_0|\not=0$. By the symmetry, it suffices to prove the first case. By definitions, we have $L'(\Lambda_{\xi,\lambda})^{\nu-|b_0|,b_0}=\sum \tilde{F}_i(L'(\Lambda_{\xi,\lambda})^{\nu-i-|b_0|,b_0})$, where the sum is taken over $i\in I$ such that $\nu_i-|b_0|_i>0$. For such $i$, by Theorem \ref{framed construction of tensor}, $\varphi$ is a $\mathbf{U}$-module isomorphism, and so it commutes with the Kashiwara's operators on $\Lambda_{\xi,\lambda},\Lambda_{\xi}\otimes \Lambda_{\lambda}$. By (a) and Lemma \ref{adjointtensor}, we have 
\begin{align*}
&(\varphi\tilde{F}_i(L'(\Lambda_{\xi,\lambda})^{\nu-i-|b_0|,b_0}),\varphi(L'(\Lambda_{\xi,\lambda})^{\nu-|b'_0|,b'_0}))_{\xi,\lambda}\\
=&(\tilde{F}_i\varphi(L'(\Lambda_{\xi,\lambda})^{\nu-i-|b_0|,b_0}),\varphi(L'(\Lambda_{\xi,\lambda})^{\nu-|b'_0|,b'_0}))_{\xi,\lambda}\\
=&(\varphi(L'(\Lambda_{\xi,\lambda})^{\nu-i-|b_0|,b_0}),\tilde{E}_i\varphi(L'(\Lambda_{\xi,\lambda})^{\nu-|b'_0|,b'_0}))_{\xi,\lambda}\ \mathrm{mod}\,v^{-1}\mathbf{A}\\
=&(\varphi(L'(\Lambda_{\xi,\lambda})^{\nu-i-|b_0|,b_0}),\varphi\tilde{E}_i(L'(\Lambda_{\xi,\lambda})^{\nu-|b'_0|,b'_0}))_{\xi,\lambda}\ \mathrm{mod}\,v^{-1}\mathbf{A}.
\end{align*}
Notice that $L'(\Lambda_{\xi,\lambda})^{\nu-i-|b_0|,b_0}$ and $\tilde{E}_i(L'(\Lambda_{\xi,\lambda})^{\nu-|b'_0|,b'_0})\subset L(\Lambda_{\xi,\lambda})_{\nu-i+|\theta_{\lambda}|}$, by the inductive hypothesis, we have $(\varphi(L'(\Lambda_{\xi,\lambda})^{\nu-i-|b_0|,b_0}),\varphi\tilde{E}_i(L'(\Lambda_{\xi,\lambda})^{\nu-|b'_0|,b'_0}))_{\xi,\lambda}\subset \mathbf{A}$. Hence we have 
$(\varphi\tilde{F}_i(L'(\Lambda_{\xi,\lambda})^{\nu-i-|b_0|,b_0}),\varphi(L'(\Lambda_{\xi,\lambda})^{\nu-|b'_0|,b'_0}))_{\xi,\lambda}\subset \mathbf{A}$,
and so $$(\varphi(L'(\Lambda_{\xi,\lambda})^{\nu-|b_0|,b_0}),\varphi(L'(\Lambda_{\xi,\lambda})^{\nu-|b'_0|,b'_0}))_{\xi,\lambda}\subset \mathbf{A},$$ 
as desired.
\end{proof}

\begin{proof}[Proof of Lemma {\rm{\ref{almost orthonormal}}}]
For any $b\in \tilde{\mathbf{B}}(\xi,\lambda)\subset \tilde{\mathbf{B}}(\mathbf{f}\theta_{\lambda}\mathbf{f})$, by Corollary \ref{18.1.7} and Lemma \ref{thetalambdaB}, there exist $b_0\in \sigma(\mathbf{B}(\lambda))$ and $i_1,...i_n\in I$ such that $b=\tilde{\phi}_{i_1}...\tilde{\phi}_{i_n}(\theta_{\lambda}b_0)\ \mathrm{mod}\, v^{-1}\mathcal{L}(\mathbf{\mathbf{f}\theta_{\lambda}\mathbf{f}})$.
By Lemma \ref{18.3.8}, we have $\pi_{\xi\odot\lambda}(b)=\tilde{F}_{i_1}...\tilde{F}_{i_n}\pi_{\xi\odot\lambda}(\theta_{\lambda}b_0)\ \mathrm{mod}\ v^{-1}L(\Lambda_{\xi,\lambda})$. By Theorem \ref{framed construction of tensor}, $\varphi$ is a $\mathbf{U}$-module isomorphism, and so it commutes with the Kashiwara's operators on $\Lambda_{\xi,\lambda},\Lambda_{\xi}\otimes \Lambda_{\lambda}$. By Lemma \ref{(L,L)inA}, we have 
\begin{align*}
(\varphi\pi_{\xi\odot\lambda}(b),\varphi\pi_{\xi\odot\lambda}(b))_{\xi,\lambda}=&(\varphi\tilde{F}_{i_1}...\tilde{F}_{i_n}\pi_{\xi\odot\lambda}(\theta_{\lambda}b_0),\varphi\tilde{F}_{i_1}...\tilde{F}_{i_n}\pi_{\xi\odot\lambda}(\theta_{\lambda}b_0))_{\xi,\lambda}\ \mathrm{mod}\,v^{-1}\mathbf{A}\\
=&(\tilde{F}_{i_1}...\tilde{F}_{i_n}\varphi\pi_{\xi\odot\lambda}(\theta_{\lambda}b_0),\tilde{F}_{i_1}...\tilde{F}_{i_n}\varphi\pi_{\xi\odot\lambda}(\theta_{\lambda}b_0))_{\xi,\lambda}\ \mathrm{mod}\,v^{-1}\mathbf{A}.
\end{align*}
For any $m\in L(\Lambda_{\xi,\lambda})$ and $i\in I$, by Lemma \ref{L=L'} and \ref{adjointtensor}, we have 
$$(\tilde{F}_i\varphi(m),\tilde{F}_i\varphi(m))_{\xi,\lambda}=(\varphi(m),\tilde{E}_i\tilde{F}_i\varphi(m))_{\xi,\lambda}\ \mathrm{mod}\,v^{-1}\mathbf{A}=(\varphi(m),\varphi(m))_{\xi,\lambda}\ \mathrm{mod}\,v^{-1}\mathbf{A}.$$
By Theorem \ref{framed construction of tensor} and \ref{properties of B(Lambda)}, we have
\begin{align*}
(\varphi\pi_{\xi\odot\lambda}(b),\varphi\pi_{\xi\odot\lambda}(b))_{\xi,\lambda}=&(\tilde{F}_{i_1}...\tilde{F}_{i_n}\varphi\pi_{\xi\odot\lambda}(\theta_{\lambda}b_0),\tilde{F}_{i_1}...\tilde{F}_{i_n}\varphi\pi_{\xi\odot\lambda}(\theta_{\lambda}b_0))_{\xi,\lambda}\ \mathrm{mod}\, v^{-1}\mathbf{A}\\
=&(\varphi\pi_{\xi\odot\lambda}(\theta_{\lambda}b_0),\varphi\pi_{\xi\odot\lambda}(\theta_{\lambda}b_0))_{\xi,\lambda}\ \mathrm{mod}\, v^{-1}\mathbf{A}\\
=&(b_0^-\eta_{\xi}\otimes \eta_{\lambda},b_0^-\eta_{\xi}\otimes \eta_{\lambda})_{\xi,\lambda}\ \mathrm{mod}\, v^{-1}\mathbf{A}\\
=&(b_0^-\eta_{\xi},b_0^-\eta_{\xi})_{\xi}(\eta_{\lambda},\eta_{\lambda})_{\lambda}\ \mathrm{mod}\, v^{-1}\mathbf{A}\in 1+v^{-1}\mathbf{A},
\end{align*}
as desired.
\end{proof}

\bibliography{mybibfile}

\begin{thebibliography}{10}

\bibitem{Bao-He-2022}
H.~Bao and X.~He.
\newblock Total positivity in twisted product of flag varieties.
\newblock {\em arXiv:2211.11168v2}, 2022.

\bibitem{Bao-Wang-2016}
H.~Bao and W.~Wang.
\newblock Canonical bases in tensor products revisited.
\newblock {\em Amer. J. Math.}, 138(6):1731--1738, 2016.

\bibitem{Fang-Lan-2023}
J.~Fang and Y.~Lan.
\newblock Lusztig sheaves and tensor products of integrable highest weight
  modules.
\newblock {\em arXiv:2310.18682v3}, 2023.

\bibitem{Fang-Lan-Xiao-2023}
J.~Fang, Y.~Lan, and J.~Xiao.
\newblock Lusztig sheaves and integrable highest weight modules.
\newblock {\em arXiv:2307.16131v4}, 2023.

\bibitem{Kang-Kashiwara-2012}
S.-J. Kang and M.~Kashiwara.
\newblock Categorification of highest weight modules via
  {K}hovanov-{L}auda-{R}ouquier algebras.
\newblock {\em Invent. Math.}, 190(3):699--742, 2012.

\bibitem{Kashiwara-1990}
M.~Kashiwara.
\newblock Crystalizing the {$q$}-analogue of universal enveloping algebras.
\newblock {\em Comm. Math. Phys.}, 133(2):249--260, 1990.

\bibitem{Kashiwara-1991}
M.~Kashiwara.
\newblock On crystal bases of the {$Q$}-analogue of universal enveloping
  algebras.
\newblock {\em Duke Math. J.}, 63(2):465--516, 1991.

\bibitem{Khovanov-Lauda-2009}
M.~Khovanov and A.~D. Lauda.
\newblock A diagrammatic approach to categorification of quantum groups. {I}.
\newblock {\em Represent. Theory}, 13:309--347, 2009.

\bibitem{Li-2014}
Y.~Li.
\newblock Tensor product varieties, perverse sheaves, and stability conditions.
\newblock {\em Selecta Math. (N.S.)}, 20(2):359--401, 2014.

\bibitem{Lusztig-1990}
G.~Lusztig.
\newblock Canonical bases arising from quantized enveloping algebras.
\newblock {\em J. Amer. Math. Soc.}, 3(2):447--498, 1990.

\bibitem{Lusztig-1991}
G.~Lusztig.
\newblock Quivers, perverse sheaves, and quantized enveloping algebras.
\newblock {\em J. Amer. Math. Soc.}, 4(2):365--421, 1991.

\bibitem{Lusztig-1992}
G.~Lusztig.
\newblock Canonical bases in tensor products.
\newblock {\em Proc. Nat. Acad. Sci. U.S.A.}, 89(17):8177--8179, 1992.

\bibitem{Lusztig-1993}
G.~Lusztig.
\newblock {\em Introduction to quantum groups}, volume 110 of {\em Progress in
  Mathematics}.
\newblock Birkh\"{a}user Boston, Inc., Boston, MA, 1993.

\bibitem{Nakajima-1998}
H.~Nakajima.
\newblock Quiver varieties and {K}ac-{M}oody algebras.
\newblock {\em Duke Math. J.}, 91(3):515--560, 1998.

\bibitem{Rouquier-2012}
R.~Rouquier.
\newblock Quiver {H}ecke algebras and 2-{L}ie algebras.
\newblock {\em Algebra Colloq.}, 19(2):359--410, 2012.

\bibitem{Varagnolo-Vasserot-2011}
M.~Varagnolo and E.~Vasserot.
\newblock Canonical bases and {KLR}-algebras.
\newblock {\em J. Reine Angew. Math.}, 659:67--100, 2011.

\bibitem{Webster-2015}
B.~Webster.
\newblock Canonical bases and higher representation theory.
\newblock {\em Compos. Math.}, 151(1):121--166, 2015.

\bibitem{Zheng-2014}
H.~Zheng.
\newblock Categorification of integrable representations of quantum groups.
\newblock {\em Acta Math. Sin. (Engl. Ser.)}, 30(6):899--932, 2014.

\end{thebibliography}

\end{document}